\documentclass[11pt]{amsart}
\usepackage{pst-all}
\usepackage{amsmath}
\usepackage{amsfonts}
\usepackage{amssymb}
\usepackage{multicol}
\usepackage{verbatim}
\usepackage{amsthm}
\usepackage{amscd}
\usepackage{pb-diagram}
\usepackage[mathscr]{eucal}
\usepackage{xcolor}
\usepackage[all]{xy}
\usepackage{amsmath}

\usepackage{amsfonts}

\usepackage{amssymb}

\usepackage{amsthm}

\usepackage{graphics}

\newcommand{\ZZ}{\mathbb{Z}}

\newcommand{\CC}{\mathbb{C}}

\newcommand{\NN}{\mathbb{N}}

\newcommand{\wt}{\mathrm{wt}}

\newcommand{\cl}{\mathrm{cl}}

\newcommand{\ffbox}[1]{
\setbox9=\hbox{$\scriptstyle\overline{1}$}
\framebox[20pt][c]{\rule{0mm}{\ht9}${\scriptstyle #1}$}
}

\newcommand{\QQ}{\mathbb{Q}}

\newcommand{\Glie}{\mathfrak{g}}

\newcommand{\Hlie}{\mathfrak{h}}

\newcommand{\U}{\mathcal{U}}

\newtheorem{thm}{Theorem}[section]

\newtheorem{defi}[thm]{Definition}

\newtheorem{cor}[thm]{Corollary}

\newtheorem{prop}[thm]{Proposition}

\newtheorem{lem}[thm]{Lemma}

\newtheorem{rem}[thm]{Remark}

\newtheorem{ex}[thm]{Example}

\title{Quantum extremal loop weight modules and monomial crystals}

\author[Mathieu Mansuy]{Mathieu Mansuy}
\address{Univ. Paris-Diderot-Paris 7, IMJ - PRG CNRS UMR 7586, B\^at. Sophie Germain, Case 7012, 75205 Paris Cedex 13, FRANCE}
\email{mansuy@math.jussieu.fr}

\begin{document}

\begin{abstract} In this paper we construct a new family of representations for the quantum toroidal algebra $\U_q(sl_{n+1}^{tor})$, which are $\ell$-extremal in the sense of Hernandez \cite{hernandez_quantum_2009}. We construct extremal loop weight modules associated to level $0$ fundamental weights $\varpi_\ell$ when $n= 2r+1$ is odd and $\ell = 1, r+1$ or $n$. To do it, we relate monomial realizations of level 0 extremal fundamental weight crystals with integrable representations of $\U_q(sl_{n+1}^{tor})$, and we introduce promotion operators for the level 0 extremal fundamental weight crystals. By specializing the quantum parameter, we get finite-dimensional modules of quantum toroidal algebras at roots of unity. In general, we give a conjectural process to construct extremal loop weight modules from monomial realizations of crystals.
\end{abstract}

\maketitle

\tableofcontents

\section{Introduction}

Let us consider a finite-dimensional simple Lie algebra $\Glie$ and its associated quantum affine algebra $\U_q(\hat{\Glie})$. Beck and Drinfeld \cite{beck_braid_1994, drinfeld_new_1988} proved that $\U_q(\hat{\Glie})$ has two realizations: first as the quantized enveloping algebra of the affine Lie algebra $\hat{\Glie}$ and second as the Drinfeld quantum affinization of the quantum group $\U_q(\Glie)$.

The representation theory of the quantum affine algebras has been intensively studied (see, among others, \cite{akasaka_finite-dimensional_1997, beck_crystal_2004, chari_quantum_1991, chari_quantum_1995, frenkel_combinatorics_2001, frenkel_$q$-characters_1999, lusztig_introduction_1993, nakajima_quiver_2001}). Kashiwara \cite{kashiwara_crystal_1994} has defined a class of integrable representations $V(\lambda)$ of these algebras, called extremal weight modules, parametrized by an integrable weight $\lambda$ and with crystal basis $\mathcal{B}(\lambda)$. When $\lambda$ is dominant, $V(\lambda)$ is the simple integrable module of highest weight $\lambda$. But in general $V(\lambda)$ is not simple and it is neither of highest weight nor of lowest weight. These representations were the subject of numerous papers (see \cite{beck_crystal_2002, beck_crystal_2004, hernandez_level_2006, kashiwara_crystal_1994, kashiwara_level-zero_2002, naito_path_2003, naito_path_2006, nakajima_extremal_2004}) and are particularly important because they have finite-dimensional quotients for some special weight $\lambda$. Kashiwara has proved in this way the existence of crystal bases for the finite-dimensional fundamental representations of $\U_q(\hat{\Glie})$ (for a special choice of the spectral parameter).

The quantum affine algebra $\U_q(\hat{\Glie})$ is also a quantum Kac-Moody algebra and thus can be affinized again by the Drinfeld quantum affinization process. One gets a toroidal (or double affine) quantum algebra $\U_q(\Glie^{tor})$ which is not a quantum Kac-Moody algebra anymore and can not be affinized again by this process (it can be viewed as ``the terminal object'' in this construction). These algebras were first introduced by Ginzburg-Kapranov-Vasserot in type A \cite{ginzburg_langlands_1995} and then in the general context \cite{jing_quantum_1998, nakajima_quiver_2001}. In type A, they are in Schur-Weyl duality with elliptic Cherednik algebras \cite{varagnolo_schur_1996}.

The representation theory of these algebras has been intensively studied (see for example \cite{feigin_quantum_2011, feigin_quantum_2011-1, feigin_quantum_2012,  feigin_representations_2013, hernandez_representations_2005, hernandez_quantum_2009, hernandez_algebra_2011, miki_representations_2000, varagnolo_double-loop_1998} and references therein). In the spirit of works of Kashiwara, Hernandez \cite{hernandez_quantum_2009} proposed the definition of extremal loop weight modules for $\U_q(\Glie^{tor})$. The main motivation is to construct finite-dimensional representations of the quantum toroidal algebra at roots of unity. He constructs the first example of such a module for $\U_q(sl_{4}^{tor})$ which is neither of $\ell$-highest weight nor of $\ell$-lowest weight. This module is generated by an $\ell$-weight vector of $\ell$-weight an analogue of the level $0$ fundamental weight $\varpi_1 = \Lambda_1-\Lambda_0$. By specializing the quantum parameter $q$ at roots of unity, he obtains finite-dimensional representations of the quantum toroidal algebra at roots of unity.

In the present paper, we construct a new family of extremal loop weight modules for the quantum toroidal algebra $\U_q(sl_{n+1}^{tor})$\footnote{After this paper appeared on the arXiv, the constructions in \cite{feigin_representations_2013} were
brought to our attention by H. Nakajima. Some of the representations constructed in this
paper (the $V(Y_{1,0}Y_{0,1}^{-1})$) are also defined in \cite{feigin_representations_2013} from another point of view and are called \textit{vector representations} there.}: we define extremal loop weight modules associated to the level $0$ fundamental weight $\varpi_\ell = \Lambda_\ell - \Lambda_0$ when $n=2r+1$ is odd and $\ell = 1, r+1$ or $n$ (Theorem \ref{thmmod}). We call them the extremal fundamental loop weight modules. This construction is based on the monomial realizations of level 0 extremal fundamental weight crystals $\mathcal{B}(\varpi_\ell)$. We relate these monomial crystals with integrable representations of $\U_q(sl_{n+1}^{tor})$ by studying their combinatorics: we introduce promotion operators for $\mathcal{B}(\varpi_\ell)$ ($1 \leq \ell \leq n$). We describe them in terms of monomials. These operators play an important role in our work: on the one hand, at the level of crystals, they are used to check that these monomial crystals are closed when $\ell = 1, r+1$ or $n$ (see Definition \ref{defclocry} for this notion, related to the theory of $q$--characters). On the other hand, at the level of representations, they enable us to define the action of the quantum toroidal algebra. We show that the representations we constructed are irreducible and, as modules over the horizontal quantum affine subalgebra, they are isomorphic to the fundamental extremal weight modules $V(\varpi_\ell)$. We give explicit formulas for the action from the associated monomial crystal. By specializing the quantum parameter $q$ at roots of unity, we get new irreducible finite-dimensional representations of the quantum toroidal algebra at roots of unity. When $\ell$ is not equal to $1, r+1$ or $n$, the corresponding monomial crystals are not closed and it is not possible to make the same construction. We give a conjectural process to define other extremal loop weight modules in this situation: as an example, we construct an extremal loop weight module of $\U_q(sl_{4}^{tor})$ associated to the weight $2 \varpi_1$.\\

Let us describe the methods used in this paper in more detail. In \cite{kashiwara_realizations_2003, nakajima_$t$-analogs_2003}, Kashiwara and Nakajima have defined a crystal $\mathcal{M}$, called the monomial crystal, whose vertices are Laurent monomials. They determined monomial realizations of crystals of finite type. These results have been extended in \cite{hernandez_level_2006} to the level 0 extremal weight $\U_q(\hat{sl}_{n+1})$-crystals $\mathcal{B}(\varpi_\ell)$ ($1 \leq \ell \leq n$): if $n = 2r+1$ is odd, it is isomorphic to a sub-$\U_q(\hat{sl}_{n+1})$-crystal $\mathcal{M}_\ell$ of $\mathcal{M}$.

The monomials occurring in these realizations of crystals can be interpreted at the level of representation theory. In fact Frenkel and Reshetikhin \cite{frenkel_$q$-characters_1999} defined a correspondence between $\ell$-weights (eigenvalues of the Cartan subalgebra for the Drinfeld realization) and these monomials. Motivated by these facts, Hernandez \cite{hernandez_quantum_2009} used the monomial $\U_q(\hat{sl}_{4})$-crystal $\mathcal{M}_1$ to construct an integrable representation of $\U_q(sl_{4}^{tor})$ whose $\ell$-weights are the monomials occurring in this crystal. He defined in this way the first example of extremal loop weight modules for $\U_q(sl_{4}^{tor})$. We use the same technical feature in this paper. We propose to relate the monomial $\U_q(\hat{sl}_{n+1})$-crystals $\mathcal{M}_\ell$ (where $n = 2r+1$ is supposed to be odd) with integrable representations of $\U_q(sl_{n+1}^{tor})$.

Let us outline the main steps of the construction of extremal fundamental loop weight modules associated to $\mathcal{M}_\ell$. It is based on the combinatorial study of these crystals. The cyclic symmetry of the Dynkin diagram of type $A_n^{(1)}$ has a counterpart at the level of crystals. Actually, these symmetry properties are already known for the $\U_q(sl_{n+1})$-crystals of finite type, and translated into the existence of promotion operators (see \cite{bandlow_uniqueness_2010, fourier_kirillov-reshetikhin_2009, okado_existence_2008, schilling_combinatorial_2008, shimozono_affine_2002} and references therein). Here we introduce promotion operators for the level 0 extremal fundamental weight crystals $\mathcal{B}(\varpi_\ell)$ ($1 \leq \ell \leq n$). We improve these operators in the monomial realizations $\mathcal{M}_\ell$ of \cite{hernandez_level_2006}. In particular, we get a new description of these monomial crystals.

A monomial set is not in general the set of $\ell$-weights of an integrable representation. In fact, it must satisfy combinatorial properties related to the theory of $q$-characters (see \cite{frenkel_combinatorics_2001, frenkel_$q$-characters_1999}). This leads us to introduce the notion of closed monomial set (Definition \ref{defclocry}). It gives a necessary condition for a set to be the set of $\ell$-weights of an integrable representation. Finally, we determine when the monomial crystal $\mathcal{M}_\ell$ is closed, using promotion operators: this is the case if and only if $\ell = 1, r+1$ or $n$ (Theorem \ref{proclocry}).

When $\mathcal{M}_\ell$ is closed, we construct an associated integrable $\U_q(sl_{n+1}^{tor})$-module whose the set of $\ell$-weights consists of monomials occurring in $\mathcal{M}_\ell$. For that, we paste together some finite-dimensional representations of the various vertical quantum affine subalgebas of $\U_q(sl_{n+1}^{tor})$. The existence of promotion operators for $\mathcal{M}_\ell$ involves that it defines a $\U_q(sl_{n+1}^{tor})$-module structure. Furthermore we check that the representations obtained in this way do satisfy the definition of extremal loop weight modules. They are irreducible, isomorphic to the level 0 fundamental extremal representations $V(\varpi_\ell)$ as modules over the horizontal quantum affine subalgebra. Moreover the action of the quantum toroidal algebra on them is explicitly known, given from the associated crystal. Finally by specializing the quantum parameter $q$ at roots of unity, we get finite-dimensional representations of the quantum toroidal algebra at roots of unity.

When the monomial crystal $\mathcal{M}_\ell$ is not closed, there is no integrable representation of $\U_q(sl_{n+1}^{tor})$ whose set of $\ell$-weights consists of monomials occurring in it. The idea is to consider instead of $\mathcal{M}_\ell$ a closed crystal containing it and to apply the preceding methods to this crystal. We treat an example of such a construction: we define a representation of $\U_q(sl_{4}^{tor})$ which satisfies the definition of extremal loop weight modules.

Let us now describe briefly the organization of this paper.

In Section 2 we recall the definitions of quantum affine algebras $\U_q(\hat{sl}_{n+1})$ and quantum toroidal algebras $\U_q(sl_{n+1}^{tor})$ and we briefly review their representation theory. In particular one defines the extremal weight modules and the extremal loop weight modules. Section 3 is devoted to the study of monomial crystals. We recall its definition and we introduce the notion of closed monomial set (Definition \ref{defclocry}). We introduce promotion operators for the level 0 fundamental extremal weight crystals. As a consequence, we determine when $\mathcal{M}_\ell$ is closed (Theorem \ref{proclocry}). In Section 4 we construct a new family of representations of $\U_q(sl_{n+1}^{tor})$ (the extremal fundamental loop weight modules) when $n$ is odd and $\mathcal{M}_\ell$ is closed (Theorem \ref{thmmod}). We check that these representations satisfy the definition of extremal loop weight modules (Theorem \ref{thmelm}) and we give formulas for the action (Theorem \ref{thmactmod}). We get finite-dimensional representations of the quantum toroidal algebra at roots of unity by specializing the quantum parameter $q$ at roots of unity (Theorem \ref{thmmodunit}). In Section 5 we treat an example where the considered monomial crystal is not closed. We construct a representation of $\U_q(sl_{4}^{tor})$ associated to the level 0 weight $2 \varpi_1$. In Section 6 other possible developments and applications of these results are discussed.\\

\textbf{Acknowledgements:} I am grateful to my advisor David Hernandez for suggesting me this problem, for his encouragement and precious comments. I would like to thank Anne Schilling and Cedric Lecouvey for their comments on the promotion operators and for pointing me out some references. I want to thank Hiraku Nakajima for interesting questions about a first version of this work. I thank Alexandre Bouayad for accurate discussions and numerous observations, Xin Fang and Dragos Fratila for useful discussions. Finally the referee deserves thanks for careful reading and many precious comments.

\section{Background}

We recall the main definitions and general properties about the representation theory of quantum affine algebras and quantum toroidal algebras of type $A$.

\subsection{Cartan matrix} \nocite{kac_infinite-dimensional_1990} Let $C=(C_{i,j})_{0\leq i,j\leq n}$ be a
Cartan matrix of type $A_n^{(1)}$ ($n \geq 2$),
$$C = \begin{pmatrix}
2 & -1 & 0 & \cdots & 0 & -1 \\ 
-1 & 2 & \ddots & \ddots & & 0 \\ 
0 & \ddots & \ddots & \ddots & \ddots & \vdots \\ 
\vdots & \ddots & \ddots & \ddots & \ddots & 0 \\ 
0 & & \ddots & \ddots & 2 & -1 \\ 
-1 & 0 & \cdots & 0 & -1 & 2
\end{pmatrix}. $$

\begin{rem}
The case $n=1$ is not studied in the article and is particular. In this case, the Cartan matrix is
$$\begin{pmatrix}
2 & -2 \\ 
-2 & 2
\end{pmatrix}$$
and involves $-2$. Furthermore, the quantum toroidal algebra $\U_q(sl_{2}^{tor})$ requires a special definition with different possible choices of the quantized Cartan matrix (see \cite[Remark 4.1]{hernandez_algebra_2011}).
\end{rem}

\noindent Set $I=\{0,\dots,n\}$ and $I_0 = \{1, \dots, n\}$. In particular, $(C_{i,j})_{i,j \in I_0}$ is the Cartan matrix of finite type $A_n$. In the following, $I$ will be often identified with the set $\ZZ / (n+1) \ZZ$. Consider the vector space of dimension $n+2$
 $$\Hlie = \QQ h_0 \oplus \QQ h_1 \oplus \dots \oplus \QQ h_n \oplus \QQ d$$
\noindent and the linear functions $\alpha_i$ (the simple roots), $\Lambda_i$ (the fundamental weights) on $\Hlie$ given by ($i,j \in I$)
\begin{eqnarray*}
\alpha_i (h_j)= C_{j,i}, & \alpha_i(d)= \delta_{0,i},\\
\Lambda_i(h_j) = \delta_{i,j}, & \Lambda_i(d)= 0.
\end{eqnarray*}

Denote by $\Pi=\{\alpha_0,\dots,\alpha_n\}\subset \Hlie^*$ the set of
simple roots and $\Pi^{\vee}=\{h_0,\dots, h_n\}\subset \Hlie$ the set of simple coroots. Let $P =\{\lambda \in\Hlie^* \mid \text{$\lambda(h_i)\in\ZZ$ for any $i\in I$}\}$ be the weight lattice and $P^+=\{\lambda \in P \mid \text{$\lambda(h_i)\geq 0$ for any $i\in I$}\}$ the semigroup of dominant weights. Let $Q=\bigoplus_{i\in I} \ZZ \alpha_i\subset P$ (the root lattice) and $Q^+=\sum_{i\in I}\NN \alpha_i\subset Q$. For $\lambda,\mu\in \Hlie^*$, write $\lambda \geq \mu$ if $\lambda-\mu\in Q^+$.

Set $\Hlie_0 = \QQ h_1 \oplus \dots \oplus \QQ h_n$ and $\Pi_0 = \{\alpha_1,\dots,\alpha_n\}$, $\Pi_0^{\vee}=\{h_1,\dots, h_n\}$. $(\Hlie_0, \Pi_0, \Pi_0^\vee)$ is a realization of $(C_{i,j})_{i,j \in I_0}$ (see \cite{kac_infinite-dimensional_1990}). We define as above the associated weight lattice $P_0$, its subset $P_0^+$ of dominant weights, and the root lattice $Q_0$.

Denote by $W$ the affine Weyl group: it is the subgroup of $GL(\Hlie^*)$ generated by the simple reflections $s_i\in GL(\Hlie^*)$ defined by $s_i(\lambda)=\lambda-\lambda(h_i)\alpha_i$ ($i\in I$). The Weyl group of finite type $W_0$ is the subgroup of $W$ generated by the $s_i$ with $i \in I_0$.

Let $c= h_0 + \dots + h_n$ and $\delta = \alpha_0 + \dots + \alpha_n$. We have
$$\{\omega \in P \mid \omega(h_i)=0 \text{ for all } i \in I \}=\QQ \delta.$$
Put $P_{\cl}=P/\QQ \delta$ and denote by $\cl : P \rightarrow P_{\cl}$ the canonical projection. Denote by $P^0=\{\lambda\in P\mid \lambda(c)=0\}$ the set of level $0$ weights.

\subsection{Quantum affine algebra $\U_q(\hat{sl}_{n+1})$} In this article $q = e^{t} \in\CC^*$ ($t \in \CC$) is not a root of unity and is fixed. For $l\in\ZZ, r\geq 0, m\geq m'\geq 0$ we set
$$[l]_q=\frac{q^l-q^{-l}}{q-q^{-1}}\in\ZZ[q^{\pm 1}],\  [r]_q!=[r]_q[r-1]_q\dots[1]_q,\ \begin{bmatrix}m\\m'\end{bmatrix}_q=\frac{[m]_q!}{[m-m']_q![m']_q!}.$$

\begin{defi} The quantum affine algebra $\U_q(\hat{sl}_{n+1})$ is the $\CC$-algebra with generators $k_h$ $(h\in \Hlie)$, $x_i^{\pm}$ $(i\in I)$ and relations
\begin{equation*}k_hk_{h'}=k_{h+h'}\text{ , }k_0=1,\end{equation*}
\begin{equation*}k_hx_j^{\pm}k_{-h}=q^{\pm \alpha_j(h)}x_j^{\pm},\end{equation*}
\begin{equation*}[x_i^+,x_j^-]=\delta_{i,j}\frac{k_i-k_{i}^{-1}}{q-q^{-1}},\end{equation*}
\begin{equation*}
(x_i^{\pm})^{(2)}x_{i+1}^{\pm} - x_i^{\pm}x_{i+1}^{\pm}x_i^{\pm} + x_{i+1}^{\pm}(x_i^{\pm})^{(2)} = 0.\end{equation*}
\end{defi}

\noindent Here we use the notation $k_i^{\pm 1} = k_{\pm h_i}$ and for all $r \geq 0$ we set $(x_i^\pm)^{(r)} = \frac{(x_i^\pm)^r}{[r]_q!}$. One defines a coproduct on $\U_q(\hat{sl}_{n+1})$ by setting
$$\Delta(k_h)=k_h\otimes k_h,$$
$$\Delta(x_i^+)=x_i^+\otimes 1 + k_i^+\otimes x_i^+\text{ , }\Delta(x_i^-)=x_i^-\otimes k_i^- + 1\otimes x_i^-.$$

Let $\U_q(\hat{sl}_{n+1})'$ be the subalgebra of $\U_q(\hat{sl}_{n+1})$ generated by $x_{i}^{\pm}$ and $k_h$ ($h\in \sum \QQ h_{i}$). This has $P_\cl$ as a weight lattice.

For $J \subset I$ denote by $\U_q(\hat{sl}_{n+1})_{J}$ the subalgebra of $\U_q(\hat{sl}_{n+1})$ generated by the $x_i^{\pm}, k_{p h_i}$ for $i\in J, p \in \QQ$. If $J=I_0$, $\U_q(\hat{sl}_{n+1})_{I_0}$ is the quantum group of finite type associated to the data $(\Hlie_0, \Pi_0, \Pi_0^\vee)$, also denoted by $\U_q(sl_{n+1})$. In particular, a $\U_q(\hat{sl}_{n+1})$-module has a structure of $\U_q(sl_{n+1})$-module. If $J= \{i\}$ with $i \in I$, $\U_q(\hat{sl}_{n+1})_{J}$ is isomorphic to $\U_{q}(sl_2)$ and denoted by $\U_i$. So a $\U_q(\hat{sl}_{n+1})$-module has also a structure of $\U_{q}(sl_2)$-module.

Let $\U_q(\hat{sl}_{n+1})^+$ (resp. $\U_q(\hat{sl}_{n+1})^-$, $\U_q(\Hlie)$) be the subalgebra of $\U_q(\hat{sl}_{n+1})$ generated by the $x_i^+$ (resp. the $x_i^-$, the $k_h$). We have a triangular decomposition of $\U_q(\hat{sl}_{n+1})$ (see \cite{lusztig_introduction_1993}): 

\begin{thm} We have an isomorphism of vector spaces
$$\U_q(\hat{sl}_{n+1}) \simeq \U_q(\hat{sl}_{n+1})^- \otimes \U_q(\Hlie) \otimes \U_q(\hat{sl}_{n+1})^+.$$
\end{thm}

\subsection{Representations of $\U_q(\hat{sl}_{n+1})$} For $V$ a representation of $\U_q(\hat{sl}_{n+1})$ and $\nu\in P$, the weight space $V_{\nu}$ of $V$ is
$$V_\nu = \{v\in V|k_h \cdot v = q^{\nu(h)}v,\forall h\in \Hlie\}.$$
Set $\wt(V) = \{\nu \in P | V_\nu \neq \{0\} \}.$

For $\lambda\in P$, a representation $V$ is said to be of highest weight $\lambda$ if there is $v\in V_\lambda$ such that for all $i\in I, x_i^+ \cdot v = 0$ and $\U_q(\hat{sl}_{n+1}) \cdot v = V$. Furthermore there is a unique simple highest weight module of highest weight $\lambda$.

\begin{defi}\label{defint} A representation  $V$ is said to be integrable if
\begin{enumerate}
\item[(i)] it admits a weight space decomposition $V = \bigoplus_{\nu\in P} V_\nu$,
\item[(ii)] all the $x_i^\pm$ ($i \in I$) are locally nilpotent.
\end{enumerate}
\end{defi}

\begin{rem}\label{remintmod} This definition differs from the one given in \cite{hernandez_quantum_2009}. In fact it is required in addition in \cite{hernandez_quantum_2009} that the representation $V$ satisfies
\begin{enumerate}
\item[(iii)] $V_\nu$ is finite-dimensional for any $\nu \in P$,
\item[(iv)] $V_{\nu \pm N\alpha_i} = \{0\}$ for all $\nu\in P$, $N>>0$, $i\in I$.
\end{enumerate}
These conditions are implied by the previous ones for the highest weight modules.
\end{rem}

\begin{thm}\cite{lusztig_introduction_1993}
The simple highest weight module of highest weight $\lambda$ is integrable if and only if $\lambda$ is dominant. We denote it $V(\lambda)$ ($\lambda \in P^{+}$).
\end{thm}

For an integrable representation $V$ of $\U_q(\hat{sl}_{n+1})$ with finite-dimensional weight spaces, one defines the usual character
$$\chi(V) = \sum_{\nu \in P} \dim(V_{\nu}) e(\nu) \in \prod_{\nu \in P} \ZZ e(\nu).$$

Similar definitions hold for the quantum group $\U_q(sl_{n+1})$. In this case, the integrable simple highest weight modules are parametrized by $P_0^+$ and denoted by $V_0(\lambda)$ ($\lambda \in P_0^+$). Further they are finite-dimensional (see \cite{lusztig_introduction_1993, rosso_representations_1991}). Let $\mathcal{C}$ be the category of integrable finite-dimensional representations of $\U_q(sl_{n+1})$ and $\mathcal{R}$ its Grothendieck ring.

\begin{thm}\cite{lusztig_introduction_1993, rosso_representations_1991} 
The category $\mathcal{C}$ is a semi-simple tensor category and the simple objects of $\mathcal{C}$ are the $(V_0(\lambda))_{\lambda\in P_0^+}$. Furthermore $\chi$ induces a ring morphism
$$\chi : \mathcal{R} \rightarrow \bigoplus_{\nu \in P_0} \ZZ e(\nu) $$
where the product at the right hand side is defined by $e(\mu) e(\nu) = e(\mu + \nu)$.
\end{thm}

We do not recall here the theory of crystal bases of quantum groups, we just refer to \cite{kashiwara_crystal_1994, kashiwara_bases_2002, kashiwara_level-zero_2002}. Let us remind only that for $\lambda \in P^+$, the $\U_q(\hat{sl}_{n+1})$-module $V(\lambda)$ has a crystal basis $\mathcal{B}(\lambda)$. In the same way we denote by $\mathcal{B}_0(\lambda)$ the crystal basis of the $\U_q(sl_{n+1})$-module $V_0(\lambda)$ ($\lambda \in P_0^+$). When we want to distinguish crystals of $\U_q(\hat{sl}_{n+1})$, $\U_q(\hat{sl}_{n+1})_{J}$ with $J \subset I$ and $\U_q(\hat{sl}_{n+1})'$, we call it respectively a $P$-crystal or $I$-crystal, a $J$-crystal and a $P_{\cl}$-crystal.

\subsection{Extremal weight modules} In this section we recall the definition and some properties of extremal weight modules for the quantum affine algebra $\U_q(\hat{sl}_{n+1})$ given by Kashiwara \cite{kashiwara_crystal_1994, kashiwara_level-zero_2002}. All of these hold for general quantum Kac-Moody algebras and in particular for $\U_q(sl_{n+1})$.

\begin{defi} For an integrable $\U_q(\hat{sl}_{n+1})$-module $V$ and $\lambda\in P$, a vector $v\in V_{\lambda}$ is called extremal of weight $\lambda$ if there are vectors $\{v_w\}_{w\in W}$ such that $v_{\text{Id}}=v$ and
$$x_i^{\pm} \cdot v_w=0 \text{ and }(x_i^{\mp})^{(\pm w(\lambda)(h_i))} \cdot v_w=v_{s_i(w)} \text{ if } \pm w(\lambda)(h_i)\geq 0.$$ \end{defi}

\noindent Note that if the vector $v$ is extremal of weight $\lambda$, then for $w\in W$, $v_{w}$ is extremal of weight $w(\lambda)$.

\begin{rem} The definition of extremal vector can be rewritten as follows (see \cite{kashiwara_crystal_1994}): for an integrable $\U_q(\hat{sl}_{n+1})$-module $V$, a weight vector $v$ of weight $\lambda$ is called $i$-extremal if $x_i^{+} \cdot v = 0$ or $x_i^{-} \cdot v = 0$. In this case we set $S_i(v) = (x_i^-)^{(\lambda(h_i))} \cdot v$ or $S_i(v) = (x_i^+)^{(-\lambda(h_i))} \cdot v$ respectively. Then a weight vector $v$ is extremal if, for any $l \geq 0$, $S_{i_1} \circ \cdots \circ S_{i_l} (v)$ is $i$-extremal for any $i, i_1, \cdots, i_l \in I$. We set in that case $$W \cdot v = \{S_{i_1} \circ \cdots \circ S_{i_l} (v) v \vert \ l \in \NN, \ i_1, \cdots, i_l \in I \}.$$
\end{rem}

The notion of extremal elements in a crystal $\mathcal{B}$ can be defined in the same way.

\begin{defi} For $\lambda\in P$, the extremal weight module $V(\lambda)$ of extremal weight $\lambda$ is the $\U_q(\hat{sl}_{n+1})$-module generated by a vector $v_{\lambda}$ with the defining relations that $v_{\lambda}$ is extremal of weight $\lambda$.\end{defi}

\begin{ex} If $\lambda$ is dominant, $V(\lambda)$ is the simple highest weight module of highest weight $\lambda$. \end{ex}

\begin{thm} \cite{kashiwara_crystal_1994} For $\lambda\in P$, the module $V(\lambda)$ is integrable and has a crystal basis $\mathcal{B}(\lambda)$.\end{thm}

Set $ \lambda = \varpi_{\ell} $, where $1 \leq \ell \leq n$ and $\varpi_{\ell}$ is the level $0$ fundamental weight $ \varpi_{\ell}=\Lambda_{\ell}- \Lambda_{0}$.

\begin{thm}\cite{kashiwara_level-zero_2002}\label{thmkasem} Let $1 \leq \ell \leq n$.
\begin{enumerate}
\item[(i)] $V(\varpi_\ell)$ is an irreducible $\U_q(\hat{sl}_{n+1})$-module.
\item[(ii)] Any non-zero integrable $\U_q(\hat{sl}_{n+1})$-module generated by an extremal weight vector of weight $\varpi_\ell$ is isomorphic to $V(\varpi_\ell)$.
\end{enumerate}
\end{thm}

Let $w$ be an element of $W$ such that $w(\varpi_\ell) = \varpi_\ell + \delta$. Such an element exists and is not unique (see \cite{kashiwara_level-zero_2002}). It defines a $\U_q(\hat{sl}_{n+1})'$-automorphism (also called $P_{\cl}$-automorphism in the following) of the restricted $\U_q(\hat{sl}_{n+1})'$-module $V(\varpi_\ell)$, which sends $v$ to $v_{w}$. It is of weight $\delta$, and denoted by $z_\ell$. Let us define the $\U_q(\hat{sl}_{n+1})'$-module
$$W(\varpi_\ell) = V(\varpi_\ell)/(z_\ell - 1) V(\varpi_\ell).$$
Then we have

\begin{thm} \cite{kashiwara_level-zero_2002} Let $1 \leq \ell \leq n$.
\begin{enumerate}
\item[(i)] $W(\varpi_{\ell})$ is a finite-dimensional irreducible $\U_q(\hat{sl}_{n+1})'$-module.
\item[(ii)] For any $\mu \in \wt(V(\varpi_{\ell}))$,
$$W(\varpi_{\ell})_{\cl(\mu)}\simeq V(\varpi_{\ell})_\mu.$$
\item[(iii)] $V(\varpi_{\ell})$ is isomorphic to $W(\varpi_{\ell})_\mathrm{aff}$ as a $\U_q(\hat{sl}_{n+1})$-module.
\end{enumerate}
\end{thm}

\noindent Here $M_{\mathrm{aff}}$ is the affinization of an integrable $\U_q(\hat{sl}_{n+1})'$-module $M$: this is the $\U_q(\hat{sl}_{n+1})$-module with a weight space decomposition $M_{\mathrm{aff}} = \bigoplus_{\nu \in P} (M_{\mathrm{aff}})_\nu$ defined by
$$(M_{\mathrm{aff}})_\nu = M_{\cl(\nu)}$$
and with the obvious action of $x_i^{\pm}$. Note also that we have an isomorphism of $\U_q(\hat{sl}_{n+1})'$-modules
$$M_{\mathrm{aff}} \simeq \CC[z, z^{-1}] \otimes M$$
where $x_i^{\pm}$ act on the right hand side by $z^{\pm \delta_{i,0}}x_{i}^{\pm}$. In the same way one defines the affinization $\mathcal{B}_{\mathrm{aff}}$ of a $P_{\cl}$-crystal $\mathcal{B}$. For an integrable $\U_q(\hat{sl}_{n+1})'$-module $M$ with associated $P_{\cl}$-crystal $\mathcal{B}$, its affinization $M_{\mathrm{aff}}$ has a $P$-crystal $\mathcal{B}_{\mathrm{aff}}$.

\subsection{Quantum toroidal algebra $\U_q(sl_{n+1}^{tor})$} In this section, we recall the definition and the main properties of the quantum toroidal algebra $\U_q(sl_{n+1}^{tor})$ (without central charge).

\begin{defi}\label{defqta} \cite{ginzburg_langlands_1995} The quantum toroidal algebra $\U_q(sl_{n+1}^{tor})$ is the $\CC$-algebra with generators $x_{i,r}^{\pm}$ $(i\in I, r\in\ZZ)$, $k_h$ $(h\in \Hlie)$, $h_{i,m}$ $(i\in I, m\in\ZZ-\{0\})$ and the following relations $(i, j \in I, r, r', r_1, r_2 \in\ZZ, m \in \ZZ-\{0\})$
\begin{equation*}k_hk_{h'}=k_{h+h'}\text{ , }k_0=1\text{ , }[k_{h},h_{j,m}]=0\text{ , }[h_{i,m},h_{j,m'}]=0,\end{equation*}
\begin{equation*}k_{h}x_{j,r}^{\pm}k_{-h}=q^{\pm \alpha_j(h)}x_{j,r}^{\pm},\end{equation*}
\begin{equation}\label{actcartld}[h_{i,m},x_{j,r}^{\pm}]=\pm \frac{1}{m}[mC_{i,j}]_qx_{j,m+r}^{\pm},\end{equation}
\begin{equation*}[x_{i,r}^+,x_{j,r'}^-]= \delta_{ij}\frac{\phi^+_{i,r+r'}-\phi^-_{i,r+r'}}{q-q^{-1}},\end{equation*}
\begin{equation*}x_{i,r+1}^{\pm}x_{j,r'}^{\pm}-q^{\pm C_{ij}}x_{j,r'}^{\pm}x_{i,r+1}^{\pm}=q^{\pm C_{ij}}x_{i,r}^{\pm}x_{j,r'+1}^{\pm}-x_{j,r'+1}^{\pm}x_{i,r}^{\pm},\end{equation*}
\begin{eqnarray*}
\begin{array}{c}
x_{i,r_1}^{\pm}x_{i,r_2}^{\pm}x_{i \pm 1,r'}^{\pm} - (q+q^{-1}) x_{i,r_1}^{\pm}x_{i \pm 1,r'}^{\pm}x_{i,r_2}^{\pm} + x_{i \pm 1,r'}^{\pm}x_{i,r_1}^{\pm}x_{i,r_2}^{\pm} = \\
- x_{i,r_2}^{\pm}x_{i,r_1}^{\pm}x_{i \pm 1,r'}^{\pm} + (q+q^{-1}) x_{i,r_2}^{\pm}x_{i \pm 1,r'}^{\pm}x_{i,r_1}^{\pm} - x_{i \pm 1,r'}^{\pm}x_{i,r_2}^{\pm}x_{i,r_1}^{\pm},
\end{array}
\end{eqnarray*}
and $[x_{i,r_1}^{\pm}, x_{j,r_2}^{\pm}] = 0$ if $i \neq j, j \pm 1$. Here for all $i\in I$ and $m\in\ZZ$, $\phi_{i,m}^{\pm}\in \U_q(sl_{n+1}^{tor})$ is determined by the formal power series in $\U_q(sl_{n+1}^{tor})[[z]]$ $($resp. in $\U_q(sl_{n+1}^{tor})[[z^{-1}]])$
$$\phi_i^\pm(z) = \underset{m\geq 0}{\sum}\phi_{i,\pm m}^{\pm}z^{\pm m}=k_{i}^{\pm 1}\exp \left(\pm(q-q^{-1})\underset{m'\geq 1}{\sum}h_{i,\pm m'}z^{\pm m'}\right)$$
and $\phi_{i,m}^+=0$ for $m<0$, $\phi_{i,m}^-=0$ for $m>0$.
\end{defi}

There is an algebra morphism $\U_q(\hat{sl}_{n+1})\rightarrow \U_q(sl_{n+1}^{tor})$ defined by $k_h\mapsto k_h$ , $x_i^{\pm}\mapsto x_{i,0}^{\pm}$ ($h\in \Hlie, i\in I$). Its image is called the horizontal quantum affine subalgebra of $\U_q(sl_{n+1}^{tor})$ and is denoted by $\U_q^{h}(sl_{n+1}^{tor})$. In particular, a $\U_q(sl_{n+1}^{tor})$-module $V$ has also a structure of $\U_q(\hat{sl}_{n+1})$-module. We denote by $\mathrm{Res}(V)$ the restricted $\U_q(\hat{sl}_{n+1})$-module obtained from $V$.

As it is said above, the quantum affine algebra $\U_q(\hat{sl}_{n+1})'$ has another realization in terms of Drinfeld generators \cite{beck_braid_1994, drinfeld_new_1988}: this is the $\CC$-algebra with generators $x_{i,r}^{\pm}$ ($i \in I_0,$ $r \in \ZZ$), $k_h$ ($h \in \Hlie_0$), $h_{i,m}$ ($i \in I_0, m \in \ZZ-\{0\}$) and the same relations as in Definition \ref{defqta}. It is isomorphic to the subalgebra $ \U_{q}^{v}(sl_{n+1}^{tor}) $ of $\U_q(sl_{n+1}^{tor})$ generated by the $x_{i,r}^{\pm}$, $k_{h}$, $h_{i,m}$ ($i\in I_{0}, r\in\ZZ, h \in \Hlie_0, m\in\ZZ-\{0\}$).  $ \U_{q}^{v}(sl_{n+1}^{tor}) $ is called the vertical quantum affine subalgebra of $\U_q(sl_{n+1}^{tor})$.

For all $j \in I$, set $I_{j} = I - \{j\}$ and define the subalgebra $ \U_{q}^{v,j}(sl_{n+1}^{tor}) $ of $ \U_{q}(sl_{n+1}^{tor}) $ generated by the $x_{i,r}^{\pm}$, $k_{h}$, $h_{i,m}$ ($i\in I_{j}, r\in\ZZ, h \in \bigoplus_{i \in I_j} \QQ h_i, m\in\ZZ-\{0\}$). In particular $ \U_{q}^{v,0}(sl_{n+1}^{tor}) $ is the vertical quantum affine subalgebra $ \U_{q}^{v}(sl_{n+1}^{tor}) $ of $\U_q(sl_{n+1}^{tor})$. All the $\U_{q}^{v,j}(sl_{n+1}^{tor})$ for various $j \in I$ are isomorphic: in fact let $ \theta $ be the automorphism of the Dynkin diagram of type $ A_{n}^{(1)} $ corresponding to the rotation such that $ \theta(k) = k + 1 $, where $ I $ is identified to the set $ \ZZ / (n+1) \ZZ $. It defines an automorphism $\theta_\Hlie$ of $\Hlie$ by sending $h_i, d$ to $h_{\theta(i)}, d$ ($i \in I$). For all $j \in J$, let $\theta^{(j)}$ be the automorphism of $\U_q(sl_{n+1}^{tor})$ which sends $ x_{i, r}^{\pm} $, $ k_{h} $, $ h_{i,m} $ to $ x_{\theta^{j}(i), r}^{\pm} $, $ k_{\theta_\Hlie^{j}(h)} $, $ h_{\theta^{j}(i),m} $ respectively (where $ i \in I $, $h \in \Hlie$, $ r \in \ZZ$, $m\in \ZZ-\{0\}$). It gives by restriction an isomorphism of algebras between $ \U_{q}^{v}(sl_{n+1}^{tor}) $ and $ \U_{q}^{v,j}(sl_{n+1}^{tor}) $, still denoted by $\theta^{(j)}$ in the following. If $V$ is a $\U_q(\hat{sl}_{n+1})'$-module, we denote by $V^{(j)}$ the induced $\U_{q}^{v,j}(sl_{n+1}^{tor})$-module.

For $i\in I$, the subalgebra $ \hat{\U}_i $ generated by the $x_{i,r}^{\pm}, h_{i,m}, k_{p h_{i}}$ ($ r \in \ZZ$, $ m \in \ZZ - \{0\}$, $p \in \QQ $) is isomorphic to $ \U_{q}(\hat{sl}_{2})' $.

We have a triangular decomposition of $\U_q(sl_{n+1}^{tor})$.

\begin{thm}\label{dtrian} \cite{miki_representations_2000, nakajima_quiver_2001} We have an isomorphism of vector spaces
$$\U_q(sl_{n+1}^{tor})\simeq \U_q(sl_{n+1}^{tor})^-\otimes\U_q(\hat{\Hlie})\otimes\U_q(sl_{n+1}^{tor})^+,$$ 
where $\U_q(sl_{n+1}^{tor})^{\pm}$ (resp. $\U_q(\hat{\Hlie})$) is generated by the $x_{i,r}^{\pm}$ (resp. the $k_h$, the $h_{i,m}$).
\end{thm}

\subsection{Representations of $\U_q(sl_{n+1}^{tor})$}

\begin{defi} A representation $V$ of $\U_q(sl_{n+1}^{tor})$ is said to be integrable if $\mathrm{Res}(V)$ is integrable as a $\U_q(\hat{sl}_{n+1})$-module.
\end{defi}

\begin{defi} A representation $V$ of $\U_q(sl_{n+1}^{tor})$ is said to be of $\ell$-highest weight if there is $v\in V$ such that 
\begin{enumerate}
\item[(i)] $V = \U_q(sl_{n+1}^{tor})^- \cdot v$, 
\item[(ii)] $\U_q(\hat{\Hlie}) \cdot v=\CC v$, 
\item[(iii)] for any $i\in I, r\in\ZZ$, $x_{i,r}^+ \cdot v=0$.
\end{enumerate}
\end{defi}

For $\gamma \in \mathrm{Hom}(\U_q(\hat{\Hlie}), \CC)$ an algebra morphism, by Theorem \ref{dtrian} we have a corresponding Verma module $M(\gamma)$ and a simple representation $V(\gamma)$ which are $\ell$-highest weight. Then we have:

\begin{thm}\label{cond} \cite{miki_representations_2000, nakajima_quiver_2001} The simple representations $V(\gamma)$ of $\U_q(sl_{n+1}^{tor})$ are integrable if there is $(\lambda ,(P_i)_{i \in I})\in P^+ \times (1+u\mathbb{C}[u])^{I}$ satisfying $\gamma(k_h) = q^{\lambda(h)}$ and for $i\in I$ the relation in $\mathbb{C}[[z]]$ (resp. in $\mathbb{C}[[z^{-1}]]$)
$$\gamma(\phi_i^\pm(z))=q^{\text{deg}(P_i)}\frac{P_i(zq^{-1})}{P_i(zq)}.$$
\end{thm}

\noindent The polynomials $P_i$ are called Drinfeld polynomials and the representation $V(\gamma)$ is then denoted by $V(\lambda, (P_i)_{i \in I})$. Such a representation is also integrable in the sense of \cite{hernandez_quantum_2009}, i.e. $V(\lambda, (P_i)_{i \in I})$ satisfies conditions (iii) and (iv) of Remark \ref{remintmod}.

The Kirillov-Reshetikhin module associated to $k \geq 0$, $a \in \CC^{\ast}$ and $0 \leq \ell \leq n$, is the simple integrable representation of weight $k \Lambda_\ell$ with the $n$--tuple
$$P_i(u) = \left\lbrace \begin{array}{l} (1-ua)(1-uaq^{2}) \cdots (1-uaq^{2(k-1)}) \ \mathrm{for} \ i = \ell, \\ 1 \ \mathrm{for} \ i \neq \ell. \end{array} \right.$$
\noindent If $k = 1$, it is also called fundamental module.

Consider an integrable representation $V$ of $\U_q(sl_{n+1}^{tor})$. As the subalgebra $\U_q(\hat{\Hlie})$ is commutative, we have a decomposition of the weight spaces $V_{\nu}$ in simultaneous generalized eigenspaces
$$V_\nu = \bigoplus_{(\nu, \gamma) \in P \times \mathrm{Hom}(\U_q(\hat{\Hlie}), \CC)} V_{(\nu, \gamma)},$$
where $V_{(\nu, \gamma)} = \{x \in V / \exists p \in \NN, \forall i \in I, \forall m \geq 0, (\phi_{i, \pm m}^{\pm} - \gamma(\phi_{i, \pm m}^{\pm}))^{p} \cdot x = 0 \}$. If $V_{(\nu, \gamma)} \neq \{0\}$, $(\nu, \gamma)$ is called an $\ell$-weight of $V$.

\begin{defi}
A $\U_q(sl_{n+1}^{tor})$-module $V$ is weighted if the Cartan subalgebra $\U_q(\hat{\Hlie})$ acts on $V$ by diagonalizable operators. The module $V$ is thin if it is weighted and the joint spectrum is simple.
\end{defi}

The terminology is different in \cite{ feigin_quantum_2011, feigin_quantum_2011-1, feigin_quantum_2012, feigin_representations_2013}: a thin module is called tame.

\begin{defi}\cite{frenkel_$q$-characters_1999, hernandez_representations_2005, nakajima_quiver_2001} The $q$--character of an integrable representation $V$ of $\U_q(sl_{n+1}^{tor})$ with finite-dimensional $\ell$-weight spaces is defined by the formal sum
$$\chi_q(V) = \sum_{(\nu, \gamma) \in P \times \mathrm{Hom}(\U_q(\hat{\Hlie}), \CC)} \dim(V_{(\nu, \gamma)}) e(\nu, \gamma).$$
\end{defi}

\noindent Furthermore if the weight spaces of $V$ are finite-dimensional we have
$$\chi(\mathrm{Res}(V)) = \beta(\chi_q(V)),$$
where $\mathrm{Res}(V)$ still denotes the restricted $\U_q(\hat{sl}_{n+1})$-module obtained from $V$ and 
$$\beta : \prod_{(\nu, \gamma) \in P \times \mathrm{Hom}(\U_q(\hat{\Hlie}), \CC)}  \ZZ e(\nu, \gamma) \rightarrow \prod_{\nu \in P}  \ZZ e(\nu)$$
is $\ZZ$-linear such that $\beta(e(\nu, \gamma)) = e(\nu)$ for all $(\nu, \gamma) \in P \times \mathrm{Hom}(\U_q(\hat{\Hlie}), \CC)$.

\begin{prop}\cite{frenkel_$q$-characters_1999, hernandez_representations_2005, nakajima_quiver_2001}
Let $V$ be an integrable representation of $\U_q(sl_{n+1}^{tor})$. An $\ell$-weight $(\nu, \gamma) \in P \times \mathrm{Hom}(\U_q(\hat{\Hlie}), \CC)$ of $V$ satisfies the property
\begin{enumerate}
\item[(i)] there exist polynomials $Q_i(z), R_i(z) \in \CC[z]$ ($i \in I$) of constant term 1 such that in $\CC[[z]]$ (resp. in $\CC[[z^{-1}]]$):
\begin{equation}\label{lwr}
\sum_{m \geq 0} \gamma(\phi_{i, \pm m}^{\pm}) z^{\pm m} = q^{\deg(Q_i) - \deg(R_i)} \dfrac{Q_i(zq^{-1}) R_i(zq)}{Q_i(zq) R_i(zq^{-1})}.
\end{equation}
\end{enumerate}
At more, if $V$ has a finite composition series $L_0 = \{0\} \subset L_1 \subset L_2 \subset \dots \subset L_k = V$ such that $L_{j+1}/L_j \simeq V(\lambda_j, (P_i^j)_{i \in I})$ where the roots of $P_i^j$ are in $q^{\ZZ}$ for all $i \in I$, $0 \leq j \leq k-1$, then
\begin{enumerate}
\item[(ii)] there exist $\omega \in P^{+}$, $\alpha \in Q^{+}$ satisfying $\nu = \omega - \alpha$,
\item[(iii)] the zeros of the polynomials $Q_i(z), R_i(z)$ are in $q^{\ZZ}$.
\end{enumerate}
\end{prop}

If $V$ is a Kirillov-Reshetikhin module, one reduces to the case where the defining parameter $a$ is in $q^{\ZZ}$ by twisting the action by the automorphisms $t_b$ of $\U_q(sl_{n+1}^{tor})$ given by ($b \in \CC^{\ast}$)
$$t_b(x_{i,r}^{\pm}) = b^{r}x_{i,r}^{\pm}, \ t_b(h_{i,m}^{\pm}) = b^{m}h_{i,m}^{\pm}, \ t_b(k_{h}) = k_{h}.$$

Consider formal variables $ Y_{i,l}^{\pm 1} $, $ e^{\nu} $ ($ i \in I $, $ l \in \ZZ$, $ \nu \in P $) and let $ A $ be the group of monomials of the form $ m = e^{\omega(m)} \prod_{i \in I, l \in \ZZ} Y_{i,l}^{u_{i,l}(m)} $ where $ u_{i,l}(m) \in \ZZ $, $ \omega(m) \in P $ are such that
$$\sum_{l \in \ZZ} u_{i,l}(m) = \omega(m)(h_i).$$

\noindent For example, $ e^{\pm \Lambda_{i}}Y_{i,l}^{\pm 1} \in A$ and $ A_{i,l} = e^{\alpha_{i}} Y_{i, l-1} Y_{i, l+1} Y_{i-1,l}^{-1}Y_{i+1,l}^{-1} \in A$. A monomial $ m $ is said to be $J$-dominant ($ J \subset I $) if for all $ j \in J $ and $ l\in \ZZ $ we have $ u_{j,l}(m) \geq 0 $. An $ I $-dominant monomial is said to be dominant.

\begin{rem}\label{remweight}
Let us fix a monomial $m \in A$ and consider monomials $m'$ which are products of $m$ with various $A_{i,l}^{\pm 1}$'s ($i \in I, l \in \ZZ$). By \cite[Remark 2.1]{hernandez_level_2006}, $\omega(m')$ is uniquely determined by $\omega(m)$ and $u_{i,l}(m')$. So in the following when we are in this situation, the term $e^{\omega(m')}$ will be safely omitted.
\end{rem}

Let $V$ be an integrable $\U_q(sl_{n+1}^{tor})$-module such that for all $\ell$-weight $(\nu, \gamma)$ of $V$, the roots of the associated polynomials  $Q_i(z)$ and $R_i(z)$ are in $q^{\ZZ}$ for all $i \in I$. For $(\nu, \gamma) \in P \times \mathrm{Hom}(\U_q(\hat{\Hlie}), \CC)$ an $\ell$-weight of $V$, one defines the monomial $m_{(\nu, \gamma)} = e^{\nu} \prod_{i \in I, l \in \ZZ} Y_{i,l}^{u_{i,l}-v_{i,l}}$ where
$$Q_{i}(z) = \prod_{l \in \ZZ} (1-zq^{l})^{u_{i,l}} \ \mathrm{and} \ R_{i}(z) = \prod_{l \in \ZZ} (1-zq^{l})^{v_{i,l}}.$$
We denote $V_{(\nu, \gamma)} = V_{m_{(\nu, \gamma)}}$. We rewrite the $q$--character of an integrable representation $V$ with finite-dimensional $\ell$-weight spaces by the formal sum
$$\chi_q(V) = \sum_{m} \dim(V_{m}) m \in \ZZ[[e^{\nu}, Y_{i,l}^{\pm 1}]]_{\nu \in P, i \in I, l \in \ZZ}.$$
Let us denote by $\mathcal{M}(V)$ the set of monomials occurring in $\chi_q(V)$.

By this correspondence between $\ell$-weights and monomials due to Frenkel-Reshetikhin \cite{frenkel_$q$-characters_1999}, the $I$-tuple of Drinfeld polynomials with zeros in $q^{\ZZ}$ are identified with the dominant monomials. In particular for a dominant monomial $m$, one denotes by $V(m)$ the simple module of $\ell$-highest weight $m$. For example $V(e^{k \Lambda_\ell}Y_{\ell, s}Y_{\ell, s+2} \dots Y_{\ell, s+2(k-1)})$ is the Kirillov-Reshetikhin module associated to $k \geq 0$, $a = q^{s} \in \CC^{\ast}$ ($s \in \ZZ$) and $\ell \in I$, and $V(e^{\Lambda_\ell}Y_{\ell, s})$ is the fundamental module associated to $a = q^{s} \in \CC^{\ast}$ ($s \in \ZZ$) and $\ell \in I$.\\

Similar results hold for the quantum affine algebra $\U_q(\hat{sl}_{n+1})'$ due to Chari-Pressley \cite{chari_guide_1994}. In this case, the simple integrable representations are finite-dimensional and denoted $V_0((P_i)_{i \in I_0})$ in the following. Note that the weights $\lambda \in P_0$ can be omitted here because they are completely determined by the Drinfeld polynomials $(P_i)_{i \in I_0}$,
$$\lambda = \deg(P_1) \Lambda_1 + \dots + \deg(P_n) \Lambda_n.$$
In the same way if $V$ is a Kirillov-Reshetikhin module of $\U_q(\hat{sl}_{n+1})'$, its $\ell$-weights can be only considered as elements of $\mathrm{Hom}(\U_q(\hat{\Hlie}_0), \CC)$ (where $\U_q(\hat{\Hlie}_0)$ is the subalgebra of $\U_q(\hat{sl}_{n+1})'$ generated by $k_h$ ($h \in \Hlie_0$) and $ h_{i,m}$ ($i \in I_0, m \in \ZZ-\{0\}$)). They still satisfy relations (\ref{lwr}). By twisting the action on $V$ by an automorphism $t_b$ of $\U_q(\hat{sl}_{n+1})'$ for some $b \in \CC^{\ast}$, it can be parametrized as above by Laurent monomials in $\ZZ[Y_{i,l}^{\pm 1}]_{i \in I_0, l \in \ZZ}$. The weight $\omega(m) \in P_0$ of a monomial $m \in \ZZ[Y_{i,l}^{\pm 1}]_{i \in I_0, l \in \ZZ}$ can be omitted here because it is completely determined by the $u_{i,l}(m)$ ($i \in I_0, l \in \ZZ$).

Let us recall that the Kirillov-Reshetikhin modules $V_0(Y_{\ell, s}Y_{\ell, s+2} \dots Y_{\ell, s+2(k-1)})$ ($k\geq 0,$ $s \in \ZZ,$ $\ell \in I_0$) over $\U_q(\hat{sl}_{n+1})'$ can be obtained from the $\U_q(sl_{n+1})$-modules $V_0(k \Lambda_\ell)$ as follows: there exist evaluation morphisms $\mathrm{ev}_a : \U_q(\hat{sl}_{n+1})' \rightarrow \U_q(sl_{n+1})$ ($a \in \CC^{\ast}$), which send $x_{i,0}, k_h$ on $x_i, k_h$ respectively ($i \in I_0, h \in \Hlie_0$). So the $\U_q(\hat{sl}_{n+1})'$-module $V_0(Y_{\ell, s}Y_{\ell, s+2} \dots Y_{\ell, s+2(k-1)})$ is obtained by pulling back the action of $\U_q(sl_{n+1})$ on $V_0(k \Lambda_\ell)$ by $\mathrm{ev}_a$ for some $a \in \CC^{\ast}$. In particular, $V_0(Y_{\ell, s}Y_{\ell, s+2} \dots Y_{\ell, s+2(k-1)})$ is irreducible as a $\U_q(sl_{n+1})$-module, isomorphic to $V_0(k \Lambda_\ell)$.

We have defined irreducible finite-dimensional $\U_q(\hat{sl}_{n+1})'$-modules $W(\varpi_\ell)$ ($\ell \in I_0$) in Section 2.4. One can determine them in term of the Drinfeld realization. For that we need the following additional result.

\begin{lem}\cite{nakajima_extremal_2004}
Let $v$ be a vector of an integrable $\U_q(\hat{sl}_{n+1})'$-module of weight \linebreak $\lambda \in \cl(P^0)$ such that for all $i \in I_0$, $\lambda(h_i) \geq 0$. Then the following conditions are equivalent:
\begin{enumerate}
\item[(i)] $v$ is an extremal vector,
\item[(ii)] $x_{i,r}^+ \cdot v = 0$ for all $i \in I_0, r \in \ZZ$.
\end{enumerate}
\end{lem}

\noindent As a direct consequence of these results, $W(\varpi_\ell)$ is isomorphic to a fundamental representation $V_0((1-\delta_{\ell,i} au)_{i \in I_0})$ for a special choice of the spectral parameter $a \in \CC^{\ast}$ (see \cite[Remark 3.3]{nakajima_extremal_2004} for an expression of it). In particular for this spectral parameter, one deduces that $V_0((1-\delta_{\ell,i} au)_{i \in I_0})$ has a crystal basis.

Let $\mathcal{C}_l$ be the category of finite-dimensional $\U_q(\hat{sl}_{n+1})'$-modules (of type 1) and $\mathcal{R}_l$ its Grothendieck ring. Recall that $\mathcal{C}_l$ is an abelian monoidal category, not semi-simple, with as simple objects the $V_0((P_i)_{i \in I_0})$ and $\mathcal{R}_l$ is the polynomial ring over $\ZZ$ in the classes $[V_0((1-\delta_{\ell,i} au)_{i \in I_0})]$ ($\ell \in I_0, a \in \CC^{\ast}$) (see \cite{chari_guide_1994, frenkel_$q$-characters_1999}). As in \cite{hernandez_cluster_2010}, we consider $\mathcal{C}_{l, \ZZ}$ the full subcategory of $\mathcal{C}_l$ whose objects $V$ satisfy
\begin{itemize}
\item[]for every composition factor $S$ of $V$, the roots of the Drinfeld polynomials $(P_i(u))_{i \in I_0}$ belong to $q^{\ZZ}$.
\end{itemize}
This is also an abelian monoidal category, not semi-simple and the Grothendieck ring $\mathcal{R}_{l,\ZZ}$ of $\mathcal{C}_{l,\ZZ}$ is the subring of $\mathcal{R}_l$ generated by the classes $[V_0(Y_{\ell, s})]$ with $\ell \in I_0, s \in \ZZ$ (see \cite{frenkel_combinatorics_2001}).

\begin{thm}\cite{frenkel_$q$-characters_1999}
$\chi_q$ induces a ring morphism $\chi_q : \mathcal{R}_{l,\ZZ} \rightarrow \ZZ[Y_{i,l}^{\pm 1}]_{i \in I_0, l \in \ZZ}$, called morphism of $q$--characters. Furthermore we have the following commutative diagram
\begin{eqnarray*}
\xymatrix{
\mathcal{R}_{l, \ZZ} \ar[rr]^{\chi_q} \ar[d]_{\mathrm{Res}} & & \ZZ[Y_{i,l}^{\pm 1}]_{i \in I_0, l \in \ZZ} \ar[d]^\beta \\
\mathcal{R} \ar[rr]^\chi & & \bigoplus_{\nu \in P_0} \ZZ e(\nu)
}
\end{eqnarray*}
where the ring morphism $\mathrm{Res} : \mathcal{R}_{l, \ZZ} \rightarrow \mathcal{R}$ is the restriction and $\beta : \ZZ[Y_{i,l}^{\pm 1}]_{i \in I_0, l \in \ZZ} \rightarrow \bigoplus_{\nu \in P_0} \ZZ e(\nu)$ is defined by $\beta(m) = e(\omega(m))$.
\end{thm}

One does not have expression of $q$--character of a representation in general. But explicit formulas exist for the fundamental modules and the Kirillov-Reshetikhin modules over $\U_q(\hat{sl}_{n+1})'$ and $\U_q(sl_{n+1}^{tor})$, given in terms of tableaux \cite{hernandez_algebra_2011, nakajima_$t$-analogs_2003}.

\subsection{Extremal loop weight modules}

We recall the notion of extremal loop weight modules for $ \mathcal{U}_{q}(sl_{n+1}^{tor}) $. The main motivation is the construction of finite-dimensional representations of the quantum toroidal algebra as in the theory of Kashiwara, but at roots of unity in this case.

\begin{defi}\label{defelm}\cite{hernandez_quantum_2009} An extremal loop weight module of $ \mathcal{U}_{q}(sl_{n+1}^{tor}) $ is an integrable representation $ V $ such that there is an $\ell$-weight vector $ v \in V $ satisfying

\begin{enumerate}
\item[(i)] $ \mathcal{U}_{q}(sl_{n+1}^{tor}) \cdot v = V $,
\item[(ii)] $ v $ is extremal for $ \mathcal{U}_{q}^{h}(sl_{n+1}^{tor}) $,
\item[(iii)] $ \mathcal{U}_{q}^{v, j}(sl_{n+1}^{tor}) \cdot w $ is finite-dimensional for all $ w \in V $ and $ j \in I $.
\end{enumerate} \end{defi}

\begin{ex}
If $m$ is dominant, the simple $\ell$-highest weight module $V(m)$ of $\ell$-highest weight $m$ is an extremal loop weight module.
\end{ex}

\noindent An example of such a representation which is neither of $\ell$-highest weight nor of $\ell$-lowest weight is given in \cite{hernandez_quantum_2009}. The goal of this article is to construct a new family of extremal loop weight modules, called \textit{extremal fundamental loop weight modules}.

\section{Study of the monomial crystals $\mathcal{M}(e^{\varpi_\ell}Y_{\ell,0} Y_{0, d_\ell}^{-1})$}

We will relate in our paper the monomial realizations $\mathcal{M}_\ell$ of level 0 extremal fundamental weight crystals $\mathcal{B}(\varpi_\ell)$ ($1 \leq \ell \leq n$) of $ \U_{q}(\hat{sl}_{n+1}) $ with integrable representations of $\U_q(sl_{n+1}^{tor})$. In this section, we study the combinatorics of these monomial realizations, the main point being the use of promotion operators for level 0 extremal fundamental weight crystals introduced below. This is the first step of the construction of integrable modules associated to $\mathcal{M}_\ell$.

In the first part, one gives the definition of the monomial $\U_q(\hat{sl}_{n+1})$-crystal $\mathcal{M}$ \cite{kashiwara_realizations_2003, nakajima_$t$-analogs_2003}. This definition holds when the considered Cartan matrix has no odd cycle. So it does not work for $\U_q(\hat{sl}_{n+1})$  when $n$ is even, and we have to assume that $n = 2r+1$ ($r \geq 1$) is odd until the end of the article. Following \cite{hernandez_level_2006}, we recall the monomial realization of $\mathcal{B}(\varpi_\ell)$ ($1 \leq \ell \leq n)$: it is isomorphic to the sub-$ \U_{q}(\hat{sl}_{n+1}) $-crystal $\mathcal{M}_\ell = \mathcal{M}(e^{\varpi_\ell}Y_{\ell,0}Y_{0, d_\ell}^{-1})$ of $\mathcal{M}$ generated by the monomial $e^{\varpi_\ell}Y_{\ell,0}Y_{0, d_\ell}^{-1}$ (with $d_\ell $ equal to $ \min(\ell, n+1-\ell)$). Furthermore we define the notions of $q$-closed monomial set and of monomial set closed by the Kashiwara operators, respectively related to the theory of $q$-characters and to the combinatorics of crystals.

The monomial crystals $\mathcal{M}_\ell$ are already studied in \cite{hernandez_level_2006}: the monomials occurring in these crystals are explicitly given for $1\leq \ell \leq n$ and their automorphisms $z_\ell$ are described in terms of monomials. We recall these results in the second part.

In the third part, we introduce promotion operators for level 0 extremal fundamental weight crystals $\mathcal{B}(\varpi_\ell)$ ($1 \leq \ell \leq n$). The promotion operators were introduced in \cite{shimozono_affine_2002} for the Young tableaux realization of the finite $\U_q(sl_{n+1})$-crystals $\mathcal{B}_0(k \Lambda_\ell)$ ($k \in \NN^{\ast}$, $1 \leq \ell \leq n$) and studied in numerous papers (see \cite{bandlow_uniqueness_2010, fourier_kirillov-reshetikhin_2009, okado_existence_2008, schilling_combinatorial_2008, shimozono_affine_2002} and references therein). It is the counterpart at the level of crystals of the cyclic symmetry of the Dynkin diagram of type $A_n^{(1)}$. After recalling these definitions, we extend the promotion operators for the level 0 extremal fundamental weight crystals $\mathcal{B}(\varpi_\ell)$. Finally we specify the promotion operator of $\mathcal{B}(\varpi_\ell)$ in its monomial realization $\mathcal{M}_\ell$.

In the last part, we use promotion operators to obtain a new description of $\mathcal{M}_\ell$. In particular, we improve results given in \cite{hernandez_level_2006} for these crystals. Furthermore we determine the $\ell \in I_0$ for which the monomial crystals $\mathcal{M}_\ell$ are closed (Theorem \ref{proclocry}): this is the case if and only if $\ell=1, r+1$ or $n$.

\subsection{Monomial crystals} In this section we define the monomial crystal $\mathcal{M}$ of \linebreak $\U_q(\hat{sl}_{n+1})$ when $n = 2r+1$ is supposed to be odd, following \cite{kashiwara_realizations_2003, nakajima_$t$-analogs_2003}. Monomial realizations of the crystals $\mathcal{B}(\lambda)$ with $\lambda \in P$, in particular of $\mathcal{B}(\varpi_\ell)$ ($1 \leq \ell \leq n$), are studied in \cite{hernandez_level_2006, kashiwara_realizations_2003, nakajima_$t$-analogs_2003}. They are obtained as subcrystals of $\mathcal{M}$ generated by a monomial.  We recall these results here. Finally we introduce new notions of $q$-closed monomial set and of monomial set closed by the Kashiwara operators.

As we have said above, the definition of the monomial crystal $\mathcal{M}$ requires that the considered Cartan matrix $C$ is without odd cycle. So we assume that $n = 2r+1$ ($r \geq 1$) is odd until the end of the article. In particular there is a function $s : I \rightarrow \{0,1\}, i \mapsto s_i$ such that $C_{i,j} = -1$ implies $s_i + s_j = 1$.

Consider the subgroup $\mathcal{M} \subset A$ defined by
$$\mathcal{M} = \{ m \in A \mid u_{i,l}(m) = 0 \text{ if } l \equiv s_i +1 \text{ mod } 2 \}.$$
Following \cite{kashiwara_realizations_2003, nakajima_$t$-analogs_2003}, let us define $ \mathrm{wt} : \mathcal{M} \rightarrow P $ and $ \varepsilon_{i}, \varphi_{i}, p_{i}, q_{i} : \mathcal{M} \rightarrow \ZZ \cup \{ \infty \} \cup \{- \infty \} $ for $ i \in I $ by ($ m \in \mathcal{M} $)

\begin{center}
$\mathrm{wt}(m) = \omega(m),$ \\
$\varphi_{i,L}(m) = \displaystyle{\sum_{l \leq L}} u_{i,l}(m), \: \varphi_{i}(m) = \mathrm{max} \{ \varphi_{i,L}(m) | L \in \ZZ \} \geq 0,$ \\
$\varepsilon_{i,L}(m) = - \displaystyle{\sum_{l \geq L}} u_{i,l}(m), \: \varepsilon_{i}(m) = \mathrm{max} \{ \varepsilon_{i,L}(m) | L \in \ZZ \} \geq 0, $ \\
$\begin{array}{rcl}
 p_{i}(m) & = & \mathrm{max}\{L \in \ZZ | \varepsilon_{i,L}(m) = \varepsilon_{i}(m) \} \\ 
 & = & \mathrm{max}\left\lbrace L \in \ZZ \left| \displaystyle{\sum_{l < L}} u_{i,l}(m) = \varphi_{i}(m) \right. \right\rbrace ,
\end{array} $ \\
$\begin{array}{rcl}
 q_{i}(m) & = & \mathrm{min}\{L \in \ZZ | \varphi_{i,L}(m) = \varphi_{i}(m) \} \\ 
 & = & \mathrm{min}\left\lbrace L \in \ZZ \left| - \displaystyle{\sum_{l > L}} u_{i,l}(m) = \varepsilon_{i}(m) \right. \right\rbrace .
\end{array}$
\end{center}

\noindent Then we define $ \tilde{e}_{i} , \tilde{f}_{i} : \mathcal{M} \rightarrow \mathcal{M} \cup \{0 \} $ for $ i \in I $ by
\begin{eqnarray*}
\tilde{e}_{i}\cdot m & = & \left\lbrace 
\begin{array}{ll}
0 & \mathrm{if} \: \varepsilon_{i}(m) = 0, \\
m A_{i, p_{i}(m) -1} & \mathrm{if} \: \varepsilon_{i}(m) > 0,
\end{array} \right. \\
\tilde{f}_{i} \cdot m & = & \left\lbrace 
\begin{array}{ll}
0 & \mathrm{if} \: \varphi_{i}(m) = 0, \\
m A_{i, q_{i}(m) +1}^{-1} & \mathrm{if} \: \varphi_{i}(m) > 0.
\end{array}
  \right.
\end{eqnarray*}

\begin{thm}\cite{kashiwara_realizations_2003, nakajima_$t$-analogs_2003}
$ (\mathcal{M}, \mathrm{wt}, \varepsilon_{i}, \varphi_{i}, \tilde{e}_{i}, \tilde{f}_{i}) $ is a $\U_q(\hat{sl}_{n+1})$--crystal, called the monomial crystal.
\end{thm}

\begin{rem}
When $n$ is even, the Dynkin diagram of type $A_n^{(1)}$ is not bipartite. In this case, $(\mathcal{M}, \mathrm{wt}, \varepsilon_{i}, \varphi_{i}, \tilde{e}_{i}, \tilde{f}_{i})$ does not satisfy the axioms of crystal (see \cite{kashiwara_realizations_2003}). Another crystal structures are defined on (a subset of) $A$ in \cite{kashiwara_realizations_2003}. But the monomials used are different with those occurring in the theory of $q$--characters of $\U_q(sl_{n+1}^{tor})$-modules and it is not useful for what we will do in the next sections.
\end{rem}

For $ m \in \mathcal{M} $ denote by $ \mathcal{M}(m) $ the connected subcrystal of $ \mathcal{M} $ generated by $m$. As it is explained above, the weight of a monomial $m' \in \mathcal{M}(m)$ is determined by $\omega(m)$ and $u_{i,l}(m')$ (Remark \ref{remweight}). So we will omit the term $e^{\omega(m')}$ and we just specify the weight of the monomial $m$. For $J \subset I$ and $m \in \mathcal{M}$, denote by $\mathcal{M}_J(m)$ the set of monomials obtained from $m$ by applying the Kashiwara operators $\tilde{e}_i, \tilde{f}_i$ for $i \in J$. It is a connected sub-$J$-crystal of $\mathcal{M}(m)$ generated by $m$.

For $p\in\ZZ$ and $\alpha\in \ZZ\delta$, let $\tau_{2p, \alpha}$ be the map $\tau_{2p,\alpha}\colon \mathcal{M}\rightarrow\mathcal{M}$ defined by
$$\tau_{2p,\alpha}(e^\lambda \prod Y_{i,n}^{u_{i,n}}) = e^{\lambda+\alpha} \prod Y_{i,n+2p}^{u_{i,n}}.$$
This is a $P_{\mathrm{cl}}$-crystal automorphism of the crystal $\mathcal{M}$, also called shift automorphism in the following.\\

The following result was proved in \cite{kashiwara_realizations_2003, nakajima_$t$-analogs_2003} when $m$ is a dominant monomial and is generalized in \cite{hernandez_level_2006} for all $m \in \mathcal{M}$.

\begin{thm} For $m\in\mathcal{M}$, the crystal $\mathcal{M}(m)$ is isomorphic to a connected component of the crystal
$\mathcal{B}(\lambda)$ of an extremal weight module for some $\lambda\in P$.\end{thm}

It was shown in \cite{kashiwara_level-zero_2002} that the fundamental extremal crystals $\mathcal{B}(\varpi_\ell)$ are connected for all $\ell \in I_0$. Let $$d_\ell = \min(\ell, n+1-\ell)$$ be the distance between the nodes $0$ and $\ell$ in the Dynkin diagram of type $A_n^{(1)}$. We have the following monomial realization of $\mathcal{B}(\varpi_\ell)$.

\begin{thm}\label{monrzt} \cite{hernandez_level_2006}
Set $M = e^{\varpi_\ell}Y_{\ell, 0}Y_{0,d_\ell}^{-1}$ for $\ell \in I_0$. Then $M$ is extremal in $\mathcal{M}$ and $\mathcal{M}(M)\simeq \mathcal{B}(\varpi_{\ell})$ as $P$-crystals.\end{thm}

One can define in the same way the monomial crystal $\mathcal{M}_0$ associated to $\U_q(sl_{n+1})$. It can be done for all $n \geq 2$, the Cartan matrix of type $A_n$ being without cycle. As it is said above, the weights of monomials are completely determined by the powers of variables $Y_{i,l}^{\pm 1}$ in this case. So they can be safely omitted. For $m \in \mathcal{M}_0$, we denote by $\mathcal{M}_0(m)$ the subcrystal of $\mathcal{M}_0$ generated by $m$. We have

\begin{prop}\cite{kashiwara_realizations_2003, nakajima_$t$-analogs_2003}\label{propcrystruct} The $\U_q(sl_{n+1})$-crystals $\mathcal{M}_{0}(Y_{i, k})$ and $\mathcal{B}_0(\Lambda_i)$ are isomorphic for all $i \in I_0$ and $k \in \ZZ$.
\end{prop}

For $i \in I$, set $\Xi_i : \mathcal{M} \rightarrow \mathcal{M}$ the map sending the variables $ Y_{j,*}^{\pm 1}, e^{\nu} $ to $ 1 $ for all $j \neq i$ and $\nu \in P$, and $ Y_{i, *}^{\pm 1} $ to themselves. Another map will be used below: $\Xi^i : \mathcal{M} \rightarrow \mathcal{M}$ which sends the variables $ Y_{j,*}^{\pm 1} $ to themselves if $j \neq i$ and $ Y_{i, *}^{\pm 1}, e^{\nu} $ to $1$ for all $\nu \in P$. These two maps are also defined in \cite{frenkel_combinatorics_2001} and denoted by $\beta_{\{i\}}$ and $\beta_{I_i}$ respectively.

\begin{defi}\label{defclocry}
\begin{enumerate}
\item[(i)] A set of monomials $\mathcal{S} \subset \mathcal{M}$ is said to be $q$-closed in the direction $i$ $(i \in I)$ if for all $m \in \mathcal{S}$ there exists a finite subset $$\mathcal{S}_m \subset \mathcal{S} \cap \left(m \cdot \prod_{l \in \ZZ} A_{i,l}^{\ZZ} \right)$$ containing $m$ and a sequence $(n_{s})_{s \in \mathcal{S}_m}$ of positive integers such that \linebreak $\Xi_{i}\left( \sum_{s \in \mathcal{S}_m} n_{s} \cdot s \right)$ is the $q$--character of a representation of $\hat{\mathcal{U}}_i$.
\item[(ii)] A set of monomials $\mathcal{S}$ is said to be $J$-$q$-closed ($J \subset I$), or simply $q$-closed if $J=I$, if $\mathcal{S}$ is $q$-closed in the direction $i$ for all $i \in J$.
\item[(iii)] A set of monomials $\mathcal{S} \subset \mathcal{M}$ is said to be $J$-closed by the Kashiwara operators ($J \subset I$), or simply closed by the Kashiwara operators if $J=I$, if the operators $\tilde{e}_i$, $\tilde{f}_i$ preserve $\mathcal{S}$ for all $i \in J$.
\item[(iv)] A set of monomials $\mathcal{S} \subset \mathcal{M}$ which is $J$-$q$-closed and $J$-closed by the Kashiwara operators ($J \subset I$), is called a $J$-closed set. If $J = I$, it is simply called a closed monomial set.
\end{enumerate}
\end{defi}

\begin{rem} \begin{enumerate}
\item[(i)] The definition of $q$-closed set is inspired by the theory of $q$--characters and the Frenkel-Mukhin algorithm \cite{frenkel_combinatorics_2001}. In particular, it involves $q$-characters of $\U_q(\hat{sl}_2)'$-modules. Let us recall that in this case, the image of $\chi_q : \mathcal{R}_{l, \ZZ} \rightarrow \ZZ[Y_{1, l}^{\pm 1}]_{l \in \ZZ}$ is known (see \cite{frenkel_$q$-characters_1999}): it is equal to the subring $\ZZ[(Y_{1,l} + Y_{1,l+2}^{-1})]_{l \in \ZZ}$ of $\ZZ[Y_{1, l}^{\pm 1}]_{l \in \ZZ}$ generated by the $Y_{1,l} + Y_{1,l+2}^{-1}$ ($l \in \ZZ$).
\item[(ii)] The notion of $q$-closed set holds also for the monomial $\U_q(sl_{n+1})$-crystal $\mathcal{M}_0$. Further it extends naturally when $q$ is specialized at roots of unity, by using the theory of $q$--characters at roots of unity \cite{frenkel_$q$-characters_2002}.
\end{enumerate}
\end{rem}

Let $V$ be an integrable $\U_q(sl_{n+1}^{tor})$-module such that for all $\ell$-weight $(\nu, \gamma)$ of $V$, $V_{(\nu, \gamma)}$ is finite-dimensional and the roots of the associated polynomials  $Q_i(z)$ and $R_i(z)$ are in $q^{\ZZ}$ for all $i \in I$. Then the monomial set $\mathcal{M}(V)$ is $q$-closed. Note that it is not necessary that the Frenkel-Mukhin algorithm holds for $V$: for example it does not work for the simple finite-dimensional $\U_q(\hat{sl}_3)'$-module $V_0(Y_{1,0}^{2}Y_{2,3}) \simeq V_0(Y_{1,0}Y_{2,3}) \otimes V_0(Y_{1,0})$ considered in \cite{hernandez_cluster_2010}, but $\mathcal{M}(V_0(Y_{1,0}^{2}Y_{2,3}))$ is $q$-closed.

In general, $\mathcal{M}(V)$ is not closed by the Kashiwara operators: for example the $q$--character of the $\U_q(\hat{sl}_2)'$-module $V_0(Y_{1,2}Y_{1,0}^{2})$ contains the monomial $Y_{1,0}$ but does not contain $Y_{1,2}^{-1}$. However, it holds for the fundamental $\U_q(\hat{sl}_{n+1})'$-modules. In fact by using the tableaux sum expressions of their $q$--characters given in \cite{nakajima_$t$-analogs_2003}, we have

%
%
%
%

\begin{prop}\cite{nakajima_$t$-analogs_2003}
Let $V_0(Y_{i,k})$ be a fundamental representation of $\U_q(\hat{sl}_{n+1})'$ \linebreak ($i \in I_0$, $k \in \ZZ$). Then the monomial sets $\mathcal{M}_{0}(Y_{i,k})$ and $\mathcal{M}(V_0(Y_{i,k}))$ are equal.
\end{prop}

In particular by Proposition \ref{propcrystruct}, $\mathcal{M}(V_0(Y_{i,k}))$ has a $\U_q(sl_{n+1})$-crystal structure isomorphic to $\mathcal{B}_0(\Lambda_i)$.  As a consequence we have

\begin{cor}
For all $1 \leq i \leq n$ and $k \in \ZZ$, the $\U_q(sl_{n+1})$-crystal $\mathcal{M}_{0}(Y_{i,k})$ is closed.
\end{cor}

Finally, let us give an example of monomial crystal which is not $q$-closed. Consider the $\U_q(sl_2)$-crystal $\mathcal{M}_{0}(Y_{1,4}Y_{1,0})$
$$Y_{1,4}Y_{1,0} \rightarrow Y_{1,6}^{-1}Y_{1,0} \rightarrow Y_{1,6}^{-1}Y_{1,2}^{-1}.$$
If $\mathcal{M}_{0}(Y_{1,4}Y_{1,0})$ is $q$-closed, it should contain $\mathcal{M}(V_0(Y_{1,4}Y_{1,0}))$. This is not the case, the $q$-character of $V_0(Y_{1,4}Y_{1,0})$ being
$$\chi_q(V_0(Y_{1,4}Y_{1,0})) = Y_{1,4}Y_{1,0} + Y_{1,6}^{-1}Y_{1,0} + Y_{1,4}Y_{1,2}^{-1} + Y_{1,6}^{-1}Y_{1,2}^{-1}.$$

\subsection{Description of the monomial crystal $ \mathcal{M}(e^{\varpi_\ell}Y_{\ell, 0} Y_{0, d_\ell}^{-1}) $}

Assume that $n = 2r+1$ is odd with $r \geq 1$. The monomial crystals $ \mathcal{M}(e^{\varpi_\ell}Y_{\ell, 0} Y_{0, d_\ell}^{-1}) $ are studied in \cite[Section 4]{hernandez_level_2006}: the monomials occurring in these crystals are explicitly described and the automorphisms $z_\ell$ are given in terms of monomials. We recall these results here.

To describe the monomial crystals $ \mathcal{M}(e^{\varpi_\ell}Y_{\ell, 0} Y_{0, d_\ell}^{-1}) $, we assume in this section that $ \ell \leq r+1 $ (as in \cite{hernandez_level_2006}). Let us begin by explaining why we can do that. We need the notion of twisted isomorphism of crystals (this definition appears in \cite{bandlow_uniqueness_2010}).

\begin{defi} Let $ \mathcal{B} $ and $ \mathcal{B'} $ be crystals over two isomorphic Dynkin diagrams $ D $ and $ D' $ with vertices respectively indexed by $ I $ and $ I' $ and let $ \theta : I \rightarrow I' $ be an isomorphism from $ D $ to $ D' $. Then $ \phi $ is a $ \theta $-twisted isomorphism if $\phi : \mathcal{B} \rightarrow \mathcal{B}'$ is a bijection map and for all $ b \in \mathcal{B} $ and $ i \in I $,
$$ \tilde{f}_{\theta(i)} \cdot \phi(b) = \phi( \tilde{f}_{i} \cdot b)  \: \mathrm{and} \: \tilde{e}_{\theta(i)} \cdot \phi(b) = \phi( \tilde{e}_{i} \cdot b). $$
\end{defi}

Let $\iota$ be the automorphism of the Dynkin diagram of type $A_{n}^{(1)}$ such that $\iota(k) = -k$ ($k \in I$) where $I$ is identified to the set $\ZZ/(n+1)\ZZ$. It defines an automorphism $\iota_\Hlie$ of $\Hlie$ by sending $h_i, d$ to $h_{\iota(i)}, d$ for all $i \in I$. Let $\psi : \mathcal{M} \rightarrow \mathcal{M}$ be the map defined by ($r \in \QQ$)
$$\psi \left(e^{r \delta} \prod (e^{\Lambda_i}Y_{i,n})^{u_{i,n}} \right) = e^{r \delta} \prod (e^{\Lambda_{-i}}Y_{-i,n})^{u_{i,n}}.$$
Then we show easily that $\psi$ is a $\iota$-twisted automorphism of the $P$-crystal $\mathcal{M}$. Furthermore it induces a $\iota$-twisted isomorphism
$$\psi : \mathcal{M}(e^{\varpi_\ell}Y_{\ell, 0} Y_{0, \ell}^{-1}) \longrightarrow \mathcal{M}(e^{\varpi_{n+1-\ell}}Y_{n+1-\ell, 0} Y_{0, \ell}^{-1})$$
between the monomial crystals $\mathcal{M}(e^{\varpi_\ell}Y_{\ell, 0} Y_{0, \ell}^{-1})$ and $\mathcal{M}(e^{\varpi_{n+1-\ell}}Y_{n+1-\ell, 0} Y_{0, \ell}^{-1})$ for all $1 \leq \ell \leq r+1$.

So one can assume that $1 \leq \ell \leq r+1$. In this case, $d_\ell = \ell$ and we study the crystal $\mathcal{M}(e^{\varpi_\ell}Y_{\ell,0}Y_{0,\ell}^{-1})$ (see \cite{hernandez_level_2006}). One defines the monomials

\newcommand{\te}{\tilde e}
\newcommand{\tf}{\tilde f}
\newenvironment{aenume}{%
  \begin{enumerate}%
  \renewcommand{\theenumi}{\alph{enumi}}%
  \renewcommand{\labelenumi}{(\theenumi)}%
  }{\end{enumerate}}

\begin{equation*}
  \ffbox{k}_p = Y_{k-1,p+k}^{-1} Y_{k,p+k-1} \qquad
  \text{for $1\le k\le n+1$, $p\in\ZZ$}.
\end{equation*}
with $Y_{n+1,p}=Y_{0,p}$ by convention. By Remark \ref{remweight}, the terms $e^{\omega(m')}$ can be safely omitted for all $m' \in \mathcal{M}(e^{\varpi_\ell}Y_{\ell,0}Y_{0,\ell}^{-1})$. Set  $M_0 = e^{\varpi_\ell}Y_{\ell,0}Y_{0,\ell}^{-1}$ and
\begin{equation*}
  \begin{split}
  M_j &= Y_{\ell,2j}Y_{0,n-\ell+1+2j}^{-1}
       Y_{j,\ell+j}^{-1}Y_{j,n-\ell+1+j}
\\
   &= \left( \ffbox{1}_{n-\ell+2j}\ffbox{2}_{n-\ell+2j-2}
   \cdots \ffbox{j}_{n-\ell+2}\right)
\times
  \left(\ffbox{j\!+\!1}_{\ell-1}\ffbox{j\!+\!2}_{\ell-3}
    \cdots \ffbox{\ell}_{1-\ell+2j}\right)
\\
   &=
   \prod_{p=1}^j \ffbox{p}_{n-\ell-2p+2j+2} \times
   \prod_{p=j+1}^\ell \ffbox{p}_{\ell+1-2p+2j}
  \end{split}
\end{equation*}
with $0\le j\le \ell$. In particular, $M_\ell =
Y_{\ell,n+1}Y_{0,n+1+\ell}^{-1} = \tau_{n+1, -\ell \delta}(M_0)$ and $M_1
= \tau_{2, -\delta}(M_0)$ for $\ell = r+1$. One defines other monomials in the following way: for $j \in \ZZ$ and a Young tableau of shape $(\ell)$ $T = (1 \leq i_1 < i_2 < \dots < i_\ell \leq n+1)$ we set
\begin{equation}\label{montab}
   m_{T;j} = \prod_{p=1}^j \ffbox{i_p}_{n-\ell-2p+2j+2} \times
   \prod_{p=j+1}^\ell \ffbox{i_p}_{\ell+1-2p+2j}
   \qquad\text{for $0\le j\le \ell-1$},
\end{equation}
and $m_{T;j+\ell} = \tau_{n+1, -\ell \delta}(m_{T;j})$. Note that $M_j = m_{T;j}$ with $T = (1,2, \dots, \ell)$.

By Theorem \ref{monrzt}, $\mathcal M(M_0) $ and $\mathcal B(\varpi_\ell)$ are isomorphic as $P$-crystals. Furthermore

\begin{prop}\cite{hernandez_level_2006}
\begin{enumerate}
\item[(i)] $\mathcal M_{I_0}(M_j)$ consists of $m_{T;j}$ for various
sequences $T$.

\item[(ii)] We have the equality of $I_0$-crystals
\begin{equation}\label{eqcrys}
\mathcal{M}(e^{\varpi_\ell}Y_{\ell, 0} Y_{0, d_\ell}^{-1}) = \bigsqcup_{k \in \ZZ} (\tau_{n+1, -\ell \delta})^{k} \left( \bigsqcup_{j = 0}^{\ell-1} \mathcal{M}_{I_0}(M_j) \right).
\end{equation}

\item[(iii)] The map $\sigma : \mathcal{M}(e^{\varpi_\ell}Y_{\ell, 0} Y_{0, d_\ell}^{-1}) \rightarrow \mathcal{M}(e^{\varpi_\ell}Y_{\ell, 0} Y_{0, d_\ell}^{-1})$ defined by $\sigma(m_{T;j}) = m_{T;j+1}$ is a
$P_{\mathrm{cl}}$-crystal automorphism equal to $z_\ell^{-1}$.

\item[(iv)] The Kashiwara operators $\tilde{e}_i$, $\tilde{f}_i$ are described in terms of tableaux: for $i\neq 0$ we have
$\tilde{e}_i \cdot m_{T;j} = m_{T';j}$ or $0$. Here $T'$ is obtained from
$T$ by replacing $i+1$ by $i$. If it is not possible (i.e. when we
have both $i+1$ and $i$ in $T$ or when $i+1$ does not occur in $T$), then it is zero. Similarly
$\tilde{f}_i \cdot m_{T;j}= m_{T'';j}$ or $0$, where $T''$ is given by replacing
$i$ by $i+1$. For the action of $\tilde{e}_0$, $\tilde{f}_0$, we have
\begin{equation*}
\begin{split}
\tilde{e}_0 \cdot m_{T;j} &=
        \begin{cases}
        0 &\text{ if $i_1\neq 1$ or $i_{\ell}=n+1$},
                \\  m_{(i_2,\dots,i_{\ell},n+1);j-1}&\text{ if $i_1=1$ and $i_{\ell}\neq n+1$,}
  \end{cases}
\\
\tilde{f}_0 \cdot m_{T;j} &=
        \begin{cases}
        0 &\text{ if $i_1 = 1$ or $i_{\ell} \neq n+1$},
                \\  m_{(1,i_1,\dots,i_{\ell-1});j+1}&\text{ if $i_1 \neq 1$ and $i_{\ell} = n+1$.}
  \end{cases} 
\end{split}
\end{equation*}
\end{enumerate}
\end{prop}

\begin{prop}\label{propqcloz}
There is a bijection given by $\Xi^{0}$ between $\mathcal{M}_{I_0}(M_0)$ and $\mathcal{M}(V)$, where $V=V_0(\Xi^{0}(M_0))$ is the fundamental representation of $\U_q(\hat{sl}_{n+1})'$ associated to $Y_{\ell,0}$. In particular the monomial crystal $\mathcal{M}_{I_0}(M_0)$ is $I_0$-closed.
\end{prop}

\begin{proof}
By the previous description, $\mathcal{M}_{I_0}(M_0)$ consists of the monomials $m_{T;0}$ for various sequences $T$. By applying the map $\Xi^0$, they are sent to the monomials $m_T=\prod_{p=1}^\ell \ffbox{i_p}_{\ell+1-2p}$ with $T = (1 \leq i_1 < \cdots < i_\ell \leq n+1)$ and where we set $\ffbox{1}_p = Y_{1,p}$ and $\ffbox{n+1}_p = Y_{n, p+n+1}^{-1}$ for all $p \in \ZZ$. They are exactly the monomials occurring in the tableaux sum expressions of $q$-characters of fundamental modules of $\U_q(\hat{sl}_{n+1})'$ (see \cite{nakajima_$t$-analogs_2003}). So the image of $\mathcal{M}_{I_0}(M_0)$ by $\Xi^0$ is equal to $\mathcal{M}(V_0(Y_{\ell,0}))$. Further this set is $I_0$-$q$-closed by definition, and $\mathcal{M}_{I_0}(M_0)$ is also $I_0$-closed
\end{proof}

Now let us consider the monomial crystal $\mathcal{M}(e^{\varpi_\ell}Y_{\ell, 0} Y_{0, d_\ell}^{-1})$ with $1 \leq \ell \leq n$. We determine in the next proposition when $z_\ell$ has the particular form of a shift.

\begin{prop}\label{shiftzl} 
The automorphism $z_{\ell}$ of $\mathcal{M}(e^{\varpi_\ell}Y_{\ell, 0} Y_{0, d_\ell}^{-1})$ has the special form of a shift $ \tau_{p, \alpha} $ ($p \in \ZZ, \alpha \in \ZZ \delta$) if and only if $ \ell = 1, n$ or $ \ell = r+1 $. Moreover,  we have $ z_{1} = z_{n} = \tau_{-n-1, \delta} $ and $ z_{r+1} = \tau_{-2, \delta} $.
\end{prop}

\begin{proof}
Assume that $\ell \leq r+1$. We have seen that $z_\ell = \sigma^{-1}$. So it suffices to determine when $\sigma$ is a shift. We have the equality $\sigma^{\ell} = \tau_{n+1, -\ell \delta}$. Hence if $\ell = 1$, $\sigma = \tau_{n+1, -\delta}$ is a shift. Assume that $\ell = r+1$. In this case, $M_1 = \tau_{2, -\delta}(M_0) = \sigma(M_0)$. As the crystal $\mathcal{M}(M_0)$ is connected and $\sigma$ and $ \tau_{2, -\delta}$ are automorphisms of crystals, we have $\sigma = \tau_{2, -\delta}$. For the other cases, $\sigma$ is explicitly known and is not a shift. As $\psi$ and shift automorphisms commute, the result follows for $\ell > r+1$.
\end{proof}

\subsection{Affinized promotion operators and monomial crystals $ \mathcal{M}(e^{\varpi_\ell}Y_{\ell, 0} Y_{0, d_\ell}^{-1}) $}

In this section, we introduce promotion operators for the level 0 extremal fundamental weight crystals. We describe them in the monomial realizations of $\mathcal{B}(\varpi_\ell)$ ($1 \leq \ell \leq n$).

Let us begin by some definitions and properties about the promotion operators (see \cite{bandlow_uniqueness_2010, fourier_kirillov-reshetikhin_2009, schilling_combinatorial_2008, shimozono_affine_2002} and references therein for more details). In type $A_n$, the highest weight crystal $\mathcal{B}_0(\lambda)$ of highest weight $\lambda \in P_0^+$ can be realized by the semi-standard Young tableaux of shape ($\lambda$). The weight function $\wt$ is defined by the content of tableaux, i.e. $\wt(T) :=(w_1(T), \dots, w_{n+1}(T))$ where $w_i(T)$ is the number of letters $i$ occurring in the tableau $T$. It can be viewed as an element of $P_0$ in the following way: set $\epsilon_i = \Lambda_i - \Lambda_{i-1}$ for $2\leq i \leq n$, $\epsilon_1 = \Lambda_1$ and $\epsilon_{n+1} = - \epsilon_1 - \cdots -\epsilon_n$. In particular, $\alpha_i = \epsilon_i - \epsilon_{i+1}$, $\Lambda_i = \epsilon_1 + \cdots + \epsilon_i$ ($1 \leq i \leq n$), and we have $P = \ZZ \epsilon_1 + \cdots + \ZZ \epsilon_{n+1}$. Then $\wt(T)$ corresponds to the element
$$w_1(T) \epsilon_1 + \cdots + w_{n+1}(T) \epsilon_{n+1} \in P_0$$
for all Young tableau $T$.

\begin{defi}
Let $\mathcal{B}_0 = \mathcal{B}_0(\lambda)$ be a highest weight $\U_q(sl_{n+1})$-crystal of highest weight $\lambda \in P_0^+$. A promotion operator $\mathrm{pr}$ on $\mathcal{B}_0$ is an operator $\mathrm{pr} : \mathcal{B}_0 \rightarrow \mathcal{B}_0$ such that
\begin{enumerate}
\item[(i)] $\mathrm{pr}$ shifts the content:  if $\wt(T) = (w_1, \dots, w_{n+1})$ is the content of $T \in \mathcal{B}_0$, then $\wt(\mathrm{pr}(T)) = (w_{n+1}, w_1, \dots, w_{n})$,
\item[(ii)] promotion operator has order $n+1$ : $\mathrm{pr}^{n+1} = \mathrm{id}$,
\item[(iii)] $\mathrm{pr} \circ \tilde{e}_i = \tilde{e}_{i+1} \circ \mathrm{pr}$ and $\mathrm{pr} \circ \tilde{f}_i = \tilde{f}_{i+1} \circ \mathrm{pr}$ for $i \in \{1,2, \dots, n-1\}$.
\end{enumerate}
\end{defi}

Given a promotion operator $\mathrm{pr}$ on a highest weight $\U_q(sl_{n+1})$-crystal $\mathcal{B}_0(\lambda)$ ($\lambda \in P_0^+$), one defines an associated affine $P_{\cl}$-crystal by setting
$$\tilde{e}_0 := \mathrm{pr}^{-1} \circ \tilde{e}_1 \circ \mathrm{pr} \text{ and } \tilde{f}_0 := \mathrm{pr}^{-1} \circ \tilde{f}_1 \circ \mathrm{pr}.$$
We denote the $P_{\cl}$-crystal hence obtained by $\mathcal{B}_0(\lambda)_{\mathrm{aff}}'$.

It was shown in \cite{shimozono_affine_2002} that the $\U_q(sl_{n+1})$-crystal $\mathcal{B}_0(\lambda)$ ($\lambda \in P_0$) has a unique promotion operator $\mathrm{pr}$ when $\lambda$ is rectangular (i.e. of the form $k \Lambda_\ell$ with $\ell \in I_0$ and $k\in \NN^{\ast}$), given by the Sch\"utzenberger jeu-de-taquin process. Furthermore the affine $P_{\cl}$-crystal  $\mathcal{B}_0(k \Lambda_\ell)_{\mathrm{aff}}'$ obtained by using the promotion operator $\mathrm{pr}$ is isomorphic to the crystal basis of a Kirillov-Reshetikhin module associated to $\ell \in I_0, k\in \NN^{\ast}$ (for a special choice of the spectral parameter $a \in \CC^{\ast}$ - see \cite{kang_perfect_1992}).

From the affine $P_{\cl}$-crystal $\mathcal{B}_0(k \Lambda_\ell)_{\mathrm{aff}}'$, let us consider its affinization $\mathcal{B}_0(k \Lambda_\ell)_{\mathrm{aff}}$ (see also \cite{kashiwara_level-zero_2002}): this is the $P$-crystal with vertices in  $\{ z^{s}T \vert s \in \ZZ, T \in \mathcal{B}_0(k \Lambda_\ell)_{\mathrm{aff}}' \}$ such that for all $s \in \ZZ$ and $T \in \mathcal{B}_0(k \Lambda_\ell)_{\mathrm{aff}}'$,
\begin{eqnarray*}
\wt(z^{s}T) = \wt(T) + s \delta, \;
\tilde{e}_i \cdot z^{s} T =  z^{s+\delta_{i, 0}} (\tilde{e}_i  \cdot T), \; 
\tilde{f}_i \cdot z^{s} T &=&  z^{s-\delta_{i, 0}} (\tilde{f}_i  \cdot T).
\end{eqnarray*}

Assume in the following that $\ell \leq r+1$ (the case $\ell > r+1$ is studied at the end of this section). We introduce the affinized promotion operator on $\mathcal{B}_0(k \Lambda_\ell)_{\mathrm{aff}}$.

\begin{defi}
Let us consider the crystal of finite type $\mathcal{B}_0(k \Lambda_\ell)$ (with $k \in \NN $ and $\ell \leq r+1 $), $\mathrm{pr}$ its associated promotion operator and $\mathcal{B}_0(k \Lambda_\ell)_{\mathrm{aff}}$ its affinization. The affinized promotion operator on $\mathcal{B}_0(k \Lambda_\ell)_{\mathrm{aff}}$ is the operator $\mathrm{pr}_{\mathrm{aff}} : \mathcal{B}_0(k \Lambda_\ell)_{\mathrm{aff}} \rightarrow \mathcal{B}_0(k \Lambda_\ell)_{\mathrm{aff}}$ such that for all $T \in \mathcal{B}_0(k \Lambda_\ell)_{\mathrm{aff}}'$ and $s\in\ZZ$,
\begin{equation*}
\mathrm{pr}_{\mathrm{aff}}(z^{s}T) = z^{s-w_{n+1}(T)} \mathrm{pr}(T).
\end{equation*}
\end{defi}

One checks easily the following statements.

\begin{lem}
The affinized promotion operator $\mathrm{pr}_{\mathrm{aff}}$ of $\mathcal{B}_0(k\Lambda_\ell)_{\mathrm{aff}}$ shifts the content. It satisfies $$\mathrm{pr}_{\mathrm{aff}} \circ \tilde{e}_i = \tilde{e}_{i+1} \circ \mathrm{pr}_{\mathrm{aff}} \text{ and } \mathrm{pr}_{\mathrm{aff}} \circ \tilde{f}_i = \tilde{f}_{i+1} \circ \mathrm{pr}_{\mathrm{aff}}$$
for $i \in \{0,1, \dots, n\}$ (where $\tilde{e}_{n+1}, \tilde{f}_{n+1}$ are understood to be $\te_{0},\tf_{0}$ respectively). It has infinite order, the weight of $\mathrm{pr}_{\mathrm{aff}}^{n+1}$ being $-k \ell \delta$.
\end{lem}

Recall that one has defined an automorphism $ \theta $ of the Dynkin diagram of type $ A_{n}^{(1)} $ corresponding to a rotation such that $ \theta(i) = i + 1 $ ($i \in I$). Then by the above Lemma, $\mathrm{pr}_{\mathrm{aff}}$ is a $\theta$-twisted automorphism of $\mathcal{B}_0(k\Lambda_\ell)_{\mathrm{aff}}$. Furthermore as the $P$-crystals $\mathcal{B}(\varpi_\ell)$ and $\mathcal{B}_0(\Lambda_\ell)_{\mathrm{aff}}$ are isomorphic (see \cite{kashiwara_level-zero_2002}), the affinized promotion operator \linebreak $\mathrm{pr}_{\mathrm{aff}} : \mathcal{B}_0(\Lambda_\ell)_{\mathrm{aff}} \rightarrow \mathcal{B}_0(\Lambda_\ell)_{\mathrm{aff}}$ induces a $\theta$-twisted automorphism of the level 0 fundamental extremal weight crystal $\mathcal{B}(\varpi_\ell)$ ($\ell \leq r+1$). We call it promotion operator of $\mathcal{B}(\varpi_\ell)$, also denoted by $\mathrm{pr}_{\mathrm{aff}}$.

We want to describe the promotion operators of the crystals $\mathcal{B}(\varpi_\ell)$ in the monomial realizations when $\ell \leq r+1$. To do that, let $ \phi : \mathcal{M}(e^{\varpi_\ell}Y_{\ell,0} Y_{0, \ell}^{-1}) \rightarrow \mathcal{M}(e^{\varpi_\ell}Y_{\ell,0} Y_{0, \ell}^{-1}) $ be the map such that
$$ \phi \left(\prod Y_{i,l}^{u_{i,l}} \right) = \prod Y_{i+1,l+1}^{u_{i,l}},$$
the terms $e^{\nu}$ being safely omitted in the definition by Remark \ref{remweight}. Denote by $$\varphi : \mathcal{B}(\varpi_\ell) \simeq \mathcal{B}_0(\Lambda_\ell)_{\mathrm{aff}} \rightarrow \mathcal{M}(e^{\varpi_\ell}Y_{\ell,0} Y_{0, \ell}^{-1})$$ the isomorphism of $P$-crystals between $\mathcal{B}_0(\Lambda_\ell)_{\mathrm{aff}}$ and $\mathcal{M}(e^{\varpi_\ell}Y_{\ell,0} Y_{0, \ell}^{-1})$. It is explicitly given by
$$\varphi : z^{s} T \in \mathcal{B}_0(\Lambda_\ell)_{\mathrm{aff}}  \mapsto m_{T ; -s} \in \mathcal{M}(e^{\varpi_\ell}Y_{\ell,0} Y_{0, \ell}^{-1}) \; (s \in \ZZ, T \in \mathcal{B}_0(\Lambda_\ell)).$$
The following result relates the map $\phi : \mathcal{M}(e^{\varpi_\ell}Y_{\ell,0} Y_{0, \ell}^{-1}) \rightarrow \mathcal{M}(e^{\varpi_\ell}Y_{\ell,0} Y_{0, \ell}^{-1})$ to the promotion operator $\mathrm{pr}_{\mathrm{aff}}$ of $\mathcal{B}(\varpi_\ell)$ introduced above.

\begin{prop}
Assume that $\ell \leq r+1$. The following diagram commutes
\begin{equation*}
\xymatrix{
\mathcal{B}(\varpi_\ell) \ar[rr]^\varphi \ar[d]_{\mathrm{pr}_\mathrm{aff}} & & \mathcal{M}(e^{\varpi_\ell}Y_{\ell,0} Y_{0, \ell}^{-1}) \ar[d]^\phi\\
 \mathcal{B}(\varpi_\ell) \ar[rr]^\varphi & & \mathcal{M}(e^{\varpi_\ell}Y_{\ell,0} Y_{0, \ell}^{-1})
}
\end{equation*}
\end{prop}

\begin{proof}
For $1\leq k \leq n+1$ and $p \in \ZZ$, we have
\begin{eqnarray*}
  \phi(\ffbox{k}_p) &=& \phi(Y_{k-1,p+k}^{-1} Y_{k,p+k-1}) = Y_{k,p+k+1}^{-1} Y_{k+1,p+k}\\
  &=& \begin{cases}
        \ffbox{k+1}_p &\text{ if $k \leq n$},
                \\ \ffbox{1}_{p+n+1} &\text{ if $k=n+1$.}
  \end{cases}
\end{eqnarray*}
\noindent Fix $j \in \ZZ$ and $T=(1 \leq i_1 < i_2 < \cdots < i_\ell \leq n+1)$ a Young tableau of shape ($\Lambda_\ell$). If $i_\ell \neq n+1$, we have
\begin{eqnarray*}
  \phi(m_{T;j}) &=& \phi \left( \prod_{p=1}^j \ffbox{i_p}_{n-\ell-2p+2j+2} \times
   \prod_{p=j+1}^\ell \ffbox{i_p}_{\ell+1-2p+2j} \right) \\
  &=& \prod_{p=1}^j \ffbox{i_p+1}_{n-\ell-2p+2j+2} \times
   \prod_{p=j+1}^\ell \ffbox{i_p+1}_{\ell+1-2p+2j}\\
   & = & m_{\mathrm{pr}(T); j} = \varphi(\mathrm{pr}_\mathrm{aff}(z^{-j}T)).
\end{eqnarray*}
\noindent Assume that $i_\ell = n+1$. Then

\renewcommand{\ffbox}[1]{
\setbox9=\hbox{$\scriptstyle\overline{1}$}
\framebox[30pt][c]{\rule{0mm}{\ht9}${\scriptstyle #1}$}
}

\begin{eqnarray*}
  \phi(m_{T;j}) &=& \prod_{p=1}^j \ffbox{i_p+1}_{n-\ell-2p+2j+2} \times
   \prod_{p=j+1}^{\ell-1} \ffbox{i_p+1}_{\ell+1-2p+2j} \times \ffbox{1}_{n-\ell + 2j + 2}\\
   & = & \prod_{p=2}^{j+1} \ffbox{i_{p-1}+1}_{n-\ell-2p+2(j+1)+2} \times
   \prod_{p=j+2}^{\ell} \ffbox{i_{p-1}+1}_{\ell+1-2p+2(j+1)} \times \ffbox{1}_{n-\ell + 2(j+1)}\\
   & = & m_{\mathrm{pr}(T);j+1} = \phi(z^{-j-1}\mathrm{pr}(T)) = \varphi(\mathrm{pr}_\mathrm{aff}(z^{-j}T)).
\end{eqnarray*}
\end{proof}

\renewcommand{\ffbox}[1]{
\setbox9=\hbox{$\scriptstyle\overline{1}$}
\framebox[20pt][c]{\rule{0mm}{\ht9}${\scriptstyle #1}$}
}

\begin{rem}
It follows in particular that $\phi$ is a $ \theta $-twisted automorphism of $ \mathcal{M}(e^{\varpi_\ell}Y_{\ell,0} Y_{0, \ell}^{-1}) $, since $\mathcal{B}(\varpi_\ell)$ and $\mathcal{M}(e^{\varpi_\ell}Y_{\ell,0} Y_{0, \ell}^{-1})$ are connected and $\mathrm{pr}_\mathrm{aff}$ is a $\theta$-twisted automorphism.
\end{rem}

The case $\ell \geq r+1$ is similar to the previous one. The affinized promotion operator of $\mathcal{B}_0(k\Lambda_\ell)_{\mathrm{aff}}$ ($k \in \NN^{\ast}$) is the operator
$$\tilde{\mathrm{pr}}_{\mathrm{aff}} : \mathcal{B}_0(k\Lambda_\ell)_{\mathrm{aff}} \longrightarrow \mathcal{B}_0(k\Lambda_\ell)_{\mathrm{aff}}$$
such that for all $T \in \mathcal{B}_0(k\Lambda_\ell)_{\mathrm{aff}}$ and $s \in \ZZ$,
$$\tilde{\mathrm{pr}}_{\mathrm{aff}}(z^sT) = z^{s+k-w_{n+1}(T)} \mathrm{pr}(T).$$
Note that the definition of the affinized promotion operator is different to the one when $\ell \leq r+1$. This provides to the automorphism $\iota_\Hlie$ of $\Hlie$.

Let us consider the map $\psi \circ \phi \circ \psi^{-1}: \mathcal{M}(e^{\varpi_\ell}Y_{\ell, 0} Y_{0, n+1-\ell}^{-1}) \rightarrow \mathcal{M}(e^{\varpi_\ell}Y_{\ell, 0} Y_{0, n+1-\ell}^{-1})$. It is a $\theta^{-1}$-twisted automorphism of $\mathcal{M}(e^{\varpi_\ell}Y_{\ell, 0} Y_{0, n+1-\ell}^{-1})$ such that
$$\psi \circ \phi \circ \psi^{-1} \left(\prod Y_{i,l}^{u_{i,l}} \right) = \prod Y_{i-1,l+1}^{u_{i,l}}.$$
It can be related to the promotion operator $\tilde{\mathrm{pr}}_{\mathrm{aff}}$ of $\mathcal{B}_0(k\Lambda_\ell)_{\mathrm{aff}}$: one can check that $\psi \circ \phi \circ \psi^{-1} = \tilde{\mathrm{pr}}_{\mathrm{aff}}^{-1}$.

\subsection{Application of promotion operators to the study of $\mathcal{M}(e^{\varpi_\ell}Y_{\ell, 0} Y_{0, d_\ell}^{-1})$}

In this section, we use promotion operators to obtain a new description of $\mathcal{M}(e^{\varpi_\ell}Y_{\ell, 0} Y_{0, d_\ell}^{-1})$, improving results given in \cite{hernandez_level_2006}. Moreover, we determine the $\ell$ for which the crystals $\mathcal{M}(e^{\varpi_\ell}Y_{\ell, 0} Y_{0, d_\ell}^{-1})$ are closed.

Assume first that $\ell \leq r+1$ (where $n= 2r+1$ is still supposed to be odd). Let us begin by the following remarks. The monomials $ \phi^{j}(Y_{\ell,0} Y_{0, \ell}^{-1}) = Y_{\ell + j, j} Y_{j, \ell + j}^{-1} $ will have a particular importance in the construction of extremal fundamental loop weight modules. One can give them in terms of Young tableaux, thanks to the $\theta$-twisted automorphism $\phi$ of $\mathcal{M}$:
\begin{itemize}
\item if $ j $ is such that $ \ell + j \leq n+1 $, $ Y_{\ell + j, j} Y_{j, \ell + j}^{-1}\in \mathcal{M}_{I_{0}}(M_{0}) $ and is equal to $ m_{T;0} $ with $ T = (j+1, j+2, \ldots , j+ \ell)  $,
\item if $ 1 \leq j \leq \ell - 1 $, then $ Y_{j, n - \ell + j + 1} Y_{n - \ell + j + 1, n + j + 1}^{-1}\in \mathcal{M}_{I_{0}}(M_{j}) $ and is equal to $m_{T;j}$ with $ T = (1, 2, \ldots, j, n-\ell+j+2, \ldots, n+1) $.
\end{itemize}

We will have to consider the finite sub-$I_j$-crystals of $\mathcal{M}(e^{\varpi_\ell}Y_{\ell, 0} Y_{0, \ell}^{-1})$
$$\mathcal{M}_{I_{j}}(Y_{\ell + j, j+k(n+1)} Y_{j, \ell + j+k(n+1)}^{-1})$$
for $j \in I$ and $k \in \ZZ$: this is the sub-$I_j$-crystal of $\mathcal{M}(e^{\varpi_\ell}Y_{\ell, 0} Y_{0, \ell}^{-1})$ generated by the monomial $Y_{\ell + j, j+k(n+1)} Y_{j, \ell + j+k(n+1)}^{-1}$. Note that one of these crystals can be obtained from another one by application of powers of $\phi$.

\begin{prop} Let $\ell \leq r+1$. We have the equality of sets
\begin{eqnarray*}
\mathcal{M}(e^{\varpi_\ell}Y_{\ell, 0} Y_{0, \ell}^{-1}) = \bigcup_{k \in \ZZ} (\tau_{n+1, -\ell \delta})^{k} \left( \bigcup_{j = 0}^{n} \mathcal{M}_{I_{j}}(Y_{\ell + j, j} Y_{j, \ell + j}^{-1}) \right).
\end{eqnarray*}
\end{prop}

\begin{proof} As $Y_{\ell + j, j} Y_{j, \ell + j}^{-1} \in \mathcal{M}(e^{\varpi_\ell}Y_{\ell, 0} Y_{0, \ell}^{-1})$ for all $0 \leq j \leq n$ and $\mathcal{M}(e^{\varpi_\ell}Y_{\ell, 0} Y_{0, \ell}^{-1})$ is connected,
$$ \bigcup_{j = 0}^{n} \mathcal{M}_{I_{j}}(Y_{\ell + j, j} Y_{j, \ell + j}^{-1}) \subset \mathcal{M}(e^{\varpi_\ell}Y_{\ell, 0} Y_{0, \ell}^{-1}) $$
as sets.

Let us fix $m \in \mathcal{M}(e^{\varpi_\ell}Y_{\ell, 0} Y_{0, \ell}^{-1})$. The monomial $m$ is of the form $m_{T;j}$ with \linebreak $ T = (1 \leq i_{1} < i_{2} < \ldots < i_{\ell} \leq n + 1) $ and $j \in \ZZ$. By application of the shift automorphism, one can assume that $ 0 \leq j \leq \ell - 1 $. So we have to show that $m_{T;j} \in \bigcup_{j = 0}^{n} \mathcal{M}_{I_{j}}(Y_{\ell + j, j} Y_{j, \ell + j}^{-1})$. If $ j = 0 $, we have $ m_{T;0} \in \mathcal{M}_{I_{0}}(Y_{\ell, 0} Y_{0, \ell}^{-1}) $. Assume that $1 \leq j \leq \ell -1$ and set $s = i_{j+1}-1$. Then $T = (i_{1}< \dots < i_{j}<s+1<i_{j+2}< \dots<i_{\ell})$ and by application of the Kashiwara operators $\tilde{e}_1, \dots, \tilde{e}_{s-1}, \tilde{e}_{s+2}, \dots, \tilde{e}_{n}$ on $m_{T;j}$, we show that
$$m_{T;j} \in \mathcal{M}_{I_s}(m_{T';j}) \text{ with } T' = (1< \dots < j <s+1< \dots < s+\ell-j).$$
By applying $\tilde{e}_1, \dots, \tilde{e}_{j-1}, \tilde{e}_{s+\ell-j+1}, \dots, \tilde{e}_{n}$ and $\tilde{e}_0$ on $m_{T';j}$, it is sent on
\begin{eqnarray*}
        m_{T'';0} \text{ with } T'' = (s+1< \dots < s+\ell) \text{ if $ s+\ell \leq n+1$,}
\end{eqnarray*}
and on
$$m_{T'';u} \text{ with } u=s+\ell-n-1, T'' = (1<\dots< u <s+1< \dots <n+1) \text{ otherwise.}$$
Furthermore $m_{T'';u} = \phi^{s}(Y_{\ell, 0} Y_{0, \ell}^{-1}) = Y_{\ell+s, s} Y_{s, \ell + s}^{-1}$ by the above remark and $m_{T;j}$ is also contained in $ \mathcal{M}_{I_s}(Y_{\ell+s, s} Y_{s, \ell + s}^{-1})$.
\end{proof}

\begin{rem}\label{remprom}
One of the questions treated in \cite[Section 4]{hernandez_level_2006} is to give an explicit description of monomials occurring in $ \mathcal{M}(e^{\varpi_\ell}Y_{\ell, 0} Y_{0, \ell}^{-1}) $ ($\ell \leq r+1$). Actually by the shift automorphism and the description given in \cite{hernandez_level_2006}, all the monomials in $ \mathcal{M}(e^{\varpi_\ell}Y_{\ell, 0} Y_{0, \ell}^{-1}) $ can be obtained from the monomials occurring in $\bigsqcup_{j = 0}^{\ell-1} \mathcal{M}_{I_0}(M_j)$ (see (\ref{eqcrys})). So this description requires to know $ \left[ \ell \cdot \begin{pmatrix} n+1 \\ \ell \end{pmatrix} \right]$ monomials to obtain all the other ones. The preceding proposition improves this result. In fact to determine all the vertices of $ \mathcal{M}(e^{\varpi_\ell}Y_{\ell, 0} Y_{0, \ell}^{-1}) $, it suffices to know the monomials occurring in the $I_{\{0,1\}}$-crystal $ \mathcal{M}_{I_{\{0,1\}}}(Y_{\ell, 0} Y_{0, \ell}^{-1}) $ and to apply $\phi$. Further a monomial $m_{T';0} \in \mathcal{M}_{I_{\{0,1\}}}(Y_{\ell, 0} Y_{0, \ell}^{-1})$ is such that $T'$ has the form $T'=(1<i_2< \cdots < i_{\ell})$. So by the above proposition, only $ \begin{pmatrix} n \\ \ell-1 \end{pmatrix}$ monomials are sufficient to determine all the vertices of $\mathcal{M}(e^{\varpi_\ell}Y_{\ell, 0} Y_{0, \ell}^{-1})$.
\end{rem}

The following lemma will be useful in the following.

\begin{lem}
Assume that $\ell = 1$ or $\ell = r+1$ and set $p = n+1$ or $p = 2$ respectively. We have the equality of $I_j$-crystals ($0 \leq j \leq n$)
\begin{equation}\label{eqcrys2}
\mathcal{M}(e^{\varpi_\ell}Y_{\ell, 0} Y_{0, \ell}^{-1}) = \bigsqcup_{k \in \ZZ} (\tau_{p, -\delta})^{k} \left(  \mathcal{M}_{I_{j}}(Y_{\ell + j, j} Y_{j, \ell + j}^{-1}) \right)
\end{equation}
\end{lem}

\begin{proof}
By Proposition \ref{shiftzl}, the automorphism $z_\ell$ has the special form of a shift in the considered cases. Further we know that $(\tau_{-p, \delta})^\ell = \tau_{-n-1, \ell \delta}$. Using (\ref{eqcrys}), we obtain
\begin{eqnarray*}
\mathcal{M}(e^{\varpi_\ell}Y_{\ell, 0} Y_{0, \ell}^{-1}) = \bigsqcup_{k \in \ZZ} (\tau_{p, -\delta})^{k \ell} \left( \bigsqcup_{j = 0}^{\ell-1} (\tau_{p, -\delta})^{j}\left( \mathcal{M}_{I_0}(M_0) \right) \right).
\end{eqnarray*}
As $\tau_{p, -\delta}$ and $\phi$ commute, (\ref{eqcrys2}) follows.
\end{proof}

Similar equalities of crystals can be given for $\ell \geq r+1$, by using the automorphism $\psi$. Now we are able to determine the $\ell \in I_0$ for which the monomial crystal $\mathcal{M}(e^{\varpi_\ell}Y_{\ell, 0} Y_{0, d_\ell}^{-1})$ is closed.

\begin{thm}\label{proclocry} The monomial crystal $ \mathcal{M}(e^{\varpi_\ell}Y_{\ell, 0} Y_{0, d_\ell}^{-1}) $ is closed if and only if $ \ell = 1, r+1 $ or $n$.\end{thm}

\begin{proof}Let us begin by the case $\ell \leq r+1$. Assume that the crystal $\mathcal{M}(e^{\varpi_\ell}Y_{\ell, 0} Y_{0, \ell}^{-1})$ is $q$-closed for $ 2 \leq \ell \leq r $. Consider the monomial
$$ M_{j} = Y_{\ell, 2j} Y_{0, n-\ell+1+2j}^{-1}Y_{j, \ell+j}^{-1}Y_{j, n-\ell+1+j} \in \mathcal{M}(e^{\varpi_\ell}Y_{\ell, 0} Y_{0, \ell}^{-1}) $$
with $j \neq 0$. We have $\Xi_j(M_j) = Y_{j, \ell+j}^{-1}Y_{j, n-\ell+1+j}$. By Definition \ref{defclocry}, there exists a subset
$$\mathcal{S}_{M_j} \subset \mathcal{M}(e^{\varpi_\ell}Y_{\ell, 0} Y_{0, \ell}^{-1}) \cap \left(M_j \cdot \prod_{l \in \ZZ} A_{j,l}^{\ZZ} \right)$$
containing $M_j$ such that its image $\Xi_{j}\left( \mathcal{S}_{M_j} \right)$ by $\Xi_j$ is the set of $\ell$-weights of a representation of $\hat{\mathcal{U}}_j$. By the theory of $q$--characters of $\hat{\mathcal{U}}_j$-modules, we should have
$$Y_{j, \ell+j}^{-1}Y_{j, n-\ell+1+j} \Xi_j \left( A_{j, \ell +j-1} \right) \in \Xi_j \left( \mathcal{S}_{M_j} \right).$$
Furthermore, the map $\Xi_j$ is injective when it is restricted on the set of monomials $M_j \cdot \left( \prod_{l \in \ZZ} A_{j,l}^{\ZZ} \right)$: indeed we have for all monomial $M_j \cdot \left( \prod_{l \in \ZZ} A_{j,l}^{u_{j,l}} \right)$,
\begin{align*}
\Xi_j \left(M_j \cdot \prod_{l \in \ZZ} A_{j,l}^{u_{j,l}} \right) = \Xi_j(M_j) \cdot \prod_{l \in \ZZ} \Xi_j(A_{j,l}^{u_{j,l}}) = \Xi_j(M_j) \cdot \prod_{l \in \ZZ} \tau_{\{j\}}(A_{j,l}^{u_{j,l}})
\end{align*}
where $\tau_{\{j\}}$ is the map defined in \cite[Definition 3.2]{frenkel_combinatorics_2001} (the second equality is a consequence of \cite[Lemma 3.5]{frenkel_combinatorics_2001}). The injectivity is a consequence of the injectivity of $\tau_{\{j\}}$ (see \cite[Lemma 3.3]{frenkel_combinatorics_2001}).

So the monomial
$$ m = M_j \cdot A_{j,\ell+j-1} = Y_{\ell, 2j} Y_{0, n-\ell+1+2j}^{-1}Y_{j, \ell+j-2} Y_{j-1, \ell+j-1}^{-1} Y_{j+1, \ell+j-1}^{-1} Y_{j, n-\ell+1+j}$$
should occur in $\mathcal{M}(e^{\varpi_\ell}Y_{\ell, 0} Y_{0, \ell}^{-1})$. But this is not the case, $m$ being not of the form (\ref{montab}). Hence $\mathcal{M}(e^{\varpi_\ell}Y_{\ell, 0} Y_{0, \ell}^{-1})$ is not $q$-closed when $2 \leq \ell \leq r$.

Now assume that $ \ell = 1 $ or $\ell = r+1$. In this case, $ z_{\ell} = \tau_{-p, \delta} $ with $p=n+1$ or $p=2$ respectively and by the above lemma
$$\mathcal{M}(e^{\varpi_\ell}Y_{\ell, 0} Y_{0, \ell}^{-1}) = \bigsqcup_{k \in \ZZ} (\tau_{p, -\delta})^{k} \left(  \mathcal{M}_{I_{j}}(Y_{\ell + j, j} Y_{j, \ell + j}^{-1}) \right)$$
as $I_j$-crystals  ($ 0 \leq j \leq n $). By Proposition \ref{propqcloz}, the finite crystal $\mathcal{M}_{I_{0}}(M_0)$ is $I_0$-closed. As the $I_j$-crystals $\mathcal{M}_{I_{j}}(Y_{\ell + j, j} Y_{j, \ell + j}^{-1})$ can be obtained from $\mathcal{M}_{I_{0}}(M_0)$ by application of powers of $\phi$, they are also $I_j$-closed for all $0 \leq j \leq n$. Then by the above equalities $ \mathcal{M}(e^{\varpi_\ell}Y_{\ell, 0} Y_{0, \ell}^{-1}) $ is closed.

Finally, the result follows for all the $\ell \in I_0$ by using the $\iota$-twisted automorphism $\psi$ (which preserves the notion of $q$--closed monomial set).
\end{proof}

\section{Extremal fundamental loop weight modules for $\U_q(sl_{n+1}^{tor})$ when $ \mathcal{M}(e^{\varpi_\ell}Y_{\ell, 0} Y_{0, d_\ell}^{-1}) $ is closed}

Assume that $n = 2r+1$ ($r \geq 1$) is odd and $\mathcal{M}_\ell = \mathcal{M}(e^{\varpi_\ell}Y_{\ell, 0} Y_{0, d_\ell}^{-1}) $ is closed (it holds if and only if $\ell=1, r+1$ or $n$). In this section, we relate the monomial $\U_q(\hat{sl}_{n+1})$-crystals $ \mathcal{M}_\ell$ with integrable representations of $\U_q(sl_{n+1}^{tor})$.

In the first part, we construct a new infinite family of representations $V_\ell$ of $\U_{q}(sl_{n+1}^{tor})$ (Theorem \ref{thmmod}). We call these representations the extremal fundamental loop weight modules. Let us give the outline of this construction: consider the vector space $V_\ell$ freely generated by the monomials occurring in $\mathcal{M}_\ell$. For all $0 \leq j \leq n$, we define an action of $\U_q^{v,j}(sl_{n+1}^{tor})$ on it, denoted by $V_\ell^{(j)}$, such that 
$$V_\ell^{(j)} = \bigoplus_{k \in \ZZ} V_{k}^{(j)}$$
where $V_{k}^{(j)}$ is a subvector space endowed with a structure of a simple $\ell$-highest weight $\U_q^{v,j}(sl_{n+1}^{tor})$-module. We show that it defines a $\U_{q}(sl_{n+1}^{tor})$-module structure in this way on $V_\ell$, the compatibility between the action of various vertical subalgebras being a consequence of the existence of promotion operators on $\mathcal{M}_\ell$. Furthermore the $q$--character of $V_\ell$ is the sum of monomials occurring in $ \mathcal{M}_\ell $ with multiplicity one.

In the second part, we study these representations: we show that $V_\ell$ is irreducible and it is an extremal loop weight module, generated by an extremal vector of $\ell$-weight $e^{\varpi_\ell} Y_{\ell,0}Y_{0,d_\ell}^{-1}$. Furthermore explicit formulas are given for the action of $\U_{q}(sl_{n+1}^{tor})$ on $V_\ell$. It is remarkable that these formulas are expressed only from the associated monomial crystal and are ``universal'' in the following sense: the action on all the extremal fundamental loop weight modules $V_\ell$ is completely determined by these formulas and by the data of the corresponding monomial crystals $\mathcal{M}_\ell$.  This sheds new light on the link between monomial crystals and the theory of $q$--characters already expected in \cite{hernandez_level_2006}. All these sentences hold for the fundamental $\ell$-highest weight modules $V_0(Y_{\ell, 0})$ of $\U_{q}(\hat{sl}_{n+1})'$ with the corresponding monomial crystals $\mathcal{M}_0(Y_{\ell, 0})$.

In the third part, we specialize $q$ at a root of unity $\epsilon$. We obtain new irreducible finite-dimensional representations of the quantum toroidal algebra $\U_{\epsilon}(sl_{n+1}^{tor})'$.

\subsection{Construction of the extremal fundamental loop weight modules} Let us begin by the main result of this section.

\begin{thm}\label{thmmod}
Assume that $n=2r+1$ is odd and $\ell = 1, r+1$ or $n$. There exists a thin representation of $ \U_{q}(sl_{n+1}^{tor}) $ whose $q$--character is the sum of all monomials occurring in $ \mathcal{M}(e^{\varpi_\ell}Y_{\ell, 0} Y_{0, d_\ell}^{-1}) $ with multiplicity one. It is denoted by $V_\ell = V(e^{\varpi_\ell}Y_{\ell, 0}Y_{0,d_\ell}^{-1})$ and called the extremal fundamental loop weight module of $\ell$-weight $e^{\varpi_\ell}Y_{\ell, 0}Y_{0,d_\ell}^{-1}$.
\end{thm}

To construct these representations, let us start with results about the fundamental modules $V_0(Y_{\ell, k})$ of $\U_q(\hat{sl}_{n+1})'$ ($n \in \NN^{\ast}$, $1 \leq \ell \leq n$, $k \in \ZZ$). As it is said above, it is isomorphic to the fundamental highest weight $\U_q(sl_{n+1})$-module $V_0(\Lambda_{\ell})$. So we begin by recalling some well-known facts about $V_0(\Lambda_{\ell})$ which will be useful.

\begin{lem}
All the weight spaces of the fundamental highest weight $\U_q(sl_{n+1})$-module $V_0(\Lambda_{\ell})$ ($1 \leq \ell \leq n$) are of dimension 1. Furthermore the Weyl group of finite type $W_0$ acts transitively on $\wt(V_0(\Lambda_{\ell}))$.
\end{lem}

%

\begin{prop}\label{actfun}
Let $V_0(Y_{\ell, k})$ be a fundamental module of $\U_q(\hat{sl}_{n+1})'$ ($\ell \in I_0, k \in \ZZ$). Then $V_0(Y_{\ell, k})$ is a thin $\U_q(\hat{sl}_{n+1})'$-module which admits a basis $(v_m)$ indexed by the vertices of the monomial crystal $\mathcal{M}_0(Y_{\ell, k})$, such that for all $m\in \mathcal{M}_0(Y_{\ell,k})$ and $i \in I_0$, $v_m$ is of $\ell$-weight $m$ and
$$x_{i,0}^{+} \cdot v_m = v_{\tilde{e}_i \cdot m}, \; x_{i,0}^{-} \cdot v_m = v_{\tilde{f}_i \cdot m}$$
where $v_0 = 0$ by convention.
\end{prop}

\begin{proof}
It is known that $\mathrm{Res}(V_0(Y_{\ell,k}))$ is the fundamental highest weight $\U_q(sl_{n+1})$-module $V_0(\Lambda_\ell)$. By the preceding Lemma, its weight spaces are all of dimension one. In particular its $\ell$-weight spaces are also of dimension one and $V_0(Y_{\ell,k})$ is a thin $\U_q(\hat{sl}_{n+1})'$-module.

Furthermore $\mathrm{Res}(V_0(Y_{\ell,k}))$ is the extremal weight module of extremal weight $\Lambda_\ell$, generated by an extremal vector $v$ of weight $\Lambda_\ell$. Hence, there exists $\{v_w\}_{w \in W_0}$ such that $v_{\text{Id}}=v$ and
$$x_{i,0}^{\pm} \cdot v_w=0 \text{ and }(x_{i,0}^{\mp})^{(\pm w(\Lambda_\ell)(h_i))} \cdot v_w=v_{s_i(w)} \text{ if } \pm w(\Lambda_\ell)(h_i)\geq 0.$$
By the above lemma for all $\nu \in \wt(\mathrm{Res}(V_0(Y_{\ell,k})))$, there exists $w$ such that $\nu = w(\Lambda_\ell)$. Then the corresponding vector $v_w$ is non zero of weight $\nu$. As all the weight spaces of $V_0(Y_{\ell,k})$ are of dimension one, $\{v_w\}_{w \in W_0}$ generates $V_0(Y_{\ell,k})$ as a vector space. Furthermore for all $w, w' \in W_0$,
$$w(\Lambda_\ell) = w'(\Lambda_\ell) \Leftrightarrow v_w = v_{w'}.$$
In fact, we have (see \cite[Ch. V.3.3, Proposition 2]{bourbaki_groupes_1968})
\begin{align*}
w(\Lambda_\ell) = w'(\Lambda_\ell) & \Leftrightarrow w^{-1} w'(\Lambda_\ell) = \Lambda_\ell \Leftrightarrow w^{-1} w' \in \langle s_i, i \in I_0-\{\ell\}\rangle \\
 & \Leftrightarrow w' \in w \cdot \langle s_i, i \in I_0-\{\ell\}\rangle.
\end{align*}
Fix an $\ell$-weight $m \in \mathcal{M}(V_0(Y_{\ell, k})) = \mathcal{M}_0(Y_{\ell, k})$. By that we have said above, one can define $v_m$ as the unique vector $v_w$ ($w \in W_0$) such that $w(\Lambda_\ell) = \wt(m)$. Then $\{v_m \vert m \in \mathcal{M}_0(Y_{\ell, k})\}$ is a basis of $V_0(Y_{\ell, k})$. Furthermore as the weight subspaces and the $\ell$-weight subspaces of $V_0(Y_{\ell, k})$ coincide and are of dimension one, $v_m$ is also an $\ell$-weight vector of $\ell$-weight $m$ for all $m \in \mathcal{M}_0(Y_{\ell, k})$.

We determine the action of $\U_q^{h}(\hat{sl}_{n+1})$ on this basis. Fix $m\in \mathcal{M}_{I_0}(Y_{\ell,k})$. For all $i \in I_0$, we have $\wt(m)(h_i) = 0, \pm 1$. Assume that $\wt(m)(h_i) = 0$. Then on the one hand $x_{i,0}^{\pm} \cdot v_m = 0$ by definition of the family $\{v_w\}_{w \in W_0}$. And on the other hand $\tilde{e}_i \cdot m = 0 $ and $ \tilde{f}_i \cdot m = 0$ by the description of the crystal $\mathcal{M}_0(Y_{\ell,k})$ recalled above. Now assume that $\wt(m)(h_i) = \pm 1$. The vector $S_i(v_m) = x_{i,0}^{\mp} \cdot v_m$ is of the form $v_{m'}$ with $m' \in \mathcal{M}_0(Y_{\ell,k})$ such that 
$$\wt(m') = s_i(\wt(m)) = \wt(m) \mp \alpha_i.$$
But the description of $\mathcal{M}_0(Y_{\ell,k})$ shows that the unique monomial of weight $\wt(m) \mp \alpha_i$ is $\tilde{f}_i \cdot m$ (resp. $\tilde{e}_i \cdot m$). Hence $m' $ is equal to $ \tilde{f}_i \cdot m$ (resp. $\tilde{e}_i \cdot m$). Finally we have shown that for all $i \in I_0$ and $m \in \mathcal{M}_0(Y_{\ell,k})$,
$$x_{i,0}^{+}\cdot v_m = v_{\tilde{e}_i \cdot m} \text{ and } x_{i,0}^{-}\cdot v_m = v_{\tilde{f}_i \cdot m}.$$
\end{proof}

In particular, the action of $\U_q(\hat{sl}_{n+1})'$ on the fundamental modules $V_0(Y_{\ell, k})$ is determined by the combinatorics of monomial crystals $\mathcal{M}_0(Y_{\ell, k})$: in fact, the action of operators $x_{i,r}^{\pm}$ ($1 \leq i \leq n$, $r \in \ZZ$) deduces from the action of the $x_{i,0}^{\pm}$ (given by $\mathcal{M}_0(Y_{\ell, k})$) and the action of $h_{i,r}$ (given by the $\ell$-weights $m \in \mathcal{M}_0(Y_{\ell, k})$) from (\ref{actcartld}).\\

Let us begin the construction of extremal fundamental loop weight modules. Assume that $n = 2r+1$ is odd and $\ell \leq r+1$ (the case $\ell > r+1$ is discussed below at Remark \ref{remellplus}). Consider the monomial $\U_q(\hat{sl}_{n+1})$-crystal $ \mathcal{M}(e^{\varpi_\ell}Y_{\ell, 0} Y_{0, \ell}^{-1}) $, supposed to be closed. This is the case if and only if $\ell = 1$ or $\ell= r+1$ (Theorem \ref{proclocry}). Set $p = n+1$ or $p = 2$ respectively.

Denote by $ \mathcal{E} $ (resp. $ \mathcal{E}_{j,k} $ for $ 0 \leq j \leq n $ and $ k \in \ZZ $) the set of monomials occurring in $ \mathcal{M}(e^{\varpi_\ell}Y_{\ell, 0} Y_{0, \ell}^{-1}) $ (resp. in $ (\tau_{p, -\delta})^{k} \left(  \mathcal{M}_{I_{j}}(Y_{\ell + j, j} Y_{j, \ell + j}^{-1}) \right) = \mathcal{M}_{I_{j}}(Y_{\ell + j, j + kp} Y_{j, \ell + j + kp}^{-1}) $). By (\ref{eqcrys2}), one has $ \mathcal{E} = \bigsqcup_{k \in \ZZ} \mathcal{E}_{j,k} $ for all $ 0 \leq j \leq n $.

Let
\begin{equation}\label{extrvectsp}
V(e^{\varpi_\ell}Y_{\ell, 0} Y_{0, \ell}^{-1}) = \bigoplus_{m \in \mathcal{E}} \CC v_{m}
\end{equation}
be the vector space freely generated by elements of $ \mathcal{E} $. For all $ 0 \leq j \leq n $ and $ k \in \ZZ $ set $ V_{k}^{(j)} = \bigoplus_{m \in \mathcal{E}_{j,k}} \CC v_{m} $ the subspace of $V(e^{\varpi_\ell}Y_{\ell, 0} Y_{0, \ell}^{-1})$ of dimension $ \dim(V_0(\Lambda_{\ell})) $. In particular, we have
$$V(e^{\varpi_\ell}Y_{\ell, 0} Y_{0, \ell}^{-1}) = \bigoplus_{k \in \ZZ} V_{k}^{(j)}.$$
This decomposition can be compared to the equalities of crystals (\ref{eqcrys2}).

We endow the vector space $V(e^{\varpi_\ell}Y_{\ell, 0} Y_{0, \ell}^{-1})$ with a structure of $ \U_{q}^{v,j}(sl_{n+1}^{tor}) $-module as follows ($0 \leq j \leq n$): for all $k \in \ZZ$, let $(v_m)$ be the basis of  the $ \U_{q}^{v,j}(sl_{n+1}^{tor}) $-module $ V_0(Y_{\ell, j+kp})^{(j)} $ defined in Proposition \ref{actfun}, indexed by the set of monomials
$$\Xi^j \left( \mathcal{M}_{I_j}(Y_{\ell + j, j + kp}Y_{j, \ell + j + kp}^{-1}) \right) = \left\lbrace \Xi^j(m) \vert m \in \mathcal{M}_{I_j}(Y_{\ell + j, j + kp}Y_{j, \ell + j + kp}^{-1}) \right\rbrace.$$
Let us define an isomorphism of vector spaces between $V_{k}^{(j)}$ and $V_0(Y_{\ell, j+kp})^{(j)}$ by
\begin{eqnarray*}
V_{k}^{(j)} & \longrightarrow & V_0(Y_{\ell, j+kp})^{(j)} \\
v_m & \mapsto & v_{\Xi^{j}(m)}
\end{eqnarray*}
We endow the vector space $V_{k}^{(j)}$ with a structure of $ \U_{q}^{v,j}(sl_{n+1}^{tor}) $-module by pulling back the action of $ \U_{q}^{v,j}(sl_{n+1}^{tor}) $ on $V_0(Y_{\ell, j+kp})^{(j)}$. By direct sum, $V(e^{\varpi_\ell}Y_{\ell, 0} Y_{0, \ell}^{-1})$ is a $ \U_{q}^{v,j}(sl_{n+1}^{tor}) $-module, denoted by $V(e^{\varpi_\ell}Y_{\ell, 0} Y_{0, \ell}^{-1})^{(j)}$.

\begin{prop}
There exists a structure of $ \U_{q}(sl_{n+1}^{tor}) $-module on $V(e^{\varpi_\ell}Y_{\ell, 0} Y_{0, \ell}^{-1})$ such that for all $j\in I$ the induced $ \U_{q}^{v,j}(sl_{n+1}^{tor}) $-module is isomorphic to $V(e^{\varpi_\ell}Y_{\ell, 0} Y_{0, \ell}^{-1})^{(j)}$. Furthermore the $q$--character of $V(e^{\varpi_\ell}Y_{\ell, 0} Y_{0, \ell}^{-1})$ is
$$\chi_q(V(e^{\varpi_\ell}Y_{\ell, 0} Y_{0, \ell}^{-1})) = \sum_{m \in \mathcal{E}} m,$$
where $\mathcal{E}$ is the set of the monomials occurring in $\mathcal{M}(e^{\varpi_\ell}Y_{\ell, 0} Y_{0, \ell}^{-1})$.
\end{prop}

\begin{proof}
To define an action of $\U_{q}(sl_{n+1}^{tor})$ on $V(e^{\varpi_\ell}Y_{\ell, 0} Y_{0, \ell}^{-1})$, we determine the action of the subalgebras $\hat{\U}_i$ for all $i \in I$. For that, let $j \in I$ be such that $j \neq i$. The action of $\hat{\U}_i$ on $V(e^{\varpi_\ell}Y_{\ell, 0} Y_{0, \ell}^{-1})$ is the restriction of the action of $ \U_{q}^{v,j}(sl_{n+1}^{tor}) $ on $V(e^{\varpi_\ell}Y_{\ell, 0} Y_{0, \ell}^{-1})^{(j)}$. Furthermore we set for all $h \in \Hlie$ and $m \in \mathcal{M}(e^{\varpi_\ell}Y_{\ell, 0} Y_{0, \ell}^{-1})$, $$k_h \cdot v_m = q^{\wt(m)(h)} v_m.$$

The definition of the action of $\hat{\U}_i$ ($i \in I$) is independent of the choice of $j \in I$, $j \neq i$: for $m \in \mathcal{E}$, the action of $\U_{q}^{v,j}(sl_{n+1}^{tor})$ on the vector $v_m$ is determined by the sub-$I_{j}$-crystal $\mathcal{M}_{I_{j}}(m)$ of $\mathcal{M}(e^{\varpi_\ell}Y_{\ell, 0} Y_{0, \ell}^{-1})$ and by the $\ell$-weight $\Xi^{j}(m)$. So the action of $\hat{\U}_{i}$ on $v_m$ is determined by the action of $\tilde{e}_i$ and $\tilde{f}_i$ on $m$ and by the $\ell$-weight $\Xi_i(m)$ which are independent of the choice of $j$.

Let us show that this action endows $V(e^{\varpi_\ell}Y_{\ell, 0} Y_{0, \ell}^{-1})$ with a structure of $ \U_{q}(sl_{n+1}^{tor}) $-module. We fix two indices $i_1, i_2 \in I$ and we check the relations satisfied by $\hat{\U}_{i_1}$ and $\hat{\U}_{i_2}$. The indices $i_1$ and $i_2$ are in the same connected subset $I_j$ of the set of vertices of the Dynkin diagram ($j \in I$). By construction, the action of $\hat{\U}_{i_1}$ and $\hat{\U}_{i_2}$ are restrictions of the action of $\U_q^{v,j}(sl_{n+1}^{tor})$ on $V(e^{\varpi_\ell}Y_{\ell, 0} Y_{0, \ell}^{-1})$. As $V(e^{\varpi_\ell}Y_{\ell, 0} Y_{0, \ell}^{-1})^{(j)}$ is a $\U_q^{v,j}(sl_{n+1}^{tor})$-module, the relations between $\hat{\U}_{i_1}$ and $\hat{\U}_{i_2}$ are satisfied and $V(e^{\varpi_\ell}Y_{\ell, 0} Y_{0, \ell}^{-1})$ is a $\U_q(sl_{n+1}^{tor})$-module.

By construction, the induced $ \U_{q}^{v,j}(sl_{n+1}^{tor}) $-module obtained from $V(e^{\varpi_\ell}Y_{\ell, 0} Y_{0, \ell}^{-1})$ by restriction is isomorphic to $V(e^{\varpi_\ell}Y_{\ell, 0} Y_{0, \ell}^{-1})^{(j)}$ for all $j \in I$. Furthermore the $\ell$-weight of $v_m$ is $\Xi_i(m)$ for the action of $\hat{\U}_i$ ($i \in I$). So $m$ is the $\ell$-weight of  $v_m$ and the $q$--character of $V(e^{\varpi_\ell}Y_{\ell, 0} Y_{0, \ell}^{-1})$ is the sum of monomials occurring in $\mathcal{M}(e^{\varpi_\ell}Y_{\ell, 0} Y_{0, \ell}^{-1})$ with multiplicity one.
\end{proof}

\begin{rem}\label{remellplus}
Let us consider the case $\ell > r+1$. The monomial crystal \linebreak $\mathcal{M}(e^{\varpi_\ell}Y_{\ell, 0} Y_{0, n+1-\ell}^{-1})$ is closed for $\ell = n$. We show in the same way as above that there exists also a $\U_q(sl_{n+1}^{tor})$-module $V(e^{\varpi_n}Y_{n, 0} Y_{0, 1}^{-1})$ whose $q$--character is the sum of monomials occurring in $\mathcal{M}(e^{\varpi_n}Y_{n, 0} Y_{0, 1}^{-1})$ with multiplicity one.

Actually this $\U_q(sl_{n+1}^{tor})$-module is related to the previous one for $\ell = 1$ as follows. We have defined an automorphism $\iota$ of the Dynkin diagram of type $A_{n}^{(1)}$. It induces an algebra automorphism of $\U_q(sl_{n+1}^{tor})$ we still denote $\iota$, which sends $x_{i,r}^{\pm}, h_{i,m}, k_h$ to $x_{\iota(i),r}^{\pm}, h_{\iota(i), m}, k_{\iota_\Hlie(h)}$ ($i \in I, r\in \ZZ, m \in \ZZ-\{0\}, h \in \Hlie$). Let us denote by $V(e^{\varpi_n}Y_{n, 0} Y_{0, 1}^{-1})^{\iota}$ the $\U_q(sl_{n+1}^{tor})$-module obtained from $V(e^{\varpi_n}Y_{n, 0} Y_{0, 1}^{-1})$ by twisting the action by $\iota$. Then we show easily that $V(e^{\varpi_n}Y_{n, 0} Y_{0, 1}^{-1})^{\iota}$ and $V(e^{\varpi_{1}}Y_{1, 0} Y_{0, 1}^{-1})$ are isomorphic.
\end{rem}

\subsection{Study of the extremal fundamental loop weight modules}

In this section, we study the $\U_q(sl_{n+1}^{tor})$-modules $V(e^{\varpi_\ell}Y_{\ell, 0}Y_{0,d_\ell}^{-1})$, where $n=2r+1$ is supposed to be odd and $\ell = 1, n$ or $\ell = r+1$. We set $p = n+1$ or $p=2$ respectively.

\begin{prop}
The $\U_q(sl_{n+1}^{tor})$-module $V(e^{\varpi_\ell}Y_{\ell, 0}Y_{0,d_\ell}^{-1})$ is integrable. Moreover, it satisfies properties $(iii)$ and $(iv)$ of Remark \ref{remintmod} with weight subspaces of dimension one.
\end{prop}

\begin{proof}
Assume first that $\ell = 1, r+1$. The $q$--character of $V(e^{\varpi_\ell}Y_{\ell, 0}Y_{0,\ell}^{-1})$ is known: this is the sum of monomials occurring in $\mathcal{M}(e^{\varpi_\ell}Y_{\ell, 0}Y_{0,\ell}^{-1})$ with multiplicity one. Furthermore one has the equality of $I_0$-crystals
\begin{equation*}
\mathcal{M}(e^{\varpi_\ell}Y_{\ell, 0} Y_{0, \ell}^{-1}) = \bigsqcup_{k \in \ZZ} (\tau_{p, -\delta})^{k} \left(  \mathcal{M}_{I_{0}}(e^{\varpi_\ell}Y_{\ell, 0} Y_{0, \ell}^{-1}) \right).
\end{equation*}
For all $m \in \mathcal{M}_{I_{0}}(e^{\varpi_\ell}Y_{\ell, 0} Y_{0, \ell}^{-1})$ and $k \in \ZZ$, $\wt((\tau_{p, -\delta})^{k}(m)) = \wt(m) - k \delta$. So to prove that the weight spaces of $V(e^{\varpi_\ell}Y_{\ell, 0}Y_{0,\ell}^{-1})$ are of dimension one, we have to show that the weights of monomials occurring in $\mathcal{M}_{I_{0}}(e^{\varpi_\ell}Y_{\ell, 0} Y_{0, \ell}^{-1})$ are different to each other. More precisely, it is sufficient to show that the sum
$$\sum_{m \in \mathcal{M}_{I_{0}}(e^{\varpi_\ell}Y_{\ell, 0} Y_{0, \ell}^{-1})} e\left(\wt(\Xi^{0}(m))\right) \in \bigoplus_{\nu \in P_0} \ZZ e(\nu)$$
is without multiplicity. This follows from the above results: it is the character of the $\U_q(sl_{n+1})$-module $V_0(\Lambda_{\ell})$.

For all $j \in I$, the representation $V(e^{\varpi_\ell}Y_{\ell, 0}Y_{0,\ell}^{-1})$ is completely reducible as a $\U_q^{v,j}(sl_{n+1}^{tor})$-module and we have
\begin{equation}\label{decvmod}
V(e^{\varpi_\ell}Y_{\ell, 0}Y_{0,\ell}^{-1})^{(j)} = \bigoplus_{p \in \ZZ} V_0(Y_{\ell, j + kp})^{(j)}.
\end{equation}
As the representations $V_0(Y_{\ell, j + kp})$ are all integrable, it holds for $V(e^{\varpi_\ell}Y_{\ell, 0}Y_{0,\ell}^{-1})$. Furthermore $V(e^{\varpi_\ell}Y_{\ell, 0}Y_{0,\ell}^{-1})$ satisfies the stronger property $(iv)$ of Remark \ref{remintmod}: in fact the representations $V_0(Y_{\ell, j + kp})$ are all isomorphic as $\U_q(sl_{n+1})$-modules and satisfy property $(iv)$. Hence we have $V(e^{\varpi_\ell}Y_{\ell, 0}Y_{0,\ell}^{-1})_{\nu + N\alpha_i} = \{0\}$ for all $\nu \in P$, $i \in I$, $N >>0$.

Finally, the case $\ell = n$ is deduced from the case $\ell = 1$ by the $\iota$-twisted automorphism $\psi$.
\end{proof}

\begin{thm}\label{thmelm}
The $\U_q(sl_{n+1}^{tor})$-module $V(e^{\varpi_\ell}Y_{\ell, 0}Y_{0,d_\ell}^{-1})$ is an extremal loop weight module generated by the vector $v_{e^{\varpi_\ell}Y_{\ell, 0}Y_{0,d_\ell}^{-1}}$ of $\ell$-weight $e^{\varpi_\ell}Y_{\ell, 0}Y_{0,d_\ell}^{-1}$.
\end{thm}

\begin{proof}
We treat the case $\ell = 1, r+1$ (the case $\ell = n$ can be deduced from $\ell = 1$ by using $\psi$). The formulas (\ref{decvmod}) imply immediately the third point of Definition \ref{defelm}. The first two points are consequences of the following lemmas.
\end{proof}

\begin{lem}\label{lirepcrylem}
Let $\mathcal{M}'$ be a sub-$\U_q(\hat{sl}_{n+1})$-crystal of $\mathcal{M}$. Assume that $V$ is a $\U_q(\hat{sl}_{n+1})$-module with basis $(v_m)_{m \in \mathcal{M}'}$ satisfying
\begin{eqnarray}\label{lienrepcry}
\wt(v_m) = \wt(m), \ (x_{i}^{+})^{(k)} \cdot v_m = v_{\tilde{e}_i^k \cdot m} \text{ and } (x_{i}^{-})^{(k)} \cdot v_m = v_{\tilde{f}_i^k \cdot m}
\end{eqnarray}
for all $m \in \mathcal{M}', i \in I$ and $k \in \NN$, where $v_0 = 0$ by convention. If the monomial $m$ is extremal of weight $\lambda$, then the vector $v_m$ is an extremal vector of weight $\lambda$. Furthermore if the crystal $\mathcal{M}'$ is connected, then the $\U_q(\hat{sl}_{n+1})$-module $V$ is cyclic generated by any $v_m$ with $m \in \mathcal{M}'$.
\end{lem}

\begin{proof}
Assume that $m$ is extremal of weight $\lambda$: there exists $\{m_w\}_{w \in W}$ such that $m_{Id}= m$ and
\begin{eqnarray}\label{eqinter}
\begin{array}{c}
\tilde{e}_i \cdot m_w=0 \text{ and }(\tilde{f}_i)^{w(\lambda)(h_i)} \cdot m_w= m_{s_i(w)} \text{ if } w(\lambda)(h_i)\geq 0,\\
\tilde{f}_i \cdot m_w=0 \text{ and }(\tilde{e}_i)^{-w(\lambda)(h_i)} \cdot m_w= m_{s_i(w)} \text{ if } w(\lambda)(h_i)\leq 0.
\end{array}
\end{eqnarray}
For all $w \in W$, set $v_w = v_{m_w}$. By (\ref{lienrepcry}) and (\ref{eqinter}), $\{v_w\}_{w \in W}$ satisfies $v_{Id} = v_{m}$ and
 $$x_i^{\pm} \cdot v_w=0 \text{ if } \pm w(\lambda)(h_i)\geq 0 \text{ and }(x_i^{\mp})^{(\pm w(\lambda)(h_i))} \cdot v_w=v_{s_i(w)}.$$
Hence the vector $v_m$ is extremal of weight $\lambda$.

Assume that the crystal $\mathcal{M}'$ is connected and fix $m \in \mathcal{M}'$. For $m' \in \mathcal{M}'$, there exists a product $s$ of Kashiwara operators such that $s(m) = m'$. Consider the corresponding operator $S \in \U_q(\hat{sl}_{n+1})$ at the level of $V$, i.e. $S$ has the same expression as $s$ where the operators $\tilde{e}_i^k$ (resp. $\tilde{f}_i^k$) are replaced by $(x_{i}^{+})^{(k)}$ (resp. $(x_i^{-})^{(k)}$) in the product ($k \in \NN, i \in I$). By (\ref{lienrepcry}), $S(v_m) = v_{s(m)} = v_{m'}$ and the $\U_q(\hat{sl}_{n+1})$-module $V$ is cyclic generated by $v_m$.
\end{proof}

\begin{lem}
Assume that $\ell = 1, r+1$. The $\ell$-weight vector $v_{e^{\varpi_\ell}Y_{\ell, 0}Y_{0,\ell}^{-1}} \in V(e^{\varpi_\ell}Y_{\ell, 0}Y_{0,\ell}^{-1})$ is an extremal vector of weight $\varpi_\ell$ for the action of $\U_q^h(sl_{n+1}^{tor})$. Furthermore
$$V(e^{\varpi_\ell}Y_{\ell, 0}Y_{0,\ell}^{-1}) = \U_q^h(sl_{n+1}^{tor}) \cdot v_{e^{\varpi_\ell}Y_{\ell, 0}Y_{0,\ell}^{-1}}.$$
\end{lem}

\begin{proof}
Let us begin to show that the basis $(v_m)$ of $V(e^{\varpi_\ell}Y_{\ell, 0}Y_{0,\ell}^{-1})$ introduced in (\ref{extrvectsp}) satisfies properties (\ref{lienrepcry}). For all $m \in \mathcal{M}(e^{\varpi_\ell}Y_{\ell, 0}Y_{0,\ell}^{-1})$, $v_m$ is an $\ell$-weight vector of $\ell$-weight $m$ and $\wt(v_m) = \wt(m)$. Fix $i \in I$ and let $j \in I$ be such that $j \neq i$. As a $\U_q^{v,j}(sl_{n+1}^{tor})$-module, $V(e^{\varpi_\ell}Y_{\ell, 0}Y_{0,\ell}^{-1})$ is completely reducible (see (\ref{decvmod})) and there exists $k \in \ZZ$ such that $v_m \in V_0(Y_{\ell, j + kp})^{(j)}$. As properties (\ref{lienrepcry}) are satisfied in $V_0(Y_{\ell, j + kp})^{(j)}$ (Proposition \ref{actfun}), it holds on $v_m$ for $i \in I$.

From there the result is a direct consequence of the above lemma and the fact that $e^{\varpi_\ell}Y_{\ell, 0}Y_{0,\ell}^{-1}$ is extremal in the connected crystal $\mathcal{M}(e^{\varpi_\ell}Y_{\ell, 0}Y_{0,\ell}^{-1})$ (Theorem \ref{monrzt}).
\end{proof}

\begin{prop}
The $ \U_{q}(sl_{n+1}^{tor}) $-module $V(e^{\varpi_\ell}Y_{\ell, 0}Y_{0,d_\ell}^{-1})$ is irreducible. Furthermore it is simple as a $\U_q^h(sl_{n+1}^{tor})$-module and $\mathrm{Res}(V(e^{\varpi_\ell}Y_{\ell, 0}Y_{0,d_\ell}^{-1}))$ is isomorphic to $V(\varpi_{\ell})$.
\end{prop}

\begin{proof}
Let $V$ be a non trivial sub-$ \U_{q}^{h}(sl_{n+1}^{tor})$-module of $V(e^{\varpi_\ell}Y_{\ell, 0}Y_{0,d_\ell}^{-1})$. As the weight spaces of $V(e^{\varpi_\ell}Y_{\ell, 0}Y_{0,d_\ell}^{-1})$ are all of dimension one, there exists $m \in \mathcal{M}(e^{\varpi_\ell}Y_{\ell, 0}Y_{0,d_\ell}^{-1})$ such that $v_m \in V$. By Lemma \ref{lirepcrylem}, $v_m$ generates $V(e^{\varpi_\ell}Y_{\ell, 0}Y_{0,d_\ell}^{-1})$ and $V = V(e^{\varpi_\ell}Y_{\ell, 0}Y_{0,d_\ell}^{-1})$. Hence $V(e^{\varpi_\ell}Y_{\ell, 0}Y_{0,d_\ell}^{-1})$ is simple as a $ \U_{q}^{h}(sl_{n+1}^{tor})$-module and as a $ \U_{q}(sl_{n+1}^{tor}) $-module. Furthermore, $\mathrm{Res}(V(e^{\varpi_\ell}Y_{\ell, 0}Y_{0,d_\ell}^{-1}))$ is an integrable $\U_q(\hat{sl}_{n+1})$-module generated by the extremal vector $v_{e^{\varpi_\ell}Y_{\ell, 0}Y_{0,d_\ell}^{-1}}$ of weight $\varpi_\ell$. Then by Theorem \ref{thmkasem}, $\mathrm{Res}(V(e^{\varpi_\ell}Y_{\ell, 0}Y_{0,d_\ell}^{-1}))$ is isomorphic to $V(\varpi_{\ell})$.
\end{proof}

Some readers may expect from this result that the $\U_{q}(sl_{n+1}^{tor})$-module $V(e^{\varpi_\ell} Y_{\ell,0}Y_{0,d_\ell}^{-1})$ can be obtained from the extremal weight module $V(\varpi_\ell)$ by an evaluation morphism, but this is not the case for the following reasons (which generalize arguments given in \cite{hernandez_quantum_2009}): $\mathrm{Res}(V(e^{\varpi_\ell} Y_{\ell,0}Y_{0,d_\ell}^{-1}))$ is isomorphic to the $\U_q(\hat{sl}_{n+1})$-module $V(\varpi_\ell)$. In particular, $$\tau_{p, -\delta} : V(\varpi_\ell) \rightarrow V(\varpi_\ell), v_m \mapsto v_{\tau_{p, -\delta}(m)} \text{ for all } m \in \mathcal{M}(e^{\varpi_\ell} Y_{\ell,0}Y_{0,d_\ell}^{-1})$$ is a $\U_q(\hat{sl}_{n+1})'$-automorphism of $V(\varpi_\ell)$ (with $p = n+1$ or $p = 2$ if $\ell= 1, n$ or $\ell = r+1$ respectively). If $V(e^{\varpi_\ell} Y_{\ell,0}Y_{0,d_\ell}^{-1})$ is obtained from an evaluation morphism \linebreak $\U_{q}(sl_{n+1}^{tor}) \rightarrow \U_{q}^{h}(sl_{n+1}^{tor})$, $\tau_{p, -\delta}$ should induce an automorphism of $V(e^{\varpi_\ell} Y_{\ell,0}Y_{0,d_\ell}^{-1})$. But it does not commute with the action of the $x_{i,r}^{\pm}, h_{i,r}$ for $i \in I$ and $r \in \ZZ-\{0\}$. In the same way, $V(e^{\varpi_\ell} Y_{\ell,0}Y_{0,d_\ell}^{-1})$ can not be obtained from an evaluation morphism $\U_{q}(sl_{n+1}^{tor}) \rightarrow \U_{q}^{v,j}(sl_{n+1}^{tor})$ ($j \in I$). In fact, $V(e^{\varpi_\ell} Y_{\ell,0}Y_{0,d_\ell}^{-1})$ is completely reducible as a $\U_{q}^{v,j}(sl_{n+1}^{tor})$-module and is a direct sum of fundamental modules (see (\ref{decvmod})). But it is a simple $\U_{q}(sl_{n+1}^{tor})$-module.

\begin{rem}\label{remnak1}
Let us denote $\U_q(sl_{n+1}^{tor})'$ the quantum toroidal algebra without derivation element, i.e. this is the subalgebra of $\U_q(sl_{n+1}^{tor})$ generated by $x_{i,r}^{\pm} \; (i \in I, r \in \ZZ), h_{i,m} \; (i \in I, m \in \ZZ-\{0\})$ and $k_h \; (h \in \sum \QQ h_i)$. An automorphism $\Psi$ of $\U_q(sl_{n+1}^{tor})'$ which exchanges vertical and horizontal quantum affine subalgebras is defined in \cite{miki_toroidal_1999}. Denote by $V(e^{\varpi_\ell} Y_{\ell,0}Y_{0,d_\ell}^{-1})^{\Psi}$ the $\U_q(sl_{n+1}^{tor})'$-module obtained from $V(e^{\varpi_\ell} Y_{\ell,0}Y_{0,d_\ell}^{-1})$ by twisting the action by $\Psi$. It would be interesting to determine if $V(e^{\varpi_\ell} Y_{\ell,0}Y_{0,d_\ell}^{-1})^{\Psi}$ is already known, for example if it is of $\ell$-highest weight. Actually this is not the case: for the vertical quantum affine subalgebra $\U_q^v(sl_{n+1}^{tor})'$, it is an integrable and cyclic module which is reducible. Further as a $\U_q^h(sl_{n+1}^{tor})'$-module, it is completely reducible, direct sum of irreducible finite-dimensional representations. So $V(e^{\varpi_\ell} Y_{\ell,0}Y_{0,d_\ell}^{-1})^{\Psi}$ cannot be an $\ell$-highest weight module or an $\ell$-lowest weight module.
\end{rem}

From now on, let $\mathcal{M}_0'$ be a subcrystal of $\mathcal{M}_0$ over $\U_q(sl_{n+1})$ (resp. $\mathcal{M}'$ subcrystal of $\mathcal{M}$ over $\U_q(\hat{sl}_{n+1})$). Let us consider the vector space $V$ with basis $(v_m)$ indexed by the vertices of $\mathcal{M}_0'$ (resp. $\mathcal{M}'$). We define an action of $\U_q(\hat{sl}_{n+1})'$ (resp. $\U_q(sl_{n+1}^{tor})$) on $V$ by the following formulas
\begin{eqnarray}\label{actmod}
\begin{array}{rcl}
x_{i,r}^{+} \cdot v_{m} &=& q^{r(p_{i}(m)-1)} v_{\tilde{e_i} \cdot m},\\
x_{i,r}^{-} \cdot v_{m} &=& q^{r(q_{i}(m)+1)} v_{\tilde{f_i} \cdot m},\\
\phi_{i,\pm s}^{\pm} \cdot v_{m} &=& \pm (q-q^{-1}) \left( \varphi_{i}(m) q^{\pm s (q_{i}(m)+1)}-\varepsilon_{i}(m) q^{\pm s (p_{i}(m)-1)} \right) v_{m},\\
k_{h} \cdot v_{m} &=& q^{\wt(m)(h)} v_{m}.
\end{array}
\end{eqnarray}

\noindent with $r \in \ZZ$, $s > 0$, $i \in I_0$ (resp. $i \in I$) and $h \in \Hlie_0$ (resp. $h \in \Hlie$), and where $v_0=0$ by convention. Note that $p_{i}(m)$ is well defined only if $\varepsilon_{i}(m) > 0$ or equivalently if $\tilde{e_i} \cdot m \neq 0$ and $q_{i}(m)$ is well defined only if $\varphi_{i}(m) > 0$ or equivalently if $\tilde{f_i} \cdot m \neq 0$. Then, these expressions make sense.

\begin{thm}\label{thmactmod}
\begin{enumerate}
\item[(i)] Set $n \in \NN^{\ast}$, $1 \leq \ell \leq n$. Assume that $\mathcal{M}_0' = \mathcal{M}_0(Y_{\ell, k})$. Then formulas (\ref{actmod}) endow $V$ with a structure of $\U_q(\hat{sl}_{n+1})'$-module isomorphic to the fundamental module $V_0(Y_{\ell,k})$.
\item[(ii)] Assume that $n=2r+1$ is odd and $\ell=1, r+1, n$. Set $\mathcal{M}' = \mathcal{M}(e^{\varpi_\ell} Y_{\ell,0}Y_{0,d_\ell}^{-1})$. Then formulas (\ref{actmod}) endow $V$ with a structure of $\U_q(sl_{n+1}^{tor})$-module isomorphic to the extremal fundamental loop weight module $V(e^{\varpi_\ell} Y_{\ell,0}Y_{0,d_\ell}^{-1})$.
\end{enumerate}
\end{thm}

\begin{proof}
The action of the horizontal quantum affine subalgebra and the action of the Cartan subalgebra are known on the basis $(v_m)_{m \in \mathcal{M}'}$ for the $\U_q(\hat{sl}_{n+1})'$-module $V_0(Y_{\ell,k})$ and for the $\U_q(sl_{n+1}^{tor})$-module $V(e^{\varpi_\ell} Y_{\ell,0}Y_{0,d_\ell}^{-1})$. From (\ref{actcartld}) it is straightforward to deduce the action of the $x_{i,r}^{\pm}$ on these modules ($r \in \ZZ$). We obtain formulas (\ref{actmod}) given only from the corresponding monomial crystal.
\end{proof}

\begin{rem} In \cite{hernandez_algebra_2011}, the algebra $\U_q(\hat{sl}_\infty)$ is introduced as the quantum affinization of $\U_q(sl_{\infty})$. It is defined by the same generators and relations as in Definition \ref{defqta} with the infinite Cartan matrix $C = (C_{i,j})_{i,j \in \ZZ}$ such that
$$C_{i,i} = 2, \ C_{i,i+1} = -1, \ C_{i+1, i} = -1, C_{i,j} = 0$$
if $i-j \notin \{-1, 0, 1\}$. The representation theory of $\U_q(\hat{sl}_\infty)$ is similar to the one of $\U_q(sl_{n+1}^{tor})$: the simple $\ell$-highest weight modules are parametrized by Drinfeld polynomials. In particular, the fundamental modules can be defined and they are the inductive limit of the fundamental modules for the quantum affine algebra $\U_q(\hat{sl}_{n+1})'$ when $n \rightarrow \infty$ (see \cite[Theorem 3.8]{hernandez_algebra_2011} and \cite[Proposition 3.11]{hernandez_algebra_2011}). So, the previous results about the fundamental modules of $\U_q(\hat{sl}_{n+1})'$ extend directly to the case of the fundamental modules of $\U_q(\hat{sl}_\infty)$.
\end{rem}

\begin{rem}
As we have said, relations between monomial crystals and the set of monomials occurring in the $q$--character of representations are known and have combinatorial origin (see \cite{hernandez_algebra_2011, hernandez_level_2006, nakajima_$t$-analogs_2003}). The above results, in particular Theorem \ref{thmactmod}, give one way to better understand the representation theoretical meaning of this narrow link expected in \cite{hernandez_level_2006}. In fact, formulas (\ref{actmod}) hold for all the fundamental $\U_q(\hat{sl}_{n+1})'$-modules $V_0(Y_{\ell,k})$ and for all the extremal fundamental loop weight $\U_q(sl_{n+1}^{tor})$-modules $V(e^{\varpi_\ell}Y_{\ell,0}Y_{0, d_\ell}^{-1})$. Hence the knowledge of these representations is reduced to the one of the corresponding crystals $\mathcal{M}_0(Y_{\ell,k})$ and $\mathcal{M}(e^{\varpi_\ell}Y_{\ell,0}Y_{0, d_\ell}^{-1})$ respectively, which is totally combinatorial.
\end{rem}

\begin{ex} Assume that $n = 3$ and $\ell = 1$. We study the extremal fundamental loop weight module $V(e^{\varpi_1}Y_{1,0}Y_{0,1}^{-1})$ for $\U_q(sl_{4}^{tor})$. Let us consider the monomial crystal $\mathcal{M}(e^{\varpi_1}Y_{1,0}Y_{0,1}^{-1})$. It is closed and $p = 4$ in this case. Using the notations introduced above, $\mathcal{E} = \sqcup_{k \in \ZZ} \mathcal{E}_{0,k}$ and we have
$$\mathcal{E}_{0,0} = \left\lbrace e^{\varpi_1}Y_{1,0}Y_{0,1}^{-1}, Y_{2,1}Y_{1,2}^{-1}, Y_{3,2}Y_{2,3}^{-1}, Y_{0,3}Y_{3,4}^{-1} \right\rbrace.$$
$\mathcal{E}_{0,k}$ can be obtained from $\mathcal{E}_{0,0}$ by applying $\tau_{4k, -\delta}$. In the same way, we obtain $\mathcal{E}_{j,k}$ by applying $\phi^{j+4k}$ to $\mathcal{E}_{0,0}$ . Then the $q$--character of the extremal fundamental loop weight module $V(e^{\varpi_1}Y_{1,0}Y_{0,1}^{-1})$ is
\begin{align*}
\chi_q(V(e^{\varpi_1}Y_{1,0}Y_{0,1}^{-1})) = \sum_{k \in \ZZ} & e^{\varpi_1 - k\delta} Y_{1,4k}Y_{0,1+4k}^{-1} + Y_{2,1+4k}Y_{1,2+4k}^{-1} + Y_{3,2+4k}Y_{2,3+4k}^{-1} \\
&+ Y_{0,3+4k}Y_{3,4+4k}^{-1}.
\end{align*}
Furthermore the action is explicitly given by the crystal $\mathcal{M}(e^{\varpi_1}Y_{1,0}Y_{0,1}^{-1})$ and by the formulas (\ref{actmod}). This module is already constructed in \cite{hernandez_quantum_2009}.
\end{ex}

\begin{rem}\label{remnak2}
After this paper appeared on the arXiv, the constructions in \cite{feigin_representations_2013} were brought to our
attention by H. Nakajima: some representations over the $d$-deformation $\U_{q,d}(sl_{n+1}^{tor})$ of the quantum toroidal algebra are obtained as quantum version of a module over a Lie algebra of difference operators. They are called vector representations in \cite{feigin_representations_2013}. Our works give another way to define these representations. Actually let $\omega : \U_q(sl_{n+1}^{tor}) \rightarrow \U_{q^{-1}}(sl_{n+1}^{tor})$ be the map sending $x_{i,r}^{\pm}, h_{i,m}, k_{h}$ to $x_{i,r}^{\mp}, h_{i,m}, k_h$ $(i \in I, r \in \ZZ, m \in \ZZ-\{0\}, h \in \Hlie)$. $\omega$ extends to an isomorphism of algebras. For $u \in \CC^{\ast}$, let $[V(e^{\varpi_1}Y_{1,0}Y_{0,1}^{-1})]_u$ be the $\U_q(sl_{n+1}^{tor})$-module obtained from $V(e^{\varpi_1}Y_{1,0}Y_{0,1}^{-1})$ by twisting the action by $t_{uq^{-1}} \circ \omega$. Then $[V(e^{\varpi_1}Y_{1,0}Y_{0,1}^{-1})]_u$ is isomorphic to a vector representation where we specialized the parameter $d$ at $1$ (this representation is denoted by $V^{(2)}(u)$ in \cite{feigin_representations_2013}).
\end{rem}

\begin{ex} Assume that $n=3$ and $ \ell = 2 $. Let us study the extremal fundamental loop weight module $V(e^{\varpi_2}Y_{2,0}Y_{0,2}^{-1})$ of $\U_q(sl_{4}^{tor})$. Consider the closed monomial crystal $\mathcal{M}(e^{\varpi_2}Y_{2,0}Y_{0,2}^{-1})$. In this case, $p=2$ and we have
$$ \mathcal{E}_{0,0} = \left\lbrace \begin{array}{c}
e^{\varpi_2}Y_{2, 0}Y_{0, 2}^{-1}, Y_{1, 1} Y_{2, 2}^{-1} Y_{3, 1} Y_{0, 2}^{-1}, Y_{1, 1}Y_{3, 3}^{-1},\\
Y_{3, 1}Y_{1, 3}^{-1}, Y_{1, 3}^{-1} Y_{2, 2} Y_{3, 3}^{-1} Y_{0, 2}, Y_{2,4}^{-1}Y_{0, 2}
\end{array} \right\rbrace.$$
To describe all the monomials occurring in $\mathcal{M}(e^{\varpi_2}Y_{2,0}Y_{0,2}^{-1})$, it is sufficient to consider only the sub-$I_{\{0,1\}}$-crystal
$$Y_{2, 0}Y_{0, 2}^{-1} \overset{2}{\longrightarrow} Y_{1, 1} Y_{2, 2}^{-1} Y_{3, 1} Y_{0, 2}^{-1} \overset{3}{\longrightarrow} Y_{1, 1} Y_{3, 3}^{-1}$$
and to apply the $\theta$-twisted automorphism $\phi$ (Remark \ref{remprom}). The $q$--character of \linebreak $V(e^{\varpi_2}Y_{2,0}Y_{0,2}^{-1})$ is
\begin{align*}
\chi_q(V(e^{\varpi_2}Y_{2,0}Y_{0,2}^{-1})) = \sum_{k \in \ZZ} & e^{\varpi_2-k\delta} Y_{2, 2k}Y_{0, 2+2k}^{-1} + Y_{1, 1 + 2k} Y_{2, 2+2k}^{-1} + Y_{3, 1 + 2k} Y_{0, 2 + 2k}^{-1}\\
& + Y_{1, 3+2k}^{-1}Y_{3, 1+2k} + Y_{3, 1+2k}Y_{1, 3+2k}^{-1} + Y_{1, 3 + 2k}^{-1} Y_{2, 2+2k}\\
& + Y_{3, 3 + 2k}^{-1} Y_{0, 2 + 2k} + Y_{2,4+2k}^{-1}Y_{0, 2+2k},
\end{align*}
and the action of $ \U_{q}(sl_{4}^{tor}) $ on $V(e^{\varpi_2}Y_{2,0}Y_{0,2}^{-1})$ is explicitly given by the crystal \linebreak $\mathcal{M}(e^{\varpi_2}Y_{2,0}Y_{0,2}^{-1})$ and formulas (\ref{actmod}).
\end{ex}

\subsection{Finite-dimensional representations at roots of unity} \nocite{chari_quantum_1997, frenkel_$q$-characters_2002}

The existence of shift automorphisms for $\mathcal{M}(e^{\varpi_\ell}Y_{\ell,0}Y_{0, d_\ell}^{-1})$ is related to finite-dimensional representations of quantum toroidal algebras at roots of unity. We explain that in this section.

So assume that $n = 2r+1$ is odd ($r \geq 1$) and $\ell = 1, n$ or $\ell = r+1$. In this case $\mathcal{M}(e^{\varpi_\ell}Y_{\ell,0}Y_{0, d_\ell}^{-1})$ is closed and its automorphism $z_\ell$ has the special form of a shift $\tau_{-p, \delta}$ with $p = n+1$ or $p=2$ respectively.

Set $L \geq 1$ and $\epsilon$ a primitive $(p L)$-root of unity (we assume also that $p \neq 2$ or $L >1$ in the following). Let $\U_\epsilon(sl_{n+1}^{tor})'$ be the algebra defined as $\U_q(sl_{n+1}^{tor})$ with $\epsilon$ instead of $q$ (without divided powers and derivation element).

For $N \in \NN^{\ast}$, let $\Gamma_{N} : \ZZ[Y_{i,l}^{\pm 1}]_{i \in I, l \in \ZZ} \rightarrow \ZZ[Y_{i,\overline{l}}^{\pm 1}]_{i \in I, \overline{l} \in \ZZ / N \ZZ }$ be the map defined by sending the variables $Y_{i,l}^{\pm 1}$ to $Y_{i, \overline{l}}^{\pm 1}$ ($i \in I, l \in \ZZ$). Set $\mathcal{S}_\epsilon$ the image of a monomial set $\mathcal{S}$ by $\Gamma_{(pL)}$.

Consider the monomial set $\mathcal{E}_\epsilon$. By the existence of the shift automorphism $\tau_{-p, \delta}$, we have
$$\mathcal{E}_\epsilon =  \bigsqcup_{0 \leq k \leq L-1} (\tau_{p, -\delta})^{k} \left(  (\mathcal{E}_{j,k})_\epsilon \right)$$
with $j \in I$. One checks easily that $\mathcal{E}_\epsilon$ is closed.

By specializing the representations $V(e^{\varpi_\ell}Y_{\ell,0}Y_{0, d_\ell}^{-1})$ at a root of unity $\epsilon$, we obtain

\begin{thm}\label{thmmodunit}
Assume that $\epsilon$ is a primitive $(pL)$-root of unity. There is an irreducible $\U_\epsilon(sl_{n+1}^{tor})'$-module $V(e^{\varpi_\ell}Y_{\ell,0}Y_{0, d_\ell}^{-1})_\epsilon$ of dimension $L \cdot \begin{pmatrix} n+1 \\ \ell \end{pmatrix}$ such that
$$\chi_\epsilon(V(e^{\varpi_\ell}Y_{\ell,0}Y_{0, d_\ell}^{-1})_\epsilon) = \sum_{m \in \mathcal{E}_\epsilon} m.$$
Furthermore there exists a basis $(v_m)$ of $V(e^{\varpi_\ell}Y_{\ell,0}Y_{0, d_\ell}^{-1})_\epsilon$ indexed by $\mathcal{E}_{\epsilon}$ such that the action on it is given by
\begin{eqnarray*}
\begin{array}{rcl}
x_{i,r}^{+} \cdot v_{m} &=& \epsilon^{r(p_{i}(m)-1)} v_{\tilde{e_i} \cdot m},\\
x_{i,r}^{-} \cdot v_{m} &=& \epsilon^{r(q_{i}(m)+1)} v_{\tilde{f_i} \cdot m},\\
\phi_{i,\pm s}^{\pm} \cdot v_{m} &=& \pm (\epsilon-\epsilon^{-1}) \left( \varphi_{i}(m) \epsilon^{\pm s (q_{i}(m)+1)} -\varepsilon_{i}(m) \epsilon^{\pm s (p_{i}(m)-1)} \right) v_{m},\\
k_{i}^{\pm} \cdot v_{m} &=& \epsilon^{\pm (\varphi_{i}(m)-\varepsilon_{i}(m))} v_{m}.
\end{array}
\end{eqnarray*}
\end{thm}

\section{Extremal loop weight modules for $\U_q(sl_{n+1}^{tor})$ when the considered monomial crystal is not closed}

In this section, we still assume that $n = 2r+1$ is odd and we discuss the case where the considered monomial crystal $\mathcal{M}'$ is not closed. It is not possible here to construct an integrable module whose $q$--character is a sum of monomials occurring in $ \mathcal{M}'$. In fact some monomials miss and we have to consider a larger closed monomial crystal $\overline{\mathcal{M}'}$ containing it. It is obtained from $\mathcal{M}'$ by adding other monomial crystals. But its structure is more complicated than $ \mathcal{M}' $ and it is difficult for us to construct systematically a possible representation of $\U_q(sl_{n+1}^{tor})$ associated to $\overline{\mathcal{M}'}$.

So in this section, we propose to treat an example of such a construction. Assume in the following that $n=3$ and consider the crystal $\mathcal{M}(e^{2\varpi_1}Y_{1,1}Y_{1,-1}Y_{0,2}^{-1}Y_{0,0}^{-1})$ which is not closed. We determine a closed monomial crystal $\overline{\mathcal{M}}(e^{2\varpi_1}Y_{1,1}Y_{1,-1}Y_{0,2}^{-1}Y_{0,0}^{-1})$ containing it and we construct a representation $V(e^{2\varpi_1}Y_{1,1}Y_{1,-1}Y_{0,2}^{-1}Y_{0,0}^{-1})$ of $\U_{q}(sl_{4}^{tor})$ such that its $q$--character is the sum of monomials occurring in $\overline{\mathcal{M}}(e^{2\varpi_1}Y_{1,1}Y_{1,-1}Y_{0,2}^{-1}Y_{0,0}^{-1})$ with multiplicity one (Theorem \ref{thmexmod}). Furthermore we will see that  $V(e^{2\varpi_1}Y_{1,1}Y_{1,-1}Y_{0,2}^{-1}Y_{0,0}^{-1})$ satisfies the definition of extremal loop weight module.

In the first part, we study the crystal $\mathcal{M}(e^{2\varpi_1}Y_{1,1}Y_{1,-1}Y_{0,2}^{-1}Y_{0,0}^{-1})$ and we determine a closed monomial crystal $\overline{\mathcal{M}}(e^{2\varpi_1}Y_{1,1}Y_{1,-1}Y_{0,2}^{-1}Y_{0,0}^{-1})$ containing it.

The construction of the $\U_{q}(sl_{4}^{tor})$-module $V(e^{2\varpi_1}Y_{1,1}Y_{1,-1}Y_{0,2}^{-1}Y_{0,0}^{-1})$ is done in the second part. The process is the same as in the preceding section: we consider the vector space freely generated by the vertices $m$ of $\overline{\mathcal{M}}(e^{2\varpi_1}Y_{1,1}Y_{1,-1}Y_{0,2}^{-1}Y_{0,0}^{-1})$ and we define an action of $\U_{q}(sl_{4}^{tor})$ by pasting together some finite-dimensional representations of the vertical quantum affine subalgebras $\U_{q}^{v,j}(sl_{4}^{tor})$ ($j \in I$).

In the third part, we study the representation $V(e^{2\varpi_1}Y_{1,1}Y_{1,-1}Y_{0,2}^{-1}Y_{0,0}^{-1})$: it is an integrable representation of $\U_{q}(sl_{4}^{tor})$ which is thin and irreducible. Furthermore \linebreak $V(e^{2\varpi_1}Y_{1,1}Y_{1,-1}Y_{0,2}^{-1}Y_{0,0}^{-1})$ is an extremal loop weight module of $\ell$-weight \linebreak $e^{2\varpi_1}Y_{1,1}Y_{1,-1}Y_{0,2}^{-1}Y_{0,0}^{-1}$.

In the last part, we specialize $q$ at roots of unity $\epsilon$. We get finite-dimensional representations of the quantum toroidal algebra $\U_\epsilon(sl_{4}^{tor})'$.

\begin{rem}
It could be interesting to construct other extremal fundamental loop weight modules of $\ell$-weight $e^{\varpi_\ell}Y_{\ell, 0} Y_{0, \ell}^{-1}$ with $2 \leq \ell \leq r$ in the same way. The first crystal $ \mathcal{M}(e^{\varpi_\ell}Y_{\ell, 0} Y_{0, \ell}^{-1}) $ which is not closed is obtained for $n=5$ and $\ell = 2$. We are led to consider the following closed crystal
$$ \overline{\mathcal{M}}(e^{\varpi_2}Y_{2, 0} Y_{0, 2}^{-1}) = \mathcal{M}(e^{\varpi_2}Y_{2, 0} Y_{0, 2}^{-1}) \oplus \bigoplus_{s \in \NN^{\ast}} \mathcal{M}(e^{2\Lambda_1-s\delta}Y_{1,1}Y_{1,-1+6s}Y_{0,2}^{-1}Y_{0,6s}^{-1}). $$
which contains $ \mathcal{M}(e^{\varpi_2}Y_{2, 0} Y_{0, 2}^{-1}) $. The maps $\phi$ and $\tau_{6, -2\delta}$ are automorphisms of it and the $P_{\cl}$-crystals $\mathcal{M}(e^{\varpi_2}Y_{2, 0} Y_{0, 2}^{-1}) / (\tau_{6,-2\delta})$ and $\mathcal{M}(e^{2\Lambda_1-s\delta}Y_{1,1}Y_{1,-1+6s}Y_{0,2}^{-1}Y_{0,6s}^{-1}) / (\tau_{6, -2\delta})$ have 30 vertices and 36 vertices respectively.

The example we propose to treat in this section is simpler than the case of the extremal fundamental loop weight modules and we focus only on this situation for the sake of clarity and simplicity.
\end{rem}

\subsection{Study of the monomial $\U_q(\hat{sl}_{4})$-crystal $\mathcal{M}(e^{2\varpi_1}Y_{1,1}Y_{1,-1}Y_{0,2}^{-1}Y_{0,0}^{-1})$}

We refer to the Appendix for explicit descriptions of all the crystals considered in this section. Let us study the monomial crystal $\mathcal{M}(e^{2\varpi_1}Y_{1,1}Y_{1,-1}Y_{0,2}^{-1}Y_{0,0}^{-1})$: the maps $\phi$ and $\tau_{4, -2\delta}$ are automorphisms of $\mathcal{M}(e^{2\varpi_1}Y_{1,1}Y_{1,-1}Y_{0,2}^{-1}Y_{0,0}^{-1})$. Furthermore straightforward computations lead to the following result.

\begin{prop}
\begin{enumerate}
\item[(i)] We have the equality of sets
\begin{eqnarray*}
\mathcal{M}(e^{2\varpi_1}Y_{1,1}Y_{1,-1}Y_{0,2}^{-1}Y_{0,0}^{-1}) = \bigcup_{k \in \ZZ} (\tau_{4, -2\delta})^{k} \left( \bigcup_{j = 0}^{3} \mathcal{M}_{I_{j}}(Y_{1 + j, 1 + j} Y_{1+j,-1+j} Y_{j, 2 + j}^{-1} Y_{j,j}^{-1}) \right).
\end{eqnarray*}
\item[(ii)] For all $j \in I$, the monomial crystal $\mathcal{M}_{I_{j}}(Y_{1 + j, 1 + j} Y_{1+j,-1+j} Y_{j, 2 + j}^{-1} Y_{j,j}^{-1})$ is $I_{j}$-q-closed. More precisely, we have the bijection of monomial sets
$$\Xi^{j}: \mathcal{M}_{I_{j}}(Y_{1 + j, 1 + j} Y_{1+j,-1+j} Y_{j, 2 + j}^{-1} Y_{j,j}^{-1}) \longrightarrow \mathcal{M}(V_0(Y_{1, 1 + j} Y_{1,-1+j})^{(j)})$$
where $V_0(Y_{1, 1 + j} Y_{1,-1+j})$ is the simple $\ell$-highest weight representation of $\U_{q}(\hat{sl}_{4})'$ of $\ell$-highest weight $Y_{1, 1 + j} Y_{1,-1+j}$.
\item[(iii)] For all $j \in I$, the $I_{j}$-crystal $\mathcal{M}_{I_{j}}(Y_{1 + j, 1 + j} Y_{1+j,-1+j} Y_{j, 2 + j}^{-1} Y_{j,j}^{-1})$ is not q-closed: the monomial $\phi^{j}(Y_{1,-1}Y_{3,5}^{-1}Y_{0,0}^{-1}Y_{0,4})$ occurs in this crystal, but it is not the case of $\phi^{j}(Y_{1,-1}Y_{3,5}^{-1}Y_{0,0}^{-1}Y_{0,4} \cdot A_{0,-1}) = \phi^{j}(Y_{1,5}Y_{1,-1}Y_{0,6}^{-1}Y_{0,0}^{-1})$.
\end{enumerate}
\end{prop}

Hence, we are led to consider the crystal $\mathcal{M}(e^{2\varpi_1 + \delta}Y_{1,1}Y_{1,-5}Y_{0,2}^{-1}Y_{0,-4}^{-1})$ which is also not closed. More generally we have to deal with all the monomial crystals \linebreak $\mathcal{M}(e^{2\varpi_1 + s\delta}Y_{1,1}Y_{1,-1-4s}Y_{0,2}^{-1}Y_{0,-4s}^{-1})$ with $s \in \NN$. We set
$$\overline{\mathcal{M}}(e^{2\varpi_1}Y_{1,1}Y_{1,-1}Y_{0,2}^{-1}Y_{0,0}^{-1}) = \bigoplus_{s \in \NN} \mathcal{M}(e^{2\varpi_1 + s\delta}Y_{1,1}Y_{1,-1-4s}Y_{0,2}^{-1}Y_{0,-4s}^{-1}).$$

For all $(k,s) \in \ZZ \times \NN$ and $j \in I$, denote by
\begin{itemize}
\item $\mathcal{M}_{j,k,s}^{1}$ the sub-$I_{j}$-crystal  of $\mathcal{M}(e^{2\varpi_1 + s\delta}Y_{1,1}Y_{1,-1-4s}Y_{0,2}^{-1}Y_{0,-4s}^{-1})$ generated by the monomial $\phi^{j+4k}(e^{2\varpi_1 + s\delta}Y_{1,1}Y_{1,-1-4s}Y_{0,2}^{-1}Y_{0,-4s}^{-1})$,
\item $\mathcal{M}_{j,k,s}^{2}$ the sub-$I_{j}$-crystal  of $\mathcal{M}(e^{2\varpi_1 + s\delta}Y_{1,1}Y_{1,-1-4s}Y_{0,2}^{-1}Y_{0,-4s}^{-1})$ generated by the monomial $\phi^{j+4k}(Y_{1,1}Y_{1,-3-4s}^{-1}Y_{2,-4-4s}Y_{0,2}^{-1})$.
\end{itemize}

\begin{prop}
\begin{enumerate}
\item[(i)] For all $s \in \NN$ and $j \in I$, one has the equality of $I_{j}$-crystals
\begin{eqnarray*}
\mathcal{M}(e^{2\varpi_1 + s\delta}Y_{1,1}Y_{1,-1-4s}Y_{0,2}^{-1}Y_{0,-4s}^{-1}) = \bigoplus_{k \in \ZZ} \left(\mathcal{M}_{j,k,s}^{1}\oplus \mathcal{M}_{j,k,s}^{2}\right).
\end{eqnarray*}
\item[(ii)] For all $j \in I$, $k \in \ZZ$ and $s \geq 1$, the monomial crystal $ \mathcal{M}_{j,k,s} = \mathcal{M}_{j,k,s}^{1} \oplus \mathcal{M}_{j,k,s-1}^{2}$ is $I_{j}$-q-closed. More precisely, we have the bijection of monomial sets 
$$ \Xi^{j} : \mathcal{M}_{j,k,s} \longrightarrow \mathcal{M}\left(V_0(Y_{1, 1 + j + 4k}Y_{1,-1+j+4k-4s})^{(j)} \right)$$
where $V_0(Y_{1, 1 + j + 4k}Y_{1,-1+j+4k-4s})$ is the $\ell$-highest weight representation of \linebreak $\U_{q}(\hat{sl}_{4})'$ of $\ell$-highest weight $Y_{1, 1 + j + 4k}Y_{1,-1+j+4k-4s}$.
\end{enumerate}
\end{prop}

\noindent The proof of these statements is straightforward. As a consequence of these results, we have

\begin{cor}
The monomial crystal $\overline{\mathcal{M}}(e^{2\varpi_1}Y_{1,1}Y_{1,-1}Y_{0,2}^{-1}Y_{0,0}^{-1})$
is closed.
\end{cor}

\begin{prop}
$\overline{\mathcal{M}}(e^{2\varpi_1}Y_{1,1}Y_{1,-1}Y_{0,2}^{-1}Y_{0,0}^{-1})$ is a monomial realization of the $P$-crystal $\mathcal{B}(2 \varpi_{1})$. Furthermore, the monomials $M_s = e^{2\varpi_1 + s\delta}Y_{1,1}Y_{1,-1-4s}Y_{0,2}^{-1}Y_{0,-4s}^{-1}$ are extremal of weight $2 \varpi_1 + s\delta$ $(s\in \NN)$.
\end{prop}

\begin{proof}
The monomial crystal $\mathcal{M}(e^{2\varpi_1}Y_{1,1}^{2}Y_{0,2}^{-2})$ is isomorphic to the connected component of $\mathcal{B}(2 \varpi_{1})$ generated by $v_{2 \varpi_1}$ \cite[Proposition 3.1]{hernandez_level_2006}. One checks that the map
$$\mathcal{M}(e^{2\varpi_1 + s\delta}Y_{1,1}Y_{1,-1-4s}Y_{0,2}^{-1}Y_{0,-4s}^{-1}) \longrightarrow \mathcal{M}(e^{2\varpi_1}Y_{1,1}^{2}Y_{0,2}^{-2})$$ which sends the monomial $M_s$ to the extremal element $e^{2\varpi_1}Y_{1,1}^{2}Y_{0,2}^{-2}$ is an isomorphism of $P_\cl$-crystals for all $s \in \NN$. Then the result is a direct consequence of the description of the crystal $\mathcal{B}(2\varpi_1)$ given in \cite{beck_crystal_2004}: all the connected components of $\mathcal{B}(2\varpi_1)$ are isomorphic to each other modulo shift of weight by $\delta$.
\end{proof}

\subsection{Construction of the $\U_q(sl_{4}^{tor})$-module $V(e^{2\varpi_1}Y_{1,1}Y_{1,-1}Y_{0,2}^{-1}Y_{0,0}^{-1})$.}

Let us give the main result of this section.

\begin{thm}\label{thmexmod}
There exists a thin representation of $\U_q(sl_{4}^{tor})$ whose $q$--character is the sum of monomials occurring in $\overline{\mathcal{M}}(e^{2\varpi_1}Y_{1,1}Y_{1,-1}Y_{0,2}^{-1}Y_{0,0}^{-1})$ with multiplicity one. It is denoted by $V(e^{2\varpi_1}Y_{1,1}Y_{1,-1}Y_{0,2}^{-1}Y_{0,0}^{-1})$.
\end{thm}

The construction of $V(e^{2\varpi_1}Y_{1,1}Y_{1,-1}Y_{0,2}^{-1}Y_{0,0}^{-1})$ is analogous to the one of $V(e^{\varpi_\ell}Y_{\ell, 0}Y_{0, \ell}^{-1})$ in Theorem \ref{thmmod}: we paste together the finite-dimensional representations \linebreak $V_0(Y_{1,1+j+4k}Y_{1,-1+j+4k})^{(j)}$ and $V_0(Y_{1,1+j+4k}Y_{1,-1+j+4k-4s})^{(j)}$ of $\U_q^{v,j}(sl_{4}^{tor})$ with $j \in I$, $k \in \ZZ$ and $s \in \NN^{\ast}$.\\

Let us begin by recalling some well-known facts about the Kirillov-Reshetikhin module $V_0(\Xi^{0}(M))$ over $\U_q(\hat{sl}_4)'$ with $M = e^{2\varpi_1}Y_{1,1}Y_{1,-1}Y_{0,2}^{-1}Y_{0,0}^{-1}$. It is irreducible as a $\U_q(sl_4)$-module, isomorphic to $V_0(2 \Lambda_1)$. In particular, $V_0(\Xi^{0}(M))$ is an extremal weight module of extremal weight $2 \Lambda_1$ and there exist vectors $v_{\phi^{j}(M)}$ ($j= 0, \dots, 3$) such that $v_{M}$ is an $\ell$-highest weight vector of $V_0(\Xi^0(M))$ and
\begin{align*}
(x_{i,0}^{-})^{(2)} \cdot v_{\phi^{i-1}(M)} &=  v_{\phi^{i}(M)}  \text{ for } i= 1, \dots, 3,\\
(x_{i,0}^{+})^{(2)}  \cdot v_{\phi^{i}(M)} &= v_{\phi^{i-1}(M)}  \text{ for } i= 1, \dots, 3,\\
x_{i,0}^{\pm}  \cdot v_{\phi^{j}(M)} &= 0 \text{ in the other cases.}
\end{align*}
Set
$$v_{\tilde{f}_1 \cdot M} := x_{1,0}^{-} \cdot v_{M}, \ v_{\tilde{f}_2 \tilde{f}_1 \cdot M} := x_{2,0}^{-} x_{1,0}^{-} \cdot v_{M}, \ v_{\tilde{f}_3 \tilde{f}_2 \tilde{f}_1 \cdot M} := x_{3,0}^{-} x_{2,0}^{-} x_{1,0}^{-} \cdot v_{M},$$
$$v_{\tilde{f}_2 \cdot \phi(M)} := x_{2,0}^{-} \cdot v_{\phi(M)}, \ v_{\tilde{f}_3 \tilde{f}_2 \cdot \phi(M)} := x_{3,0}^{-} x_{2,0}^{-} \cdot v_{\phi(M)},$$
$$v_{\tilde{f}_3 \cdot \phi^2(M)} := x_{3,0}^{-} \cdot v_{\phi^{2}(M)}.$$
These vectors form a basis $(v_m)$ of $V_0(\Xi^{0}(M))$, indexed by the monomials occurring in $\mathcal{M}_{I_0}(M)$. Furthermore for all $m \in \mathcal{M}_{I_0}(M)$, $v_m$ is an $\ell$-weight vector of $\ell$-weight $\Xi^0(m)$.\\

The other finite-dimensional representations of $\U_q(\hat{sl}_{4})'$ we have to consider are $V_0(\Xi^0(M_s))$ with $M_s=e^{2\varpi_1+s\delta}Y_{1,1}Y_{1,-1-4s}Y_{0,2}^{-1}Y_{0,-4s}^{-1}$ and $s \in \NN^{\ast}$. The following two points are well-known:
\begin{enumerate}
\item[(i)] $V_0(\Xi^0(M_s))$ is an irreducible $\U_q(\hat{sl}_{4})'$-module isomorphic to $V_0(Y_{1,1}) \otimes V_0(Y_{1,-1-4s})$,
\item[(ii)] $\mathrm{Res}(V_0(\Xi^0(M_s)))$ is a completely reducible $\U_q(sl_{4})$-module isomorphic to \linebreak $V_0(2\Lambda_1) \oplus V_0(\Lambda_2)$.
\end{enumerate}
Furthermore there exist vectors $v_{\phi^{j}(M_s)}$ ($j= 0, \dots, 3$) such that $v_{M_s}$ is an $\ell$-highest weight vector of $V_0(\Xi^0(M_s))$ and
\begin{align*}
(x_{i,0}^{-})^{(2)} \cdot v_{\phi^{i-1}(M_s)} &=  v_{\phi^{i}(M_s)}  \text{ for } i= 1, \dots, 3,\\
(x_{i,0}^{+})^{(2)} \cdot v_{\phi^{i}(M_s)} &= v_{\phi^{i-1}(M_s)}  \text{ for } i= 1, \dots, 3,\\
x_{i,0}^{\pm} \cdot v_{\phi^{j}(M_s)} &= 0 \text{ in the other cases.}
\end{align*}
To complete this family of vectors to a basis of $V_0(\Xi^0(M_s))$, the following example is used.

\begin{ex}\label{exprod} Let $a, b \in \ZZ$ be such that $a \neq b$ and $a \neq b \pm 2$. Consider the $\U_q(\hat{sl}_2)'$-module $V_0(Y_{1,a}Y_{1,b})$. This module was already studied in \cite{hernandez_simple_2010}. We have
$$\chi_q(V_0(Y_{1,a}Y_{1,b})) = Y_{1,a}Y_{1,b} + Y_{1,a}Y_{1,b+2}^{-1} + Y_{1,a+2}^{-1} Y_{1,b} + Y_{1,a+2}^{-1}Y_{1,b+2}^{-1}.$$
In particular, it was shown that there exists a basis
$$\lbrace v_{Y_{1,a}Y_{1,b}}, v_{Y_{1,a+2}^{-1}Y_{1,b}}, v_{Y_{1,a}Y_{1,b+2}^{-1}}, v_{Y_{1,a+2}^{-1}Y_{1,b+2}^{-1}} \rbrace$$
where the action of the Drinfeld generators on it is given by
\begin{eqnarray*}
x_{r}^{+} \cdot v_{Y_{1,a}Y_{1,b}}& = & 0,\\
x_{r}^{-} \cdot v_{Y_{1,a}Y_{1,b}} & = & \dfrac{q^{b-1} - q^{a+1}}{q^b-q^a} q^{r(a+1)}v_{Y_{1,a+2}^{-1}Y_{1,b}} +\dfrac{q^{b+1} - q^{a-1}}{q^b-q^a} q^{r(b+1)}v_{Y_{1,a}Y_{1,b+2}^{-1}},\\
x_{r}^{+} \cdot  v_{Y_{1,a+2}^{-1}Y_{1,b}} & = & q^{r(a+1)} v_{Y_{1,a}Y_{1,b}},\\
x_{r}^{-} \cdot  v_{Y_{1,a+2}^{-1}Y_{1,b}} & = & q^{r(b+1)} v_{Y_{1,a+2}^{-1}Y_{1,b+2}^{-1}},\\
x_{r}^{+} \cdot  v_{Y_{1,a}Y_{1,b+2}^{-1}} & = & q^{r(b+1)} v_{Y_{1,a}Y_{1,b}},\\
x_{r}^{-} \cdot  v_{Y_{1,a}Y_{1,b+2}^{-1}} & = & q^{r(a+1)} v_{Y_{1,a+2}^{-1}Y_{1,b+2}^{-1}},\\
x_{r}^{+} \cdot v_{Y_{1,a+2}^{-1}Y_{1,b+2}^{-1}} & = & \dfrac{q^{b-1} - q^{a+1}}{q^b-q^a} q^{r(b+1)}v_{Y_{1,a+2}^{-1}Y_{1,b}} + \dfrac{q^{b+1} - q^{a-1}}{q^b-q^a} q^{r(a+1)} v_{Y_{1,a}Y_{1,b+2}^{-1}},\\
x_{r}^{-} \cdot v_{Y_{1,a+2}^{-1}Y_{1,b+2}^{-1}} & = & 0,
\end{eqnarray*}
and with $v_{m}$ of $\ell$-weight $m$ for $m=Y_{1,a}Y_{1,b}, \dots, Y_{1,a+2}^{-1}Y_{1,b+2}^{-1}$. Note that the basis used in \cite{hernandez_simple_2010} is renormalized here and we have
$$(x_0^{-})^{(2)} \cdot v_{Y_{1,a}Y_{1,b}} = v_{Y_{1,a+2}^{-1}Y_{1,b+2}^{-1}}, \ (x_0^{+})^{(2)} \cdot v_{Y_{1,a+2}^{-1}Y_{1,b+2}^{-1}} = v_{Y_{1,a}Y_{1,b}}.$$
As $a \neq b \pm 2$, it is well-known that the $\U_q(\hat{sl}_2)'$-module $V_0(Y_{1,a}Y_{1,b})$ is isomorphic to $V_0(Y_{1,a}) \otimes V_0(Y_{1,b})$. Furthermore the $\U_q(sl_2)$-module $\mathrm{Res}(V_0(Y_{1,a}Y_{1,b}))$ is not irreducible, but it is cyclic generated by one of vectors $v_{Y_{1,a}Y_{1,b+2}^{-1}}$ or $v_{Y_{1,a+2}^{-1}Y_{1,b}}$.
\end{ex}

Set $M_s^1 = Y_{1,3}^{-1}Y_{1,-1-4s}Y_{2,2}Y_{0,-4s}^{-1}$ and $M_s^2 = Y_{1,1}Y_{1,1-4s}^{-1}Y_{2,-4s}Y_{0,2}^{-1}$. Let $v_{M_s^1}$ and $v_{M_s^2} \in V_0(\Xi^0(M))$ of $\ell$-weight $M_s^1$ and $M_s^2$ respectively, be such that
$$x_{1,0}^{-} \cdot v_{M_s} = \dfrac{q^{-2-4s} - q^{4}}{q^{-1-4s}-q^{3}} v_{M_s^1} + \dfrac{q^{-4s} - q^{2}}{q^{-1-4s}-q^{3}} v_{M_s^2}.$$
Set
$$v_{\tilde{f}_2 \cdot M_s^u} := x_{2,0}^{-} \cdot v_{M_s^u}, \ v_{\tilde{f}_3 \tilde{f}_2 \cdot M_s^u} := x_{3,0}^{-} x_{2,0}^{-} \cdot v_{M_s^u}$$
with $u = 1, 2$. In the same way, one can define $v_{\phi(M_s^u)}, v_{\tilde{f}_3 \cdot \phi(M_s^u)}$ and $v_{\phi^{2}(M_s^u)}$ for $u = 1, 2$. We check that these vectors form a basis $(v_m)$ of $V_0(\Xi^0(M_s))$, indexed by the monomials occurring in $\mathcal{M}_{0,0,s}$. Moreover $v_m$ is an $\ell$-weight vector of $\ell$-weight $\Xi^0(m)$ for all $m$.\\

By twisting the action of $\U_q(\hat{sl}_4)'$ on $V_0(Y_{1,1}Y_{1,-1})$ and $V_0(Y_{1,1}Y_{1,-1-4s})$ by $\theta^{(j)}$ and $t_b$ for some $b \in \CC^{\ast}$, we obtain for all $j \in I, k \in \ZZ$ and $s \in \NN^{\ast}$
\begin{itemize}
\item the $\U_q^{v,j}(sl_{4}^{tor})$-modules $V_0(Y_{1,1+j+4k}Y_{1,-1+j+4k})^{(j)}$ which we call modules of type KR below,
\item the $\U_q^{v,j}(sl_{4}^{tor})$-modules $V_0(Y_{1,1+j+4k}Y_{1,-1+j+4k-4s})^{(j)}$ which we call modules of type $s$-TP below. The modules of type $s$-TP for various $s \in \NN^{\ast}$ are called modules of type TP.
\end{itemize}
From the construction done above, we get bases $(v_m)$ of these modules indexed by the monomial crystals $\mathcal{M}_{I_j}(\phi^{j+4k}(Y_{1,1}Y_{1,-1}Y_{0,2}^{-1}Y_{0,0}^{-1}))$ (resp. $\mathcal{M}_{j,k,s}$) with analogous properties as the previous ones. In particular, the action on a vector $v_m$ is completely determined by the action of the horizontal quantum affine subalgebra on it and by its $\ell$-weight $m$.\\

Let us begin the construction of the $\U_q(sl_{4}^{tor})$-module $V(e^{2\varpi_1}Y_{1,1}Y_{1,-1}Y_{0,2}^{-1}Y_{0,0}^{-1})$. Denote by $ \mathcal{E} $ the set of monomials occurring in $\overline{\mathcal{M}}(e^{2\varpi_1}Y_{1,1}Y_{1,-1}Y_{0,2}^{-1}Y_{0,0}^{-1})$ and for all $j \in I$, $k \in \ZZ$ and $s \in \NN^{\ast}$, $ \mathcal{E}_{j,k,0} $ (resp. $ \mathcal{E}_{j,k,s} $) the set of monomials corresponding to $\mathcal{M}_{j,k,0}^{1}$ (resp. $\mathcal{M}_{j,k,s}$). We have for all $ 0 \leq j \leq 3 $,
$$ \mathcal{E} = \bigsqcup_{k \in \ZZ} \mathcal{E}_{j,k,0} \sqcup \bigsqcup_{k \in \ZZ, s \in \NN^{\ast}} \mathcal{E}_{j,k,s}.$$
Let
$$ V(e^{2\varpi_1}Y_{1,1}Y_{1,-1}Y_{0,2}^{-1}Y_{0,0}^{-1}) = \bigoplus_{m \in \mathcal{E}} \CC v_{m} $$
be the vector space freely generated by $ \mathcal{E} $. For $ 0 \leq j \leq 3 $, $ k \in \ZZ $ and $s \in \NN^{\ast}$, set $ V_{k}^{(j)} = \bigoplus_{m \in \mathcal{E}_{j,k,0}} \CC v_{m} $ (resp. $ V_{k,s}^{(j)} = \bigoplus_{m \in \mathcal{E}_{j,k,s}} \CC v_{m} $). Then for all $ 0 \leq j \leq 3 $,
$$V(e^{2\varpi_1}Y_{1,1}Y_{1,-1}Y_{0,2}^{-1}Y_{0,0}^{-1}) = \bigoplus_{k \in \ZZ} V_{k}^{(j)} \oplus \bigoplus_{ k \in \ZZ, s \in \NN^{\ast}} V_{k,s}^{(j)}.$$

For all $j \in I$, we endow $ V(e^{2\varpi_1}Y_{1,1}Y_{1,-1}Y_{0,2}^{-1}Y_{0,0}^{-1})$ with a structure of $ \U_{q}^{v,j}(sl_{4}^{tor}) $-module as follows: for $ k \in \ZZ $ and $s \in \NN^{\ast}$, the vector space $ V_{k}^{(j)}$ (resp. $ V_{k,s}^{(j)}$) is isomorphic to $V_0(Y_{1,1+j+4k}Y_{1,-1+j+4k})^{(j)}$ (resp. $V_0(Y_{1,1+j+4k}Y_{1,-1+j+4k-4s})^{(j)}$)  by identifying the corresponding bases. So $ V_{k}^{(j)}$ (resp. $ V_{k,s}^{(j)}$) is endowed with a structure of $ \U_{q}^{v,j}(sl_{4}^{tor}) $-module, and $V(e^{2\varpi_1}Y_{1,1}Y_{1,-1}Y_{0,2}^{-1}Y_{0,0}^{-1})$ also by direct sum. We denote it by $V(e^{2\varpi_1}Y_{1,1}Y_{1,-1}Y_{0,2}^{-1}Y_{0,0}^{-1})^{(j)}$.

\begin{prop}
There exists a $ \U_{q}(sl_{4}^{tor}) $-module structure on $V(e^{2\varpi_1}Y_{1,1}Y_{1,-1}Y_{0,2}^{-1}Y_{0,0}^{-1})$ such that for all $j\in I$ the induced $ \U_{q}^{v,j}(sl_{n+1}^{tor}) $-module is isomorphic to \linebreak $V(e^{2\varpi_1}Y_{1,1}Y_{1,-1}Y_{0,2}^{-1}Y_{0,0}^{-1})^{(j)}$. Furthermore the $q$--character of $V(e^{2\varpi_1}Y_{1,1}Y_{1,-1}Y_{0,2}^{-1}Y_{0,0}^{-1})$ is
$$\chi_q(V(e^{2\varpi_1}Y_{1,1}Y_{1,-1}Y_{0,2}^{-1}Y_{0,0}^{-1})) = \sum_{m \in \mathcal{E}} m,$$
where $\mathcal{E}$ is the set of monomials occurring in $\overline{\mathcal{M}}(e^{2\varpi_1}Y_{1,1}Y_{1,-1}Y_{0,2}^{-1}Y_{0,0}^{-1})$.
\end{prop}

\begin{proof}
The process is the same as in Theorem \ref{thmmod}: to define an action of $\U_{q}(sl_{4}^{tor})$, we determine the action of the subalgebras $\hat{\U}_i$ for all $i \in I$. For that, let $j \in I$ be such that $j \neq i$. Then the action of $\hat{\U}_i$ on $V(e^{2\varpi_1}Y_{1,1}Y_{1,-1}Y_{0,2}^{-1}Y_{0,0}^{-1})$ is the restriction of the action of $ \U_{q}^{v,j}(sl_4^{tor}) $ on $V(e^{2\varpi_1}Y_{1,1}Y_{1,-1}Y_{0,2}^{-1}Y_{0,0}^{-1})^{(j)}$. We check that this is independent of the choice of $j \neq i$.

Let us show that this action endows $V(e^{2\varpi_1}Y_{1,1}Y_{1,-1}Y_{0,2}^{-1}Y_{0,0}^{-1})$ with a structure of $ \U_{q}(sl_{4}^{tor}) $-module. For that, we have to distinguish two types of monomials:
\begin{itemize}
\item the $m$ such that there is no $s, s' \in \NN$ with $s \neq s'$ and $m \in \mathcal{E}_{j,k,s} \cap \mathcal{E}_{j',k',s'}$ for some $0 \leq j, j' \leq 3$ and $k, k' \in \ZZ$. For such a monomial, the defined action on $v_m$ comes from the same type of modules, i.e. only of modules of type KR or only on modules of type $s$-TP for one $s \in \NN^{\ast}$,
\item the $m$ such that there is $s, s' \in \NN$ with $s \neq s'$ and $m \in \mathcal{E}_{j,k,s} \cap \mathcal{E}_{j',k',s'}$ for some $0 \leq j, j' \leq 3$ and $k, k' \in \ZZ$. For such a monomial, the defined action on $v_m$ comes from two different types of modules, i.e. of modules of type KR and of type TP or of modules of type $s$-TP and of type $s'$-TP with $s \neq s'$.
\end{itemize}
For the first ones, the same process as in Theorem \ref{thmmod} (using promotion operator) implies that the defining relations of $ \U_{q}(sl_{4}^{tor}) $ hold on it.
For the other ones, this is more complicated. Such a monomial is of the form $m = \phi^{j+4k}(Y_{1,-1-4s}Y_{3,5}^{-1}Y_{0,4}Y_{0,-4s}^{-1})$ with $ 0 \leq j \leq 3 $, $ k \in \ZZ $, $s \in \NN^{\ast}$. Promotion operator implies some relations on $v_m$ but not all and we check directly that they are satisfied. We do not detail the calculations here.
\end{proof}

\subsection{Study of the $\U_q(sl_{4}^{tor})$-module $V(e^{2\varpi_1}Y_{1,1}Y_{1,-1}Y_{0,2}^{-1}Y_{0,0}^{-1})$}

\begin{prop}
The $\U_q(sl_{4}^{tor})$-module $V(e^{2\varpi_1}Y_{1,1}Y_{1,-1}Y_{0,2}^{-1}Y_{0,0}^{-1})$ is integrable. Moreover, it satisfies property $(iv)$ of Remark \ref{remintmod}.
\end{prop}

\begin{proof}
For all $j \in I$, $V(e^{2\varpi_1}Y_{1,1}Y_{1,-1}Y_{0,2}^{-1}Y_{0,0}^{-1})$ is completely reducible as a $\U_q^{v,j}(sl_{n+1}^{tor})$-module and we have
\begin{equation}\label{decmodex}
V(e^{2\varpi_1}Y_{1,1}Y_{1,-1}Y_{0,2}^{-1}Y_{0,0}^{-1})^{(j)} = \bigoplus_{s \in \NN, k \in \ZZ} V_0(Y_{1,1+j+4k} Y_{1, -1 + j + 4k-4s})^{(j)}.
\end{equation}
The representations occurring in the direct sum at the right hand side are integrable. Hence $V(e^{2\varpi_1}Y_{1,1}Y_{1,-1}Y_{0,2}^{-1}Y_{0,0}^{-1})$ is an integrable $\U_q(sl_{n+1}^{tor})$-module. Furthermore the modules of type KR (resp. of type TP) are all isomorphic as $\U_q(sl_{4})$-modules and satisfy property $(iv)$ of Remark \ref{remintmod}. Then, it holds for $V(e^{2\varpi_1}Y_{1,1}Y_{1,-1}Y_{0,2}^{-1}Y_{0,0}^{-1})$, i.e. $$V(e^{2\varpi_1}Y_{1,1}Y_{1,-1}Y_{0,2}^{-1}Y_{0,0}^{-1})_{\nu + N\alpha_i} = \{0\} \text{ for all } \nu \in P, i \in I, N >>0.$$
\end{proof}

\begin{rem}
The weight spaces of $V(e^{2\varpi_1}Y_{1,1}Y_{1,-1}Y_{0,2}^{-1}Y_{0,0}^{-1})$ are infinite-dimensional and property $(iii)$ does not hold. However its $\ell$-weight spaces are all of dimension one.
\end{rem}

The main result of this section is the following.

\begin{thm}
Set $M = e^{2\varpi_1}Y_{1,1}Y_{1,-1}Y_{0,2}^{-1}Y_{0,0}^{-1}$. The representation $V(M)$ is an extremal loop weight module generated by the vector $v_{M}$ of $\ell$-weight $M$.
\end{thm}

\begin{proof}
The third point of Definition \ref{defelm} is a consequence of (\ref{decmodex}). For the first two points, we use the following results.
\end{proof}

\begin{lem}
Let $V$ be a $\U_q(\hat{sl}_{n+1})$-module with basis $(v_m)_{m \in \mathcal{M}'}$ indexed by a subcrystal $\mathcal{M}'$ of $\mathcal{M}$. Assume that $M \in \mathcal{M}'$ is extremal of weight $\wt(M)$ and for all $i \in I$ and $m \in W \cdot M$,
$$\wt(v_m) = \wt(m), \ x_i^{\pm} \cdot v_m = 0 \text{ and } (x_i^{\mp})^{(\pm \wt(m)(h_i))} \cdot v_m = v_{S_i(m)} \text{ if } \pm \wt(m)(h_i) \geq 0.$$
Then $v_M$ is an extremal vector of weight $\wt(M)$.
\end{lem}

\begin{proof}
The proof is analogue to the one of Lemma \ref{lirepcrylem}.
\end{proof}

\begin{cor}
Set $M_s = e^{2\varpi_1 + s\delta}Y_{1,1}Y_{1,-1-4s}Y_{0,2}^{-1}Y_{0,-4s}^{-1}$ $(s \in \NN)$. Then $v_{M_s}$ is an extremal vector of $ V(e^{2\varpi_1}Y_{1,1}Y_{1,-1}Y_{0,2}^{-1}Y_{0,0}^{-1})$ of weight $2 \varpi_1+s\delta$ for the horizontal quantum affine subalgebra $\U_q^h(sl_{4}^{tor})$.
\end{cor}

\begin{proof}
By construction of the basis $(v_m)$ of $V(e^{2\varpi_1}Y_{1,1}Y_{1,-1}Y_{0,2}^{-1}Y_{0,0}^{-1})$, we have $\wt(v_m) = \wt(m)$ for all $m$. Furthermore a monomial in $W \cdot M_s$ is of the form $\phi^{j+4k}(M_s)$ with $j \in I_0$ and $k \in \ZZ$ and we have $\wt(\phi^{j+4k}(M_s)) = 2\Lambda_{j+1} - 2\Lambda_j - 2k \delta$,
\begin{eqnarray*}
(x_{j,0}^{+})^{(2)} \cdot v_{\phi^{j+4k}(M_s)} &=& v_{\phi^{j-1+4k}(M_s)} = v_{S_j(\phi^{j+4k}(M_s))}, \\
(x_{j+1,0}^-)^{(2)} \cdot v_{\phi^{j+4k}(M_s)} &=& v_{\phi^{j+1+4k}(M_s)} = v_{S_{j+1}(\phi^{j+4k}(M_s))}, \\
x_{i}^{\pm} \cdot v_{\phi^{j+4k}(M_s)} &=& 0 \text{ in the other cases.}
\end{eqnarray*}
Hence the hypotheses of the above lemma are satisfied and $v_{M_s}$ is extremal of weight $2 \varpi_1+s\delta$ for the horizontal quantum affine subalgebra $\U_q^h(sl_{4}^{tor})$.
\end{proof}

\begin{prop}
The representation $V(e^{2\varpi_1}Y_{1,1}Y_{1,-1}Y_{0,2}^{-1}Y_{0,0}^{-1})$ is cyclic as a $\U_q^h(sl_{4}^{tor})$-module, generated by the vector $v=v_{e^{2\varpi_1}Y_{1,1}Y_{1,-1}Y_{0,2}^{-1}Y_{0,0}^{-1}}$.
\end{prop}

\begin{proof}
Consider the sub-$\U_q^h(sl_4^{tor})$-module $V$ generated by $v$. By construction of the basis $(v_m)$ of $V(e^{2\varpi_1}Y_{1,1}Y_{1,-1}Y_{0,2}^{-1}Y_{0,0}^{-1})$, $v_m \in V$ for all $m \in \mathcal{M}(e^{2\varpi_1}Y_{1,1}Y_{1,-1}Y_{0,2}^{-1}Y_{0,0}^{-1})$.
We proceed by a recursive argument: assume that for one $s \in \NN$, we have $v_m \in V$ for all $m \in \mathcal{M}(e^{2\varpi_1-t\delta}Y_{1,1}Y_{1,-1-4t}Y_{0,2}^{-1}Y_{0,-4t}^{-1})$ with $0 \leq t \leq s$. In particular $v_{Y_{1,-1-4s}Y_{3,5}^{-1}Y_{0,4}Y_{0,-4s}^{-1}}$ is in $V$ and by Example \ref{exprod},
$$x_{0,0}^- \cdot v_{Y_{1,-1-4s}Y_{3,5}^{-1}Y_{0,4}Y_{0,-4s}^{-1}} = v_{e^{2\varpi_1-(s+1)\delta}Y_{1,5}Y_{1,-1-4s}Y_{0,6}^{-1}Y_{0,-4s}^{-1}} \in V.$$
In the same way $v_{\phi^k(Y_{1,-1-4s}Y_{3,5}^{-1}Y_{0,4}Y_{0,-4s}^{-1})}$ and $v_{\phi^k(e^{2\varpi_1-(s+1)\delta}Y_{1,5}Y_{1,-1-4s}Y_{0,6}^{-1}Y_{0,-4s}^{-1})}$ are in $V$ for any $k \in \ZZ$. But all the $v_m$ with $m \in \mathcal{M}(e^{2\varpi_1-(s+1)\delta}Y_{1,5}Y_{1,-1-4s}Y_{0,6}^{-1}Y_{0,-4s}^{-1})$ can be obtained from these vectors by action of $\U_q^h(sl_4^{tor})$: this is straightforward from Example \ref{exprod} and the construction of the basis $(v_m)$.
\end{proof}

\begin{prop}
The $\U_q(sl_{4}^{tor})$-module $V(e^{2\varpi_1}Y_{1,1}Y_{1,-1}Y_{0,2}^{-1}Y_{0,0}^{-1})$ is irreducible.
\end{prop}

\begin{proof}
Let $V$ be a non trivial sub-$\U_q(sl_{4}^{tor})$-module of $V(e^{2\varpi_1}Y_{1,1}Y_{1,-1}Y_{0,2}^{-1}Y_{0,0}^{-1})$.
As the $\ell$-weight spaces are of dimension one, there exists $s \in \NN$ and a monomial \linebreak  $m \in \mathcal{M}(e^{2\varpi_1 + s\delta}Y_{1,1}Y_{1,-1-4s}Y_{0,2}^{-1}Y_{0,-4s}^{-1})$ such that $v_m \in V$.

If $s=0$, we have already shown in the above proof that $V(e^{2\varpi_1}Y_{1,1}Y_{1,-1}Y_{0,2}^{-1}Y_{0,0}^{-1})$ is cyclic generated by $v_m$ and $V = V(e^{2\varpi_1}Y_{1,1}Y_{1,-1}Y_{0,2}^{-1}Y_{0,0}^{-1})$. Assume that $s \in \NN^{\ast}$. By Example \ref{exprod} and the construction of $(v_m)$, there exists $x \in \U_q^{h}(sl_4^{tor})$ such that
$$x \cdot v_m = v_{e^{2\varpi_1 + s\delta}Y_{1,1}Y_{1,-1-4s}Y_{0,2}^{-1}Y_{0,-4s}^{-1}}.$$
Furthermore $\hat{\U}_{1} \cdot v_{e^{2\varpi_1 + s\delta}Y_{1,1}Y_{1,-1-4s}Y_{0,2}^{-1}Y_{0,-4s}^{-1}}$ is the simple $\ell$-highest weight $\hat{\U}_{1}$-module of $\ell$-highest weight $Y_{1,1}Y_{1,-1-4s}$ and there exists $y \in \U_q(sl_4^{tor})$ such that $$y \cdot v_{m}=v_{Y_{1,1}Y_{1,1-4s}^{-1}Y_{2,-4s}Y_{0,2}^{-1}}$$ with
$Y_{1,1}Y_{1,1-4s}^{-1}Y_{2,-4s}Y_{0,2}^{-1} \in \mathcal{M}\left(e^{2\varpi_1 + (s-1)\delta}Y_{1,1}Y_{1,-1-4(s-1)}Y_{0,2}^{-1}Y_{0,-4(s-1)}^{-1}\right)$.
Repeating this argument, one shows that the vector $v_{e^{2\varpi_1}Y_{1,1}Y_{1,-1}Y_{0,2}^{-1}Y_{0,0}^{-1}}$ is in $V$. By the above proposition we get $V = V(e^{2\varpi_1}Y_{1,1}Y_{1,-1}Y_{0,2}^{-1}Y_{0,0}^{-1})$.
\end{proof}

\begin{prop}
The $\U_q(\hat{sl}_{4})$-module $\mathrm{Res}(V(e^{2\varpi_1}Y_{1,1}Y_{1,-1}Y_{0,2}^{-1}Y_{0,0}^{-1}))$ has a crystal basis isomorphic to $\mathcal{B}(2 \varpi_1)$.
\end{prop}

\begin{proof} Set $K = \CC(q)$ with $q$ an indeterminate and let $A$ be the subring of $K$ consisting of rational functions in $K$ without pole at $q=0$.
We normalize the basis $(v_m)$ of the $\CC(q)$-vector space $V(e^{2\varpi_1}Y_{1,1}Y_{1,-1}Y_{0,2}^{-1}Y_{0,0}^{-1})$ as follows. For all $m \in \overline{\mathcal{M}}(e^{2\varpi_1}Y_{1,1}Y_{1,-1}Y_{0,2}^{-1}Y_{0,0}^{-1})$, let $w_m$ be the vector defined by $w_m = \frac{1}{q} v_m$ if there exists $k \in \ZZ, s \in \NN^{\ast}$ such that $m = \phi^{k}(e^{2\varpi_1 + s\delta}Y_{1,1}Y_{1,-1-4s}Y_{0,2}^{-1}Y_{0,-4s}^{-1})$ and $w_m = v_m$ otherwise. Set $\mathcal{B} = (w_m)_m$ and $\mathcal{L} = \bigoplus_{m} A w_m$.
We check directly that $(\mathcal{L}, \mathcal{B})$ is a crystal basis of the $\U_q(\hat{sl}_{4})$-module $\mathrm{Res}(V(e^{2\varpi_1}Y_{1,1}Y_{1,-1}Y_{0,2}^{-1}Y_{0,0}^{-1}))$, isomorphic to $\mathcal{B}(2 \varpi_1)$. We do not detail the calculations.
\end{proof}

\begin{rem}
All these results suggest that $\mathrm{Res}(V(e^{2\varpi_1}Y_{1,1}Y_{1,-1}Y_{0,2}^{-1}Y_{0,0}^{-1}))$ is isomorphic to the extremal weight $\U_q(\hat{sl}_4)$-module $V(2 \varpi_1)$. One expects to prove such a result for all the extremal loop weight modules constructed by the conjectural process given above.
\end{rem}

\subsection{Finite-dimensional representations at roots of unity}

Set $L \geq 1$ and let $\epsilon$ be a primitive $(4L)$-root of unity.

Denote by $\mathcal{E}_s$ the set of monomials occurring in $\mathcal{M}(e^{2\varpi_1 + s\delta}Y_{1,1}Y_{1,-1-4s}Y_{0,2}^{-1}Y_{0,-4s}^{-1})$ for all $s \in \NN$. Consider $\mathcal{E}'$ the subset of $\mathcal{E}$ defined by
\begin{eqnarray*}
\mathcal{E}' = \bigsqcup_{0 \leq s \leq L-1} \mathcal{E}_{s}.
\end{eqnarray*}

Let $\mathcal{E}_\epsilon$ and $\mathcal{E}'_\epsilon$ be the images of the sets $\mathcal{E}$ and $\mathcal{E}'$ respectively by the map $\Gamma_{(4L)}$: $\mathcal{E}'_\epsilon$ is a finite monomial set of cardinality $16L^{2}$.

\begin{thm}
Assume that $\epsilon$ is a primitive $4L$-root of unity. There exists an irreducible $\U_\epsilon(sl_{4}^{tor})'$-module $V_\epsilon$ of dimension $16 L^{2}$ such that
$$\chi_\epsilon(V_\epsilon) = \sum_{m \in \mathcal{E}'_\epsilon} m.$$
\end{thm}

\begin{proof}
The main difficulty is to specialize $q$ at $\epsilon$ in the $\U_q^{v,j}(sl_4^{tor})$-modules of type TP. In fact, these modules can be undefined or reducible after specialization. For better understand these phenomena, let us study the specialized $\U_\epsilon(\hat{sl}_2)'$-module $V_0(Y_{1,a}Y_{1,b})_\epsilon$ with $a, b \in \ZZ$. This representation is well defined if $a \notin b + 4L \ZZ$. Assume that in the following and study $V_0(Y_{1,a}Y_{1,b})_\epsilon$. If $a \notin b \pm 2+4L\ZZ $, this representation is irreducible. If $a \in b + 2 + 4L\ZZ$, it is not irreducible: in fact
$$\U_\epsilon(\hat{sl}_2)' \cdot v_{Y_{1,a}Y_{1,b}} = \CC v_{Y_{1,a}Y_{1,b}} \oplus \CC v_{Y_{1,a+2}^{-1}Y_{1,b}} \oplus \CC v_{Y_{1,a+2}^{-1}Y_{1,b+2}^{-1}}$$
is an irreducible submodule of $V_0(Y_{1,a}Y_{1,b})_\epsilon$.

By our study of the $\U_\epsilon(\hat{sl}_2)'$-module $V_0(Y_{1,a}Y_{1,b})_\epsilon$, one can specialize $q$ at $\epsilon$ in the defining relations of the action on the basis $(v_m)$ of $V(e^{2\varpi_1}Y_{1,1}Y_{1,-1}Y_{0,2}^{-1}Y_{0,0}^{-1})$. Moreover one checks that
$$\U_\epsilon(sl_{4}^{tor})' \cdot v_{Y_{1,1}Y_{1,-1-4L}Y_{0,2}^{-1}Y_{0,-4L}^{-1}} = \bigoplus_{m \in \mathcal{E} - \mathcal{E}'} \CC v_m$$
is a sub-$\U_\epsilon(sl_{4}^{tor})'$-module of $V(e^{2\varpi_1}Y_{1,1}Y_{1,-1}Y_{0,2}^{-1}Y_{0,0}^{-1})_\epsilon$. By taking the quotient, we obtain a $\U_\epsilon(sl_{4}^{tor})'$-module
$$V_\epsilon = \bigoplus_{m \in \mathcal{E}'} \CC v_m$$
which is irreducible: this is straightforward with the explicit formulas of the action.
\end{proof}

\section{Further possible developments and applications}

In this last section, we give other promising directions to study the extremal loop weight modules for quantum toroidal algebras of general types. Moreover we give some possible applications of the results obtained in this article. This will be done in further papers.

In our construction of level 0 extremal loop weight modules in type $A$, monomial realizations of crystals and promotion operators on the finite crystals have a crucial role. Let us give some results which suggest that a similar construction is possible in other types. In \cite{hernandez_level_2006}, an explicit description of monomial realizations of level 0 extremal fundamental weight crystals of quantum affine algebras is given for all the non exceptional types. The automorphisms $z_\ell$ are determined in these cases. Furthermore in other types, there exists also symmetry properties for crystals arising from automorphisms of the associated Dynkin diagram (analogue of promotion operators in type $A$). Using that, a combinatorial process allows to obtain Kirillov-Reshetikhin crystals from crystals of finite type (see \cite{fourier_kirillov-reshetikhin_2009, kang_perfect_1992, okado_existence_2008}). These symmetry properties will be useful for a similar construction of extremal loop weight modules in other types.

As we have viewed, the extremal fundamental loop weight modules $V(e^{\varpi_\ell}Y_{\ell,0}Y_{0, d_\ell}^{-1})$ ($n=2r+1$ and $\ell = 1, r+1$ or $n$) are completely reducible as $\U_q^{v,0}(sl_{n+1}^{tor})$-modules: they are direct sums of fundamental modules of $\U_q(\hat{sl}_{n+1})$. Similar vector spaces are considered in \cite{chari_quantum_2003} for the quantum affine algebra $\U_q(\hat{\Glie})$ associated to a simple Lie algebra $\Glie$ over $\CC$. In fact for a finite-dimensional representation $V$ of $\U_q(\hat{\Glie})'$, the vector space $V \otimes_{\CC} \CC[z,z^{-1}]$ is endowed with a structure of $\U_q(\hat{\Glie})$-module by using the grading of this algebra. So the action is very different to the one defined in this article and we do not have a way to extend this action for the quantum toroidal algebra $\U_q(\Glie^{tor})$. But it would be interesting to study an analogous construction for the quantum toroidal algebra $\U_q(\Glie^{tor})$. We can expect to construct other examples of extremal loop weight modules by this process.

Let us explain another approach to construct extremal loop weight modules which could be fruitful. Let $\Glie$ be a Kac-Moody algebra. For an integral weight $\lambda$, one defines $\lambda_{+}= \sum_{\lambda(h_{i})\geq 0} \lambda(h_{i}) \Lambda_{i} $ and $\lambda_{-} = \lambda_{+} - \lambda$. To study the extremal weight module $V(\lambda)$, Kashiwara \cite{kashiwara_crystal_1994} considers the tensor product $V'(\lambda) = V(\lambda_{+}) \otimes V(\lambda_{-})$ of the simple highest weight module $V(\lambda_+)$ and the simple lowest weight module $V(\lambda_-)$. By analogy, it would be interesting to define an action of the quantum affinization $\U_q(\hat{\Glie})$ on the tensor product of simple $\ell$-highest weight modules and simple $\ell$-lowest weight modules, in the spirit of \cite{hernandez_representations_2005, hernandez_drinfeld_2007} and \cite{feigin_quantum_2011, feigin_quantum_2011-1, feigin_quantum_2012, feigin_representations_2013}. This will be studied in a further paper.

Another possible direction is to study the finite-dimensional representations of Double Affine Hecke Algebras (or Cherednick algebras) at roots of unity obtained from the new finite-dimensional representations of $\U_\epsilon(sl_{n+1}^{tor})$ defined above, via Schur-Weyl duality \cite{varagnolo_schur_1996}.

In this article, we have defined promotion operators for the level 0 extremal fundamental weight crystals $\mathcal{B}(\varpi_\ell)$ in type $A_n$ ($n \geq 2$ odd, $1 \leq \ell \leq n$). It will be interesting to discuss the existence of promotion operators for other level 0 extremal weight crystals and the uniqueness of them in the spirit of \cite{shimozono_affine_2002}.

\section*{Appendix}


In this part, we describe the monomial crystal
$$\overline{\mathcal{M}}(e^{2\varpi_1}Y_{1,1}Y_{1,-1}Y_{0,2}^{-1}Y_{0,0}^{-1}) = \bigoplus_{s \in \NN} \mathcal{M}(e^{2\varpi_1 + s\delta}Y_{1,1}Y_{1,-1-4s}Y_{0,2}^{-1}Y_{0,-4s}^{-1}).$$
More precisely, we represent the two connected components $\mathcal{M}(e^{2\varpi_1}Y_{1,1}Y_{1,-1}Y_{0,2}^{-1}Y_{0,0}^{-1})$ and $\mathcal{M}(e^{2\varpi_1+\delta}Y_{1,1}Y_{1,-5}Y_{0,2}^{-1}Y_{0,-4}^{-1})$ of $\overline{\mathcal{M}}(e^{2\varpi_1}Y_{1,1}Y_{1,-1}Y_{0,2}^{-1}Y_{0,0}^{-1}) $. Recall that all its connected components are isomorphic to each other modulo shift of weight by $\delta$. Furthermore the map $\tau_{4, -2\delta}$ is an automorphism of these crystals and we only give a part of them. The full crystals are obtained by applying the automorphism $\tau_{4, -2\delta}$. The sub-$I_0$-crystals
$$\textcolor{blue}{\mathcal{M}_{0,0,0}^{1}} = \mathcal{M}_{I_0}(Y_{1,1}Y_{1,-1}Y_{0,2}^{-1}Y_{0,0}^{-1})$$
and
$$\textcolor{red}{\mathcal{M}_{0,0,1}} = \mathcal{M}_{I_0}(Y_{1,1}Y_{1,-5}Y_{0,2}^{-1}Y_{0,-4}^{-1}) \oplus \mathcal{M}_{I_0}(Y_{1,1}^{-1}Y_{1,5}Y_{2,0}Y_{0,6}^{-1})$$
are explicitly given.

Note that the $\theta$-twisted automorphism $\phi$ of $\overline{\mathcal{M}}(e^{2\varpi_1}Y_{1,1}Y_{1,-1}Y_{0,2}^{-1}Y_{0,0}^{-1})$ can be viewed as a descent of one diagonal in these crystals.

\begin{eqnarray*}
\xymatrix{
 & \ar@{-->}[ld]_0 & & \\
\textcolor{blue}{e^{2\varpi_1}Y_{1,1}Y_{1,-1}Y_{0,2}^{-1}Y_{0,0}^{-1}} \ar[rd]^1  & & \ar@{-->}[ld]_0 & \\
  & \textcolor{blue}{Y_{1,3}^{-1}Y_{1,-1}Y_{2,2}Y_{0,0}^{-1}} \ar[ld]_1 \ar[rd]^2 & & \ar@{-->}[ld]_0 \\
 \textcolor{blue}{Y_{1,3}^{-1}Y_{1,1}^{-1}Y_{2,2}Y_{2,0}} \ar[rd]^2 & & \textcolor{blue}{Y_{1,-1}Y_{2,4}^{-1}Y_{3,3}Y_{0,0}^{-1}} \ar[ld]_1 \ar[rd]^3 & \\
 & \textcolor{blue}{Y_{1,1}^{-1}Y_{2,4}^{-1}Y_{2,0}Y_{3,3}}  \ar[ld]_2 \ar[rd]^3 & & \textcolor{blue}{Y_{1,-1}Y_{3,5}^{-1}Y_{0,4}Y_{0,0}^{-1}} \ar[ld]_1 \\
 \textcolor{blue}{Y_{2,4}^{-1}Y_{2,2}^{-1}Y_{3,3}Y_{3,1}} \ar[rd]^3 & & \textcolor{blue}{Y_{1,1}^{-1}Y_{2,0}Y_{3,5}^{-1}Y_{0,4}}  \ar[ld]_2 \ar[rd]^0 & \\
 & \textcolor{blue}{Y_{2,2}^{-1}Y_{3,5}^{-1}Y_{3,1}Y_{0,4}} \ar[ld]_3 \ar[rd]^0 & & \textcolor{red}{Y_{1,5}Y_{1,1}^{-1}Y_{2,0}Y_{0,6}^{-1}} \ar[ld]_2 \\
\textcolor{blue}{Y_{3,5}^{-1}Y_{3,3}^{-1}Y_{0,4}Y_{0,2}} \ar[rd]^0 & & \textcolor{red}{Y_{1,5}Y_{2,2}^{-1}Y_{3,1}Y_{0,6}^{-1}} \ar[ld]_3 \ar[rd]^1 & \\
 & \textcolor{red}{Y_{1,5}Y_{3,3}^{-1}Y_{0,6}^{-1}Y_{0,2}} \ar@{-->}[ld]_0 \ar[rd]^1 & & \textcolor{red}{Y_{1,7}^{-1}Y_{2,6}Y_{2,2}^{-1}Y_{3,1}} \ar[ld]_3 \\
Y_{1,5}Y_{1,3}Y_{0,6}^{-1}Y_{0,4}^{-1} & & \textcolor{red}{Y_{1,7}^{-1}Y_{2,6}Y_{3,3}^{-1}Y_{0,2}} \ar[rd]^2 \ar@{-->}[ld]_0 &  \\
 & \ \ \ \  & & \textcolor{red}{Y_{2,8}^{-1}Y_{3,7}Y_{3,3}^{-1}Y_{0,2}} \ar@{-->}[ld]_0  \\
  & & & 
 }
\end{eqnarray*}
\begin{center}
The $\U_q(\hat{sl}_4)$-crystal $\mathcal{M}(e^{2\varpi_1}Y_{1,1}Y_{1,-1}Y_{0,2}^{-1}Y_{0,0}^{-1})$.
\end{center}

\begin{eqnarray*}
\xymatrix{
 & \ar@{-->}[ld]_0 & & \\
\textcolor{red}{e^{2\varpi_1+\delta}Y_{1,1}Y_{1,-5}Y_{0,2}^{-1}Y_{0,-4}^{-1}} \ar[rd]^1 & & \ar@{-->}[ld]_0 & \\
  & \textcolor{red}{Y_{1,3}^{-1}Y_{1,-5}Y_{2,2}Y_{0,-4}^{-1}} \ar[ld]_1 \ar[rd]^2 & & \ar@{-->}[ld]_0 \\
 \textcolor{red}{Y_{1,3}^{-1}Y_{1,-3}^{-1}Y_{2,2}Y_{2,-4}} \ar[rd]^2 & & \textcolor{red}{Y_{1,-5}Y_{2,4}^{-1}Y_{3,3}Y_{0,-4}^{-1}} \ar[ld]_1 \ar[rd]^3 & \\
 &\textcolor{red}{ Y_{1,-3}^{-1}Y_{2,4}^{-1}Y_{2,-4}Y_{3,3}}  \ar[ld]_2 \ar[rd]^3 & & \textcolor{red}{Y_{1,-5}Y_{3,5}^{-1}Y_{0,4}Y_{0,-4}^{-1}} \ar[ld]_1 \\
 \textcolor{red}{Y_{2,4}^{-1}Y_{2,-2}^{-1}Y_{3,3}Y_{3,-3}} \ar[rd]^3 & & \textcolor{red}{Y_{1,-3}^{-1}Y_{2,-4}Y_{3,5}^{-1}Y_{0,4}}  \ar[ld]_2 \ar[rd]^0 & \\
 & \textcolor{red}{Y_{2,-2}^{-1}Y_{3,5}^{-1}Y_{3,-3}Y_{0,4}} \ar[ld]_3 \ar[rd]^0 & & Y_{1,5}Y_{1,-3}^{-1}Y_{2,-4}Y_{0,6}^{-1} \ar[ld]_2 \\
\textcolor{red}{Y_{3,5}^{-1}Y_{3,-1}^{-1}Y_{0,4}Y_{0,-2}} \ar[rd]^0 & & Y_{1,5}Y_{2,-2}^{-1}Y_{3,-3}Y_{0,6}^{-1} \ar[ld]_3 \ar[rd]^1 & \\
 & Y_{1,5}Y_{3,-1}^{-1}Y_{0,6}^{-1}Y_{0,-2} \ar@{-->}[ld]_0 \ar[rd]^1 & & Y_{1,7}^{-1}Y_{2,6}Y_{2,-2}^{-1}Y_{3,-3} \ar[ld]_3 \\
Y_{1,5}Y_{1,-1}Y_{0,6}^{-1}Y_{0,0}^{-1} & & Y_{1,7}^{-1}Y_{2,6}Y_{3,-1}^{-1}Y_{0,-2} \ar[rd]^2 \ar@{-->}[ld]_0 &  \\
 & \ \ \ \  & & Y_{2,8}^{-1}Y_{3,7}Y_{3,-1}^{-1}Y_{0,-2} \ar@{-->}[ld]_0  \\
  & & & 
 }
\end{eqnarray*}
\begin{center}
\begin{center}
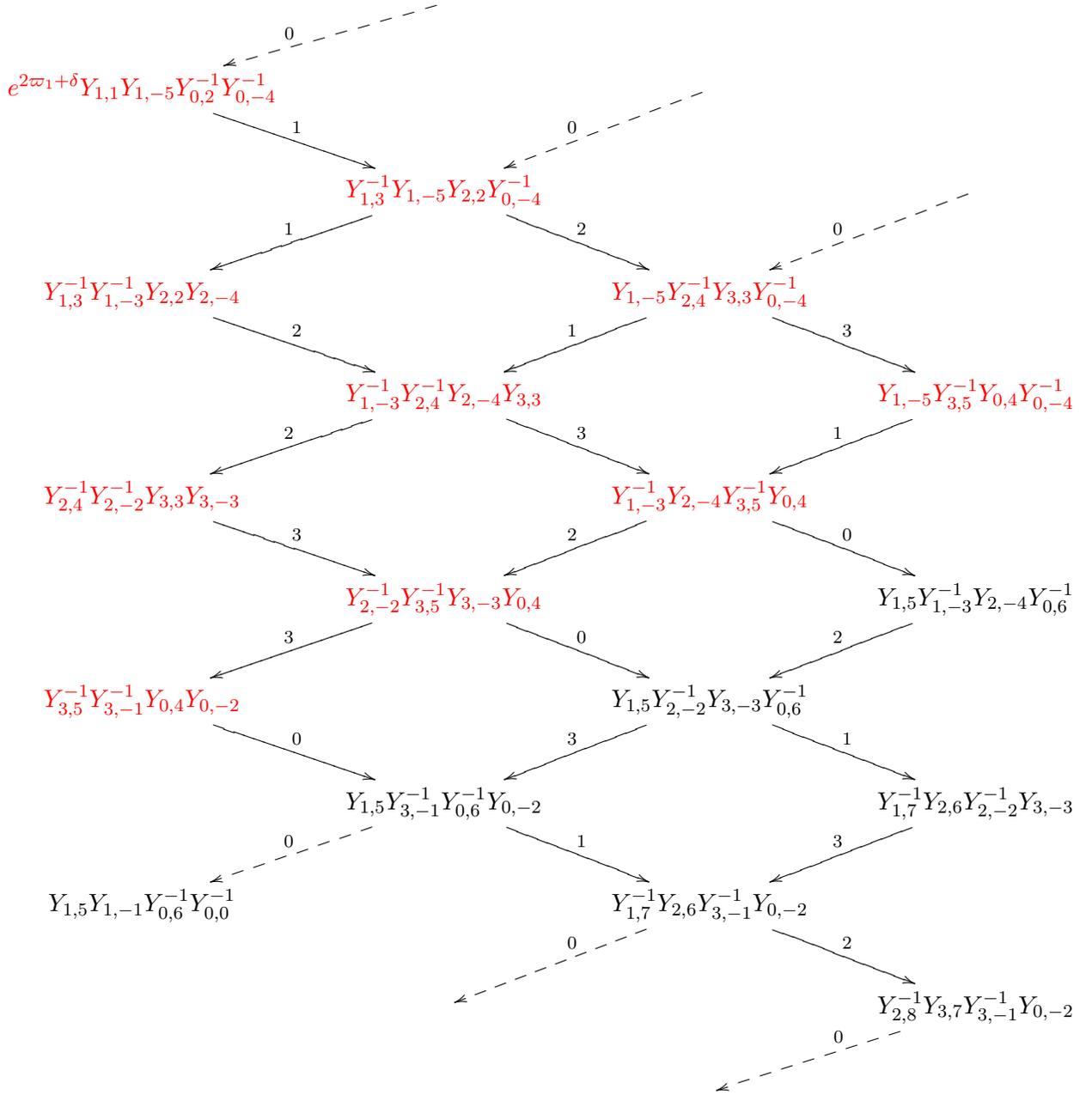
The $\U_q(\hat{sl}_4)$-crystal $\mathcal{M}(e^{2\varpi_1+\delta}Y_{1,1}Y_{1,-5}Y_{0,2}^{-1}Y_{0,-4}^{-1})$.
\end{center}
\end{center}

\newpage

\bibliographystyle{acm}

\end{document}